\documentclass[11pt,letterpaper,colorlinks=true,linkcolor=blue,citecolor=blue,urlcolor=blue]{scrartcl}
\usepackage[T1]{fontenc}
\usepackage[utf8]{inputenc}
\usepackage{lmodern}
\usepackage[english]{babel}
\usepackage[english]{isodate}
\usepackage{amsthm}
\usepackage{amsmath}
\usepackage{amssymb}
\usepackage{amsfonts}
\usepackage{mathtools}
\usepackage{mathrsfs}
\usepackage{dsfont}
\usepackage{upgreek}
\usepackage{xpatch}
\usepackage{bm}
\usepackage{bbm}
\usepackage{enumitem}
\usepackage[toc,page]{appendix}
\usepackage{footnotebackref}
\usepackage[totalwidth=480pt, totalheight=680pt]{geometry}
\usepackage{hyperref}

\theoremstyle{remark}
\newtheorem{assumptions}{Assumptions}[section]
\newtheorem{remark}[assumptions]{Remark}

\theoremstyle{plain}
\newtheorem{theorem}[assumptions]{Theorem}
\newtheorem{proposition}[assumptions]{Proposition}
\newtheorem{lemma}[assumptions]{Lemma}
\newtheorem{corollary}[assumptions]{Corollary}

\addtokomafont{disposition}{\normalfont\scshape}

\makeatletter
\xpatchcmd{\@sec@pppage}{
\bfseries}{
\normalfont\scshape\Large}{}{}
\makeatother

\setkomafont{section}{\Large}
\setkomafont{subsection}{\large}

\numberwithin{equation}{section}

\begin{document}


\title{
\LARGE A trajectorial approach to the gradient flow properties of Langevin--Smoluchowski diffusions\thanks{
We thank Luigi Ambrosio, Mathias Beiglb{\"o}ck, Max Fathi, Ivan Gentil, David Kinderlehrer, Daniel Lacker, Michel Ledoux, Jan Maas, Felix Otto, Chris Rogers, Oleg Szehr, and Lane Chun Yeung for their advice and comments during the preparation of this paper. Special thanks go to Luigi Ambrosio, Michel Ledoux and Jan Maas for their expert guidance, which helped us navigate several difficult narrows successfully. \smallskip\newline
\noindent I. Karatzas acknowledges support from the National Science Foundation (NSF) under grants NSF-DMS-14-05210 and NSF-DMS-20-04997. W. Schachermayer and B. Tschiderer acknowledge support by the Austrian Science Fund (FWF) under grant P28661. W. Schachermayer additionally appreciates support by the Vienna Science and Technology Fund (WWTF) through projects MA14-008 and MA16-021. \smallskip\newline
\noindent Much of this work was done in Fall 2018, when W. Schachermayer was visiting the Department of Mathematics at Columbia University as Minerva Foundation Fellow.
}
}


\author{
\large Ioannis Karatzas\thanks{
Department of Mathematics, Columbia University, 2990 Broadway, New York, NY 10027, USA 
\newline (email: \href{mailto: ik1@columbia.edu}{ik1@columbia.edu}); 
\newline and INTECH Investment Management, One Palmer Square, Suite 441, Princeton, NJ 08542, USA 
\newline (email: \href{mailto: ikaratzas@intechjanus.com}{ikaratzas@intechjanus.com}).
} 
\and \large Walter Schachermayer\thanks{
Faculty of Mathematics, University of Vienna, Oskar-Morgenstern-Platz 1, 1090 Vienna, Austria 
\newline (email: \href{mailto: walter.schachermayer@univie.ac.at}{\mbox{walter}.schachermayer@univie.ac.at}); 
\newline and Department of Mathematics, Columbia University, 2990 Broadway, New York, NY 10027, USA.
} 
\and \large Bertram Tschiderer\thanks{
Faculty of Mathematics, University of Vienna, Oskar-Morgenstern-Platz 1, 1090 Vienna, Austria \newline (email: \href{mailto: bertram.tschiderer@univie.ac.at}{bertram.tschiderer@univie.ac.at}).
}
}


\date{\normalsize 20th August 2020}


\maketitle


\begin{abstract} \small \noindent \textsc{Abstract.} We revisit the variational characterization of conservative diffusion as entropic gradient flow and provide for it a probabilistic interpretation based on stochastic calculus. It was shown by Jordan, Kinderlehrer, and Otto that, for diffusions of Langevin--Smoluchowski type, the Fokker--Planck probability density flow maximizes the rate of relative entropy dissipation, as measured by the distance traveled in the ambient space of probability measures with finite second moments, in terms of the quadratic Wasserstein metric. We obtain novel, stochastic-process versions of these features, valid along almost every trajectory of the diffusive motion in the backward direction of time, using a very direct perturbation analysis. By averaging our trajectorial results with respect to the underlying measure on path space, we establish the maximal rate of entropy dissipation along the Fokker--Planck flow and measure exactly the deviation from this maximum that corresponds to any given perturbation. As a bonus of our trajectorial approach we derive the HWI inequality relating relative entropy (H), Wasserstein distance (W) and relative Fisher information (I).

\bigskip

\small \noindent \href{https://mathscinet.ams.org/mathscinet/msc/msc2010.html}{\textit{MSC 2010 subject classifications:}} Primary 60H30, 60G44; secondary 82C31, 60J60, 94A17

\bigskip

\end{abstract}


\newpage


\section{Introduction} \label{soti} 


We provide a trajectorial interpretation of a seminal result by Jordan, Kinderlehrer, and Otto \cite{JKO98}, and present a proof based on stochastic analysis. The basic theme of our approach could be described epigrammatically as ``applying It\^{o} calculus to Otto calculus''. More precisely, we follow a stochastic analysis approach to the characterization of diffusions of Langevin--Smoluchowski type as entropic gradient flows in Wasserstein space, as in \cite{JKO98}. We provide stronger, trajectorial versions of these results. For consistency and readability we adopt the setting and notation of \cite{JKO98}, and even copy some paragraphs of this paper almost verbatim in the remainder of this section. 

\smallskip

Along the lines of \cite{JKO98}, we consider thus a Fokker--Planck or forward Kolmogorov \cite{Kol31} equation of the form
\begin{equation} \label{fpeqnwfp}
\partial_{t} p(t,x) = \operatorname{div}\big(\nabla \Psi(x) \, p(t,x) \big) + \tfrac{1}{2} \Delta p(t,x), \qquad (t,x) \in (0,\infty) \times \mathds{R}^{n},
\end{equation}
with initial condition
\begin{equation} \label{icnwfp}
p(0,x) = p^{0}(x), \qquad x \in \mathds{R}^{n}.
\end{equation}
Here, $p$ is a real-valued function defined for $(t,x) \in [0,\infty) \times \mathds{R}^{n}$, the function $\Psi \colon \mathds{R}^{n} \rightarrow [0,\infty)$ is smooth and plays the role of a potential, and $p^{0}$ is a probability density on $\mathds{R}^{n}$. The solution $p(t,x)$ of \hyperref[fpeqnwfp]{(\ref*{fpeqnwfp})} with initial condition \hyperref[icnwfp]{(\ref*{icnwfp})} stays non-negative and conserves its mass, which means that the spatial integral $\int_{\mathds{R}^{n}} p(t,x) \, \textnormal{d}x$ is independent of the time parameter $t \geqslant 0$ and is thus equal to $\int p^{0} \, \textnormal{d}x = 1$. Therefore, $p(t, \, \cdot \,)$ must be a probability density on $\mathds{R}^{n}$ for every fixed time $t \geqslant 0$.

\smallskip

As in \cite{JKO98} we note that the Fokker--Planck equation \hyperref[fpeqnwfp]{(\ref*{fpeqnwfp})} with initial condition \hyperref[icnwfp]{(\ref*{icnwfp})} is inherently related to the stochastic differential equation of Langevin--Smoluchowski type \cite{Fri75,Gar09,Ris96,Sch80} 
\begin{equation} \label{sdeids}
\textnormal{d}X(t) = - \nabla \Psi\big(X(t)\big) \, \textnormal{d}t + \textnormal{d}W(t), \qquad t \geqslant 0.
\end{equation}
In the equation above, $(W(t))_{t \geqslant 0}$ is an $n$-dimensional Brownian motion started at the origin, and the $\mathds{R}^{n}$-valued random variable $X(0)$ is independent of the process $(W(t))_{t \geqslant 0}$. The probability distribution of $X(0)$ has density $p^{0}$ and, unless specified otherwise, the reference measure will always be Lebesgue measure on $\mathds{R}^{n}$. Then $p(t, \, \cdot \,)$, the solution of \hyperref[fpeqnwfp]{(\ref*{fpeqnwfp})} with initial condition \hyperref[icnwfp]{(\ref*{icnwfp})}, gives at any given time $t \geqslant 0$ the probability density function of the random variable $X(t)$ from \hyperref[sdeids]{(\ref*{sdeids})}. 

\smallskip

If the potential $\Psi$ grows rapidly enough so that $\mathrm{e}^{-2 \Psi} \in L^{1}(\mathds{R}^{n})$, then the partition constant
\begin{equation} \label{czpf}
Z  = \int_{\mathds{R}^{n}} \mathrm{e}^{- 2 \Psi(x)} \, \textnormal{d}x
\end{equation}
is finite and there exists a unique stationary solution of the Fokker--Planck equation \hyperref[fpeqnwfp]{(\ref*{fpeqnwfp})}; namely, the probability density $q_{Z}$ of the \textit{Gibbs distribution} given by \cite{Gar09,JK96,Ris96}
\begin{equation} \label{usdotgdm}
q_{Z}(x) = Z^{-1} \, \mathrm{e}^{ -2 \Psi(x)}
\end{equation}
for $x \in \mathds{R}^{n}$. When it exists, the probability measure on $\mathds{R}^{n}$ with density function $q_{Z}$ is called Gibbs distribution, and is the unique invariant measure for the Markov process $(X(t))_{t \geqslant 0}$ defined by the stochastic differential equation \hyperref[sdeids]{(\ref*{sdeids})}; see, e.g., \cite[Exercise 5.6.18, p.\ 361]{KS88}.

\smallskip

In \cite{JK96} it is shown that the stationary probability density $q_{Z}$ satisfies the following variational principle: it minimizes the \textit{free energy functional}
\begin{equation} \label{fef}
\mathscr{F}(p) = \mathcal{E}(p) + \tfrac{1}{2} \, \mathcal{S}(p)
\end{equation}
over all probability densities $p$ on $\mathds{R}^{n}$. Here, the functionals
\begin{equation} 
\mathcal{E}(p) \vcentcolon = \int_{\mathds{R}^{n}} \Psi(x) \, p(x) \, \textnormal{d}x, \qquad \qquad \mathcal{S}(p) \vcentcolon = \int_{\mathds{R}^{n}} p(x) \log p(x) \, \textnormal{d}x
\end{equation}
model respectively the potential energy and the internal energy (given by the negative of the Gibbs-Boltzmann entropy functional).

\subsection{Preview}

We set up in \hyperref[snaas]{Section \ref*{snaas}} the model for the Langevin--Smoluchowski diffusion, and introduce its fundamental quantities: the current and the invariant distributions of particles, the resulting likelihood ratio process, the associated concepts of free energy, relative entropy and relative Fisher information. In \hyperref[sub.reg.ass.]{Subsection \ref*{sub.reg.ass.}} we discuss the regularity assumptions of the present paper. 

\smallskip

\hyperref[thetheoseccth]{Sections \ref*{thetheoseccth}} and \hyperref[thetheoscthtra]{\ref*{thetheoscthtra}} present the basic results. These include \hyperref[thetone]{Theorem \ref*{thetone}}, which computes in terms of the relative Fisher information the rate of relative entropy decay in the ambient Wasserstein space of probability density functions with finite second moment; and its ``perturbed'' counterpart, \hyperref[thetthre]{Theorem \ref*{thetthre}}. We compute explicitly the difference between these perturbed and unperturbed rates, and show that it is always non-negative --- in fact strictly positive, unless the perturbation and the gradient of the log-likelihood ratio function are collinear. This way, the Langevin--Smoluchowski diffusion emerges as the \textit{steepest descent} (or ``gradient flow'') of the relative entropy functional with respect to the Wasserstein metric. 

\smallskip

The essence of \hyperref[thetone]{Theorems \ref*{thetone}} and \hyperref[thetthre]{\ref*{thetthre}} is well known, and the special case $\Psi(x) = \frac{1}{2} \vert x \vert^{2}$ of Ornstein--Uhlenbeck dynamics goes back as far as the 1950's. Our novel contribution is that \hyperref[thetone]{Theorems \ref*{thetone}} and \hyperref[thetthre]{\ref*{thetthre}} are simple consequences of their stronger, trajectorial versions, \hyperref[thetsix]{Theorems \ref*{thetsix}} and \hyperref[thetthretv]{\ref*{thetthretv}}, respectively. These are the main results of this work. They provide very detailed descriptions for the semimartingale dynamics of the relative entropy process in both its ``pure'' and ``perturbed'' forms, and are most transparent when time is \textit{reversed}. \hyperref[thetone]{Theorems \ref*{thetone}} and \hyperref[thetthre]{\ref*{thetthre}} then follow from \hyperref[thetsix]{Theorems \ref*{thetsix}} and \hyperref[thetthretv]{\ref*{thetthretv}} simply by taking expectations. 

\smallskip

Several consequences and ramifications of the main results, \hyperref[thetsix]{Theorems \ref*{thetsix}} and \hyperref[thetthretv]{\ref*{thetthretv}}, are developed in \hyperref[subimpconq]{Subsections \ref*{subimpconq}} and \hyperref[ramifications]{\ref*{ramifications}}, including a derivation of the famous HWI inequality of Otto and Villani \cite{OV00, Vil03, Vil09, CE02} that relates relative entropy (H) to Wasserstein distance (W) and to relative Fisher information (I). Detailed arguments and proofs are collected in \hyperref[ssgrodwudm]{Section \ref*{ssgrodwudm}}. The limiting behavior of the Wasserstein distance along the Langevin--Smoluchowski diffusion is analyzed in \hyperref[stwt]{Section \ref*{stwt}}; here, most of the effort goes into showing that relative entropy and Wasserstein distance have exactly the same exceptional sets of zero Lebesgue measure, for their temporal rate of change. This, seemingly purely technical, point, is of paramount importance for the rigorous justification of the perturbation analysis deployed in \hyperref[thetthre]{Theorem \ref*{thetthre}}; it turns out also to be rather delicate.

\smallskip

The present paper is a condensed version of the more detailed presentation \cite{KST20} available on \href{https://arxiv.org}{arXiv} under \href{https://arxiv.org/abs/1811.08686}{https://arxiv.org/abs/1811.08686}. This extended version contains more details, and several of its appendices present background material and known results used in our approach. 


\section{The stochastic approach} \label{snaas}


In \hyperref[soti]{Section \ref*{soti}} we were mostly quoting from \cite{JKO98}. We adopt now a more probabilistic point of view, and translate our setting into the language of stochastic processes and probability measures. 

\smallskip

Let $P(0)$ be a probability measure on the Borel sets of $\mathds{R}^{n}$ with density function $p^{0} = p(0, \, \cdot \,)$. This measure induces a probability measure $\mathds{P}$ on path space $\Omega = \mathcal{C}(\mathds{R}_{+};\mathds{R}^{n})$ of $\mathds{R}^{n}$-valued continuous functions on $\mathds{R}_{+}=[0,\infty)$, under which the canonical coordinate process $(X(t,\omega))_{t \geqslant 0} = (\omega(t))_{t \geqslant 0}$ satisfies the stochastic differential equation \hyperref[sdeids]{(\ref*{sdeids})} with initial probability distribution $P(0)$. We shall denote by $P(t)$ the probability distribution of the random vector $X(t)$ under $\mathds{P}$, and by $p(t) \equiv p (t, \, \cdot \, )$ the corresponding probability density function, at each time $t \geqslant 0$. This function solves the equation \hyperref[fpeqnwfp]{(\ref*{fpeqnwfp})} with initial condition \hyperref[icnwfp]{(\ref*{icnwfp})}.

\smallskip

An important role will be played by the \textit{Radon--Nikod\'{y}m derivative}, or \textit{likelihood ratio process}, 
\begin{equation} \label{rndlr}
\frac{\textnormal{d}P(t)}{\textnormal{d}\mathrm{Q}}\big(X(t)\big) = \ell\big(t,X(t)\big), 
\qquad \textnormal{ where } \quad 
\ell(t,x) \vcentcolon = \frac{p(t,x)}{q(x)} = p(t,x) \, \mathrm{e}^{2 \Psi(x)}
\end{equation}
for $t \geqslant 0$ and $x \in \mathds{R}^{n}$. Here and throughout, we denote by $\mathrm{Q}$ the $\sigma$-finite measure on the Borel sets of $\mathds{R}^{n}$, whose density with respect to Lebesgue measure is
\begin{equation} \label{qdq}
q(x) \vcentcolon = \mathrm{e}^{-2\Psi(x)}, \qquad x \in \mathds{R}^{n}.
\end{equation}
The \textit{relative entropy} and the \textit{relative Fisher information} (see, e.g., \cite{OV00,CT06}) of $P(t)$ with respect to this measure $\mathrm{Q}$, are defined respectively as
\begingroup
\addtolength{\jot}{0.7em}
\begin{align}
H\big( P(t) \, \vert \, \mathrm{Q} \big) &\vcentcolon = \mathds{E}_{\mathds{P}}\big[ \log \ell\big(t,X(t)\big) \big] = \int_{\mathds{R}^{n}} \log \bigg(\frac{p(t,x)}{q(x)}\bigg) \, p(t,x) \, \textnormal{d}x, \qquad t \geqslant 0, \label{doref} \\    
I\big( P(t) \, \vert \, \mathrm{Q}\big) &\vcentcolon = \mathds{E}_{\mathds{P}}\Big[ \, \big\vert \nabla \log \ell\big(t,X(t)\big) \big\vert^{2} \, \Big] = \int_{\mathds{R}^{n}} \big\vert \nabla \log \ell(t,x) \big\vert^{2} \, p(t,x) \, \textnormal{d}x, \qquad t \geqslant 0. \label{rfi}
\end{align}
\endgroup
It follows from Section 2 in \cite{Leo14} (see also Appendix C in \cite{KST20}) that the relative entropy $H( P \, \vert \, \mathrm{Q})$ is well-defined and takes values in $(-\infty,\infty]$ if the probability measure $P$ has finite second moment. The latter is always the case in our paper.

\smallskip

Direct computation reveals that, along the curve of probability measures $(P(t))_{t \geqslant 0}$, the free energy functional \hyperref[fef]{(\ref*{fef})} and the relative entropy \hyperref[doref]{(\ref*{doref})} are related for each $t \geqslant 0$ through the equation
\begin{equation} \label{reeef} 
2 \, \mathscr{F}\big(p(t, \, \cdot \,)\big) = H\big( P(t) \, \vert \, \mathrm{Q} \big).
\end{equation}
This shows that studying the decay of the free energy $t \mapsto \mathscr{F}(p(t, \, \cdot \,))$ is equivalent to studying the decay of the relative entropy $t \mapsto H( P(t) \, \vert \, \mathrm{Q})$, a key aspect of thermodynamics. In light of condition \hyperref[saosaojko]{\ref*{saosaojko}} in \hyperref[osaojko]{Assumptions \ref*{osaojko}} below, the identity \hyperref[reeef]{(\ref*{reeef})} implies that $H( P(0) \, \vert \, \mathrm{Q} )$ is finite, so the quantity in \hyperref[doref]{(\ref*{doref})} is finite for $t = 0$; thus, on account of \hyperref[stewrtpefitre]{(\ref*{stewrtpefitre})} below, finite also for $t > 0$.

\subsection{Regularity assumptions} \label{sub.reg.ass.}

In order to provide mathematically precise formulations of subsequent results, we have to specify convenient regularity assumptions. These issues are of a rather technical nature, and \hyperref[sub.reg.ass.]{Subsection \ref*{sub.reg.ass.}} might be skipped at a first reading of this paper. 

By analogy with \cite[Theorem 5.1]{JKO98} we consider the following assumptions.
\begin{assumptions} \label{osaojko} \
\begin{enumerate}[label=(\roman*)] 
\item \label{faosaojko} The potential $\Psi \colon \mathds{R}^{n} \rightarrow [0,\infty)$ is of class $\mathcal{C}^{\infty}(\mathds{R}^{n};[0,\infty))$.
\item \label{saosaojko} The distribution $P(0)$ of $X(0)$ in \hyperref[sdeids]{(\ref*{sdeids})} has probability density function $p^{0} = p(0, \, \cdot \,)$ with respect to Lebesgue measure on $\mathds{R}^{n}$, with finite second moment and free energy, i.e.,
\begin{equation} \label{ffecaoo}
\int_{\mathds{R}^{n}} p^{0}(x) \, \vert x \vert^{2} \, \textnormal{d}x < \infty  
\qquad \textnormal{ and } \qquad
\mathscr{F}(p^{0}) = \tfrac{1}{2} \, H\big( P(0) \, \vert \, \mathrm{Q} \big) \in (-\infty,\infty). 
\end{equation}
\end{enumerate}
\end{assumptions}

In \cite{JKO98} it is also assumed that the potential $\Psi$ satisfies, for some real constant $C > 0$, the bound $\vert \nabla \Psi \vert \leqslant C \, (\Psi + 1)$, which we do not need here. Instead of this requirement, we shall impose the following rather weak assumptions.

\begin{assumptions}[\textsf{Regularity assumptions for the trajectorial results of the present paper}] \label{sosaojkoia} In addition to conditions \hyperref[faosaojko]{\ref*{faosaojko}} and \hyperref[saosaojko]{\ref*{saosaojko}} of \hyperref[osaojko]{Assumptions \ref*{osaojko}}, we also impose that:
\begin{enumerate}[label=(\roman*)] 
\setcounter{enumi}{2}
\item \label{naltsaosaojko} The potential $\Psi$ satisfies, for some real constants $c \geqslant 0$ and $R \geqslant 0$, the drift (or coercivity) condition
\begin{equation} \label{tppstdc}
\forall \, x \in \mathds{R}^{n}, \vert x \vert \geqslant R \colon \qquad \big\langle x \, , \nabla \Psi(x) \big\rangle \geqslant - c \, \vert x \vert^{2}.
\end{equation}
\item \label{tsaosaojko} The potential $\Psi$ is sufficiently well-behaved to guarantee that the solution of \hyperref[sdeids]{(\ref*{sdeids})} is unique and well-defined for all $t \geqslant 0$, and that the solution $(t,x) \mapsto p(t,x)$ of \hyperref[fpeqnwfp]{(\ref*{fpeqnwfp})} with initial condition \hyperref[icnwfp]{(\ref*{icnwfp})} is continuous and strictly positive on $(0,\infty) \times \mathds{R}^{n}$, differentiable with respect to the time variable $t$ for each $x \in \mathds{R}^{n}$, and smooth in the space variable $x$ for each $t > 0$. We also assume that the logarithmic derivative $(t,x) \mapsto \nabla \log p(t,x)$ is continuous on $(0,\infty) \times \mathds{R}^{n}$. For example, by requiring that all derivatives of $\Psi$ grow at most exponentially as $\vert x \vert$ tends to infinity, one may adapt the arguments from \cite{Rog85} showing that this is indeed the case.
\end{enumerate}
For the formulation of \hyperref[thetthre]{Theorem \ref*{thetthre}} we will need a vector field $\beta \colon \mathds{R}^{n} \rightarrow \mathds{R}^{n}$ which is the gradient of a potential $B \colon \mathds{R}^{n} \rightarrow \mathds{R}$ satisfying the following regularity assumption:
\begin{enumerate}[label=(\roman*)] 
\setcounter{enumi}{4} 
\item \label{naltsaosaojkos} The potential $B \colon \mathds{R}^{n} \rightarrow \mathds{R}$ is of class $\mathcal{C}^{\infty}(\mathds{R}^{n};\mathds{R})$ and has compact support. Consequently, its gradient $\beta \vcentcolon = \nabla B \colon \mathds{R}^{n} \rightarrow \mathds{R}^{n}$ is of class $\mathcal{C}^{\infty}(\mathds{R}^{n};\mathds{R}^{n})$ and again compactly supported. We also assume that, for every such $\beta$, the perturbed potential $\Psi + B$ satisfies condition \hyperref[tsaosaojko]{\ref*{tsaosaojko}}.
\end{enumerate}
\end{assumptions}

The \hyperref[sosaojkoia]{Assumptions \ref*{sosaojkoia}} are satisfied by typical convex potentials $\Psi$. They also accommodate examples such as double-well potentials of the form $\Psi(x) = (x^{2}-\alpha^{2})^{2}$ on the real line, for real constants $\alpha > 0$. It is important to point out, that these assumptions do not rule out the case when the constant $Z$ in \hyperref[czpf]{(\ref*{czpf})} is infinite; thus, they allow for cases (such as $\Psi \equiv 0$) in which the stationary probability density function $q_{Z}$ in \hyperref[usdotgdm]{(\ref*{usdotgdm})} does not exist. In fact, in \cite{JKO98} the authors point out explicitly that, even when the stationary probability density $q_{Z}$ is not defined, the free energy \textnormal{\hyperref[fef]{(\ref*{fef})}} of a density $p(t,x)$ satisfying the Fokker--Planck equation \hyperref[fpeqnwfp]{(\ref*{fpeqnwfp})} with initial condition \hyperref[icnwfp]{(\ref*{icnwfp})} can be defined, provided that the free energy $\mathscr{F}(p^{0})$ is finite. Furthermore, we note that the \hyperref[sosaojkoia]{Assumptions \ref*{sosaojkoia}} are designed in such a way that they are invariant when passing from the potential $\Psi$ to $\Psi + B$ if $B$ satisfies condition \hyperref[naltsaosaojkos]{\ref*{naltsaosaojkos}}.

Under the \textnormal{\hyperref[sosaojkoia]{Assumptions \ref*{sosaojkoia}}}, the Langevin--Smoluchowski diffusion equation \textnormal{\hyperref[sdeids]{(\ref*{sdeids})}} with initial distribution $P(0)$ admits a pathwise unique, strong solution, which satisfies $P(t) \in \mathscr{P}_{2}(\mathds{R}^{n})$ for all $t \geqslant 0$; here $\mathscr{P}_{2}(\mathds{R}^{n})$ is the set of probability measures on the Borel sets of $\mathds{R}^{n}$ with finite second moment. Indeed, the drift condition \hyperref[tppstdc]{(\ref*{tppstdc})} guarantees that the second-moment condition in \hyperref[ffecaoo]{(\ref*{ffecaoo})} propagates in time, i.e.,
\begin{equation} \label{1.12}
\forall \, t \geqslant 0 \colon \qquad \int_{\mathds{R}^{n}} p(t,x) \, \vert x \vert^{2} \, \textnormal{d}x < \infty;
\end{equation}
see the first problem on p.\ 125 of \cite{Fri75}, and Appendix B in \cite{KST20}.

\begin{assumptions}[\textsf{Regularity assumptions regarding the Wasserstein distance}] \label{sosaojkoianoew} In addition to conditions \hyperref[faosaojko]{\ref*{faosaojko}} -- \hyperref[naltsaosaojkos]{\ref*{naltsaosaojkos}} of \hyperref[sosaojkoia]{Assumptions \ref*{sosaojkoia}}, we require that:
\begin{enumerate}[label=(\roman*)] 
\setcounter{enumi}{5} 
\item \label{nalwstasas} For every $t \geqslant 0$, there exists a sequence of functions $(\varphi_{m}(t, \cdot \,))_{m \geqslant 1} \subseteq \mathcal{C}_{c}^{\infty}(\mathds{R}^{n};\mathds{R})$, whose gradients $( \nabla \varphi_{m}(t, \cdot \,))_{m \geqslant 1}$ converge in $L^{2}(P(t))$ to the velocity field $v(t, \, \cdot \,) = \nabla \varphi(t, \, \cdot \, )$ of gradient type as in \hyperref[tdvfvtx]{(\ref*{tdvfvtx})} with $\varphi(t,x) = - \Psi(x) - \frac{1}{2} \log p(t,x)$, as $m \rightarrow \infty$. 
\end{enumerate}
\end{assumptions}

This last requirement guarantees, for every $t \geqslant 0$, that the velocity field $v(t, \, \cdot \,)$ is an element of the tangent space of $\mathscr{P}_{2}(\mathds{R}^{n})$ at the point $P(t) \in \mathscr{P}_{2}(\mathds{R}^{n})$ in the sense of \cite[Definition 8.4.1]{AGS08}. For the details we refer to \hyperref[stwt]{Section \ref*{stwt}} below, in particular, the display \hyperref[tanpcvecup]{(\ref*{tanpcvecup})}. However, we do not know whether this condition \hyperref[nalwstasas]{\ref*{nalwstasas}} in \hyperref[sosaojkoianoew]{Assumptions \ref*{sosaojkoianoew}} is actually an additional requirement, or whether it is automatically satisfied in our setting. But as this issue only affects the Wasserstein distance, and has no relevance for our novel trajectorial results \hyperref[thetsix]{Theorems \ref*{thetsix}} and \hyperref[thetthretv]{\ref*{thetthretv}} which constitute the main point of this work, we will not pursue this issue here further.

The condition \hyperref[nalwstasas]{\ref*{nalwstasas}} in \hyperref[sosaojkoianoew]{Assumptions \ref*{sosaojkoianoew}} is satisfied by simple potentials such as for example $\Psi \equiv 0$ or $\Psi(x) = \frac{1}{2} \vert x \vert^{2}$. More generally, potentials with a curvature lower bound $\textnormal{Hess}(\Psi) \geqslant \kappa \, I_{n}$, for some $\kappa \in \mathds{R}$ (as in \hyperref[sndcbe]{(\ref*{sndcbe})} below), for instance the double-well potential $\Psi(x) = (x^{2}-\alpha^{2})^{2}$ on the real line, satisfy this condition; more on this theme can be found in \cite[Theorem 10.4.13]{AGS08}, as was kindly pointed out to us by Luigi Ambrosio.


\section{The main theorems in aggregate form} \label{thetheoseccth}


In light of \hyperref[reeef]{(\ref*{reeef})}, the goal of \cite{JKO98} is to relate the decay of the \textit{relative entropy functional}
\begin{equation} \label{dotref}
\mathscr{P}_{2}(\mathds{R}^{n}) \ni P \longmapsto H(P \, \vert \, \mathrm{Q}) \in (-\infty,\infty]
\end{equation}
along the curve $(P(t))_{t \geqslant 0}$, to the quadratic Wasserstein distance 
\begin{equation}
W_{2}(\mu,\nu) = \Big( \, \inf_{\scriptscriptstyle Y \sim \mu, Z \sim \nu} \mathds{E} \vert Y - Z \vert^{2} \,  \Big)^{1/2}, \qquad \mu, \nu \in \mathscr{P}_{2}(\mathds{R}^{n})
\end{equation}
on $\mathscr{P}_{2}(\mathds{R}^{n})$ (cf.\ \cite{Vil03, AGS08, AG13}). We resume the remarkable relation between these two quantities in the following two theorems; these quantify the relationship between displacement in the ambient space (the denominator in \hyperref[ttofeosl]{(\ref*{ttofeosl})}) and fluctuations of the free energy, or equivalently of the relative entropy (the numerator in \hyperref[ttofeosl]{(\ref*{ttofeosl})}). The proofs will be given in \hyperref[subimpconq]{Subsection \ref*{subimpconq}} below.

\medskip

\begin{theorem} \label{thetone} Under the \textnormal{\hyperref[sosaojkoianoew]{Assumptions \ref*{sosaojkoianoew}}}, the relative Fisher information $I( P(t_{0}) \, \vert \, \mathrm{Q})$ is finite for Lebesgue-a.e.\ $t_{0} \geqslant 0$, and we have the \textnormal{\textsf{generalized de Bruijn identity}}
\begin{equation} \label{flffflnv}
\lim_{t \rightarrow t_{0}} \, \frac{H\big( P(t) \, \vert \, \mathrm{Q} \big) - H\big( P(t_{0}) \, \vert \, \mathrm{Q}\big)}{t-t_{0}} 
= - \tfrac{1}{2} \, I\big( P(t_{0}) \, \vert \, \mathrm{Q}\big),
\end{equation}
as well as the limiting behavior of the quadratic Wasserstein distance
\begin{equation} \label{agswtuffasihnv}
\lim_{t \rightarrow t_{0}} \, \frac{W_{2}\big(P(t),P(t_{0})\big)}{\vert t - t_{0} \vert} 
= \tfrac{1}{2} \, \sqrt{I\big( P(t_{0}) \, \vert \, \mathrm{Q}\big)},
\end{equation}
so that
\begin{equation} \label{ttofeosl}
\lim_{t \rightarrow t_{0}} \, \Bigg( \operatorname{sgn}(t-t_{0}) \cdot \frac{H\big( P(t) \, \vert \, \mathrm{Q} \big) - H\big( P(t_{0}) \, \vert \, \mathrm{Q}\big) }{W_{2}\big( P(t),P(t_{0})\big)} \Bigg)
= - \sqrt{I\big( P(t_{0}) \, \vert \, \mathrm{Q}\big)}.
\end{equation}
Furthermore, if $t_{0} \geqslant 0$ is chosen so that the generalized de Bruijn identity \textnormal{\hyperref[flffflnv]{(\ref*{flffflnv})}} does hold, then the limiting assertions \textnormal{\hyperref[agswtuffasihnv]{(\ref*{agswtuffasihnv})}} and \textnormal{\hyperref[ttofeosl]{(\ref*{ttofeosl})}} are also valid.
\end{theorem}

\medskip

The ratio on the left-hand side of \hyperref[ttofeosl]{(\ref*{ttofeosl})} can be interpreted as the rate of decay for the relative entropy functional \hyperref[dotref]{(\ref*{dotref})} at $P = P(t_{0})$ along the curve $(P(t))_{t \geqslant 0}$, if distances in the ambient space $\mathscr{P}_{2}(\mathds{R}^{n})$ are measured by the quadratic Wasserstein distance $W_{2}$. The quantity appearing on the right-hand side of \hyperref[ttofeosl]{(\ref*{ttofeosl})} is the square root of the relative Fisher information in \hyperref[rfi]{(\ref*{rfi})}, written more explicitly in terms of the ``score function'' $\nabla \ell(t, \, \cdot \,) / \ell(t, \, \cdot \,)$ as
\begin{equation} \label{merfi}
I\big( P(t) \, \vert \, \mathrm{Q}\big) = \mathds{E}_{\mathds{P}}\Bigg[ \ \frac{\big\vert \nabla \ell\big(t,X(t)\big) \big\vert^{2}}{\ell\big(t,X(t)\big)^{2}} \ \Bigg] 
= \int_{\mathds{R}^{n}} \bigg\vert \frac{\nabla p(t,x)}{p(t,x)} + 2 \, \nabla \Psi(x) \bigg\vert^{2} \, p(t,x) \, \textnormal{d}x.
\end{equation}

\smallskip

For future reference, we denote by $N$ the set of exceptional points $t_{0} \geqslant 0$ for which the right-sided version of the limit in \hyperref[flffflnv]{(\ref*{flffflnv})}, i.e., the limiting assertion  
\begin{equation} \label{rgtsflffflnv}
\lim_{t \downarrow t_{0}} \, \frac{H\big( P(t) \, \vert \, \mathrm{Q} \big) - H\big( P(t_{0}) \, \vert \, \mathrm{Q}\big)}{t-t_{0}} 
= - \tfrac{1}{2} \, I\big( P(t_{0}) \, \vert \, \mathrm{Q}\big),
\end{equation}
fails. According to \hyperref[thetone]{Theorem \ref*{thetone}}, this exceptional set $N$ has zero Lebesgue measure.

\bigskip

The remarkable insight of \cite{JKO98} states that the rate of entropy decay \hyperref[ttofeosl]{(\ref*{ttofeosl})} along the curve $(P(t))_{t \geqslant 0}$ is, in fact, the slope of \textit{steepest descent} for the relative entropy functional \hyperref[dotref]{(\ref*{dotref})} with respect to the Wasserstein distance $W_{2}$ at the point $P = P(t_{0})$ on the curve. To formalize this assertion, we fix a time $t_{0} \geqslant 0$ and let the vector field $\beta = \nabla B \colon \mathds{R}^{n} \rightarrow \mathds{R}^{n}$ be the gradient of a potential $B$, as in condition \hyperref[naltsaosaojkos]{\ref*{naltsaosaojkos}} of \hyperref[sosaojkoia]{Assumptions \ref*{sosaojkoia}}. This gradient vector field $\beta$ will serve as a perturbation in
\begin{equation} \label{pfpeq}
\partial_{t} p^{\beta}(t,x) = \operatorname{div}\Big(\big(\nabla \Psi(x) + \beta(x) \big) \, p^{\beta}(t,x) \Big) + \tfrac{1}{2} \Delta p^{\beta}(t,x), \qquad (t,x) \in (t_{0},\infty) \times \mathds{R}^{n},
\end{equation}
the thus perturbed Fokker--Planck equation with initial condition
\begin{equation} \label{pic}
p^{\beta}(t_{0},x) = p(t_{0},x), \qquad x \in \mathds{R}^{n}.
\end{equation}
We denote by $\mathds{P}^{\beta}$ the probability measure on path space $\Omega = \mathcal{C}([t_{0},\infty);\mathds{R}^{n})$, under which the canonical coordinate process $(X(t))_{t \geqslant t_{0}}$ satisfies the stochastic differential equation 
\begin{equation} \label{wpsdeids}
\textnormal{d}X(t) = - \Big( \nabla \Psi\big(X(t)\big) + \beta\big(X(t)\big) \Big) \, \textnormal{d}t + \textnormal{d}W^{\beta}(t), \qquad t \geqslant t_{0}
\end{equation}
with initial probability distribution $P(t_{0})$. Here, the process $(W^{\beta}(t))_{t \geqslant t_{0}}$ is Brownian motion under $\mathds{P}^{\beta}$. The probability distribution of $X(t)$ under $\mathds{P}^{\beta}$ on $\mathds{R}^{n}$ will be denoted by $P^{\beta}(t)$, for $t \geqslant t_{0}$; as before, the corresponding probability density function $p^{\beta}(t) \equiv p^{\beta}(t, \, \cdot \, )$ solves the equation \hyperref[pfpeq]{(\ref*{pfpeq})} subject to the initial condition \hyperref[pic]{(\ref*{pic})}.

\smallskip

After these preparations we can state the result formalizing the gradient flow, or \textit{steepest descent}, property of the curve $(P(t))_{t \geqslant 0}$ generated by the Langevin--Smoluchowski diffusion \textnormal{\hyperref[sdeids]{(\ref*{sdeids})}} in the ambient space of probability measures $\mathscr{P}_{2}(\mathds{R}^{n})$ endowed with the quadratic Wasserstein metric.

\begin{theorem} \label{thetthre} Under the \textnormal{\hyperref[sosaojkoianoew]{Assumptions \ref*{sosaojkoianoew}}}, the following assertions hold for every point $t_{0} \in \mathds{R}_{+} \setminus N$ \textnormal{(}at which the right-sided limiting identity \textnormal{\hyperref[rgtsflffflnv]{(\ref*{rgtsflffflnv})}} is valid\textnormal{):}

\smallskip

\noindent The $\mathds{R}^{n}$-valued random vectors 
\begin{equation} \label{ttrvzo}
a \vcentcolon = \nabla \log \ell\big(t_{0},X(t_{0})\big) = \nabla \log p\big(t_{0},X(t_{0})\big) + 2 \, \nabla \Psi\big( X(t_{0}) \big) \, , \, \qquad b \vcentcolon = \beta\big(X(t_{0})\big)    
\end{equation}
are elements of the Hilbert space $L^{2}(\mathds{P})$, and the perturbed version of the generalized de Bruijn identity \textnormal{\hyperref[flffflnv]{(\ref*{flffflnv})}} reads  
\begin{equation} \label{tatpvotgdbi}
\lim_{t \downarrow t_{0}} \, \frac{H\big( P^{\beta}(t) \, \vert \, \mathrm{Q} \big) - H\big( P^{\beta}(t_{0}) \, \vert \, \mathrm{Q}\big)}{t-t_{0}} 
= - \tfrac{1}{2} \, I\big( P(t_{0}) \, \vert \, \mathrm{Q}\big) - \langle a,b \rangle_{L^{2}(\mathds{P})} = -  \tfrac{1}{2} \, \big\langle a , a + 2b \big\rangle_{L^{2}(\mathds{P})}.
\end{equation}
The limiting behavior of the quadratic Wasserstein distance \textnormal{\hyperref[agswtuffasihnv]{(\ref*{agswtuffasihnv})}} in this perturbed context is given by
\begin{equation} \label{ftlbotwditpc}
\lim_{t \downarrow t_{0}} \, \frac{W_{2}\big( P^{\beta}(t),P^{\beta}(t_{0})\big)}{t-t_{0}}
=  \tfrac{1}{2}  \, \| a + 2 b\|_{L^{2}(\mathds{P})}.
\end{equation}
Combining \textnormal{\hyperref[tatpvotgdbi]{(\ref*{tatpvotgdbi})}} with \textnormal{\hyperref[ftlbotwditpc]{(\ref*{ftlbotwditpc})}}, and assuming $a + 2b \neq 0$, we have
\begin{equation}
\lim_{t \downarrow t_{0}} \, \frac{H\big( P^{\beta}(t) \, \vert \, \mathrm{Q} \big) - H\big( P^{\beta}(t_{0}) \, \vert \, \mathrm{Q}\big)}{W_{2}\big( P^{\beta}(t),P^{\beta}(t_{0})\big)} 
= - \Bigg\langle a \, , \, \frac{a+2b}{\| a + 2b\|_{L^{2}(\mathds{P})}} \Bigg\rangle_{L^{2}(\mathds{P})} \, ,
\end{equation}
and therefore
\begingroup
\addtolength{\jot}{0.7em}
\begin{align} 
\lim_{t \downarrow t_{0}} \, \Bigg( \, \frac{H\big( P^{\beta}(t) \, \vert \, \mathrm{Q} \big) - H\big( P^{\beta}(t_{0}) \, \vert \, \mathrm{Q}\big)}{W_{2}\big( P^{\beta}(t),P^{\beta}(t_{0})\big)} \, &- \, 
\frac{H\big( P(t) \, \vert \, \mathrm{Q} \big) - H\big( P(t_{0}) \, \vert \, \mathrm{Q}\big)}{W_{2}\big( P(t),P(t_{0})\big)} \, \Bigg) \label{nlwpthtt} \\
&= \|a\|_{L^{2}(\mathds{P})}  - \Bigg\langle a \, , \, \frac{a+2b}{\| a + 2b\|_{L^{2}(\mathds{P})}} \Bigg\rangle_{L^{2}(\mathds{P})} \, . \label{slnlwpthtt}
\end{align}
\endgroup
\end{theorem}

On the strength of the Cauchy--Schwarz inequality, the expression in \hyperref[slnlwpthtt]{(\ref*{slnlwpthtt})} is non-negative, and vanishes if and only if $a+2b$ is a positive multiple of $a$. Consequently, when the vector field $\beta$ is not a scalar multiple of $\nabla \log \ell(t_{0},\, \cdot \,)$, the difference of the two slopes in \hyperref[nlwpthtt]{(\ref*{nlwpthtt})} is strictly positive. In other words, the slope quantified by the first term of the difference \hyperref[nlwpthtt]{(\ref*{nlwpthtt})}, is then strictly bigger than the (negative) slope expressed by the second term of \hyperref[nlwpthtt]{(\ref*{nlwpthtt})}.

\medskip

These two theorems are essentially well known. They build upon a vast amount of previous work. In the quadratic case $\Psi(x) = \frac{1}{2} \vert x \vert^{2}$, i.e., when the process $(X(t))_{t \geqslant 0}$ in \hyperref[sdeids]{(\ref*{sdeids})} is Ornstein--Uhlenbeck with invariant measure in \hyperref[usdotgdm]{(\ref*{usdotgdm})} being standard Gaussian, the relation
\begin{equation} \label{dbirbtfimsf}
\tfrac{\textnormal{d}}{\textnormal{d}t} \, H\big( P(t) \, \vert \, \mathrm{Q} \big) = - \tfrac{1}{2} \, I\big( P(t) \, \vert \, \mathrm{Q}\big)
\end{equation}
has been known since \cite{Sta59} as \textit{de Bruijn's identity}. This relationship between the two fundamental information measures, due to Shannon and Fisher, respectively, is a dominant theme in many aspects of information theory and probability. We refer to the book \cite{CT06} by Cover and Thomas for an account of the results by Barron, Blachman, Brown, Linnik, R\'{e}nyi, Shannon, Stam and many others; in a similar vein, see also the seminal work \cite{BE85} by Bakry and {\'E}mery, as well as the paper \cite{MV00} by Markowich and Villani, and the book \cite{Vil03} by Villani. Consult also Carlen and Soffer \cite{CS91} and Johnson \cite{Joh04} on the relation of \hyperref[dbirbtfimsf]{(\ref*{dbirbtfimsf})} to the central limit theorem. For the connections with large deviations we refer to \cite{ADPZ13} and \cite{Fat16}. 

\smallskip

The paper \cite{JKO98} broke new ground in this respect, as it considered a general potential $\Psi$ and established the relation to the quadratic Wasserstein distance, culminating with the characterization of the curve $(P(t))_{t \geqslant 0}$ as a gradient flow. This relation was further investigated by Otto in the paper \cite{Ott01}, where the theory now known as \textit{``Otto calculus''} was developed. For a recent application of Otto calculus to the Schr\"odinger problem, see \cite{GLR20}.

\smallskip

The statements of our \hyperref[thetone]{Theorems \ref*{thetone}} and \hyperref[thetthre]{\ref*{thetthre}} complement the existing results in some details, e.g., the precise form \hyperref[slnlwpthtt]{(\ref*{slnlwpthtt})}, measuring the difference of the two slopes appearing in \hyperref[nlwpthtt]{(\ref*{nlwpthtt})}. The main novelty of our approach, however, will only become apparent below with the formulation of \hyperref[thetsix]{Theorems \ref*{thetsix}} and \hyperref[thetthretv]{\ref*{thetthretv}}, the trajectorial versions of \hyperref[thetone]{Theorems \ref*{thetone}} and \hyperref[thetthre]{\ref*{thetthre}}. 


\section{The main theorems in trajectorial form} \label{thetheoscthtra}


Our main goal is to investigate \hyperref[thetone]{Theorems \ref*{thetone}} and \hyperref[thetthre]{\ref*{thetthre}} in a trajectorial fashion, by considering the \textit{relative entropy process} 
\begin{equation} \label{stlrpd}
\log \ell\big(t,X(t)\big) = \log \Bigg( \frac{p\big(t,X(t)\big)}{q\big(X(t)\big)} \Bigg) = \log p\big(t,X(t)\big) + 2 \, \Psi\big(X(t)\big) \, , \qquad t \geqslant 0
\end{equation}
along each trajectory of the canonical coordinate process $(X(t))_{t \geqslant 0}$, and calculating its dynamics (stochastic differential) under the probability measure $\mathds{P}$. The $\mathds{P}$-expectation of this quantity is, of course, the relative entropy in \hyperref[doref]{(\ref*{doref})}. A decisive tool in the analysis of the relative entropy process \hyperref[stlrpd]{(\ref*{stlrpd})} is to \textit{reverse time}, and use a remarkable insight due to Fontbona and Jourdain \cite{FJ16}. These authors consider the canonical coordinate process $(X(t))_{0 \leqslant t \leqslant T}$ on path space $\Omega = \mathcal{C}([0,T];\mathds{R}^{n})$ in the reverse direction of time, i.e., they work with the time-reversed process $(X(T-s))_{0 \leqslant s \leqslant T}$; it is then notationally convenient to consider a finite time interval $[0,T]$, rather than $\mathds{R}_{+}$. 

\smallskip

At this stage it becomes important to specify the relevant filtrations: We denote by $(\mathcal{F}(t))_{t \geqslant 0}$ the smallest continuous filtration to which the canonical coordinate process $(X(t))_{t \geqslant 0}$ is adapted. That is, modulo $\mathds{P}$-augmentation, we have
\begin{equation} \label{fftbfmgtmt}
\mathcal{F}(t) = \sigma \big( X(u) \colon \, 0 \leqslant u \leqslant t \big), \qquad t \geqslant 0;
\end{equation} 
and we call $(\mathcal{F}(t))_{t \geqslant 0}$ the ``filtration generated by $(X(t))_{t \geqslant 0}$''. Likewise, we let $(\mathcal{G}(T-s))_{0 \leqslant s \leqslant T}$ be the ``filtration generated by the time-reversed canonical coordinate process $(X(T-s))_{0 \leqslant s \leqslant T}$'' in the same sense as before. In other words,
\begin{equation} \label{tbfmgtmt}
\mathcal{G}(T-s) = \sigma \big( X(T-u) \colon \, 0 \leqslant u \leqslant s \big), \qquad 0 \leqslant s \leqslant T,
\end{equation}
modulo $\mathds{P}$-augmentation. For the necessary measure-theoretic operations that ensure the continuity (from both left and right) of filtrations associated with continuous processes, the reader may consult Section 2.7 in \cite{KS88}; in particular, Problems 7.1 -- 7.6 and Proposition 7.7.

\smallskip

The following two \hyperref[thetsix]{Theorems \ref*{thetsix}} and \hyperref[thetthretv]{\ref*{thetthretv}} are the main new results of this paper. They can be regarded as trajectorial versions of \hyperref[thetone]{Theorems \ref*{thetone}} and \hyperref[thetthre]{\ref*{thetthre}}, whose proofs will follow from \hyperref[thetsix]{Theorems \ref*{thetsix}} and \hyperref[thetthretv]{\ref*{thetthretv}} simply by taking expectations. Similar trajectorial approaches have already been applied successfully to the temporal dissipation of relative entropy and Fisher information \cite{FJ16}, to the theory of optimal stopping \cite{DK94}, to Doob's martingale inequalities \cite{ABPST13}, and to the Burkholder--Davis--Gundy inequality \cite{BS15}. 

\smallskip

The significance of \hyperref[thetsix]{Theorem \ref*{thetsix}} right below, is that the trade-off between the temporal decay of relative entropy, and the temporal growth of the quadratic Wasserstein distance along the curve of probability measures $(P(t))_{t \geqslant 0}$, both of which are characterized in terms of the cumulative relative Fisher information process, is valid not only in expectation, but also along (almost) every trajectory, provided we run time in the reverse direction.\footnote{As David Kinderlehrer kindly pointed out to the second named author, the implicit Euler scheme used in \cite{JKO98} also reflects the idea of going back in time at each step of the discretization.}

\begin{theorem} \label{thetsix} Under the \textnormal{\hyperref[sosaojkoia]{Assumptions \ref*{sosaojkoia}}}, we fix $T \in (0,\infty)$ and define the cumulative relative Fisher information process, accumulated from the right, as
\begingroup
\addtolength{\jot}{0.7em}
\begin{equation} \label{fispdaftr}
\begin{aligned} 
F(T-s) \vcentcolon =& \int_{0}^{s} \frac{1}{2} \frac{\big\vert \nabla \ell\big(T-u,X(T-u)\big) \big\vert^{2}}{\ell\big(T-u,X(T-u)\big)^{2}} \, \textnormal{d}u \\
=& \int_{0}^{s} \frac{1}{2} \bigg\vert \frac{\nabla p\big(T-u,X(T-u)\big)}{p\big(T-u,X(T-u)\big)} + 2 \, \nabla \Psi\big(X(T-u)\big) \bigg\vert^{2} \, \textnormal{d}u 
\end{aligned}
\end{equation}
\endgroup
for $0 \leqslant s \leqslant T$. Then 
\begin{equation}  \label{intfisinfint}
H\big( P(0) \, \vert \, \mathrm{Q} \big) - H\big( P(T) \, \vert \, \mathrm{Q} \big) = \mathds{E}_{\mathds{P}}\big[F(0)\big] = \tfrac{1}{2} \int_{0}^{T} I\big( P(t) \, \vert \, \mathrm{Q}\big) \, \textnormal{d}t < \infty,
\end{equation}
and the process
\begin{equation} \label{dollafipim}
M(T-s) \vcentcolon = \Big( \log \ell\big(T-s,X(T-s)\big) - \log \ell\big(T,X(T)\big) \Big) - F(T-s)  \, , \qquad 0 \leqslant s \leqslant T
\end{equation}
is a square-integrable martingale of the backwards filtration $(\mathcal{G}(T-s))_{0 \leqslant s \leqslant T}$ under the probability measure $\mathds{P}$. More explicitly, the martingale of \textnormal{\hyperref[dollafipim]{(\ref*{dollafipim})}} can be represented as 
\begin{equation} \label{erdollafipim}
M(T-s) = \int_{0}^{s} \Bigg\langle \frac{\nabla \ell\big(T-u,X(T-u)\big)}{\ell\big(T-u,X(T-u)\big)} \, , \, \textnormal{d}\overline{W}^\mathds{P}(T-u) \Bigg\rangle \, , \qquad 0 \leqslant s \leqslant T,
\end{equation}
for a $\mathds{P}$-Brownian motion $\big(\overline{W}^\mathds{P}(T-s)\big)_{0 \leqslant s \leqslant T}$ of the backwards filtration $(\mathcal{G}(T-s))_{0 \leqslant s \leqslant T}$. In particular, the quadratic variation of the martingale of \textnormal{\hyperref[dollafipim]{(\ref*{dollafipim})}} is given by the non-decreasing process in \textnormal{\hyperref[fispdaftr]{(\ref*{fispdaftr})}}, up to the multiplicative factor of $1/2$.
\end{theorem}

\smallskip

Next, we state the trajectorial version of \hyperref[thetthre]{Theorem \ref*{thetthre}} --- or equivalently, the ``perturbed'' analogue of \hyperref[thetsix]{Theorem \ref*{thetsix}}. As we did in \hyperref[thetthre]{Theorem \ref*{thetthre}}, in particular in the preceding equations \hyperref[pfpeq]{(\ref*{pfpeq})} -- \hyperref[wpsdeids]{(\ref*{wpsdeids})}, we consider the perturbation $\beta \colon \mathds{R}^{n} \rightarrow \mathds{R}^{n}$ and denote the \textit{perturbed likelihood ratio function} by
\begin{equation} \label{dotplrfb}
\ell^{\beta}(t,x) \vcentcolon = \frac{p^{\beta}(t,x)}{q(x)} = p^{\beta}(t,x) \, \mathrm{e}^{2 \Psi(x)} \, , \qquad (t,x) \in [t_{0},\infty) \times \mathds{R}^{n}.
\end{equation}
The stochastic analogue of this quantity is the \textit{perturbed likelihood ratio process} 
\begin{equation} \label{pvdotplrfbs}
\ell^{\beta}\big(t,X(t)\big) = \frac{p^{\beta}\big(t,X(t)\big)}{q\big(X(t)\big)} = p^{\beta}\big(t,X(t)\big) \, \mathrm{e}^{2 \Psi(X(t))} \, , \qquad t \geqslant t_{0}.
\end{equation}
The logarithm of this process is the \textit{perturbed relative entropy process}
\begin{equation} \label{prerepitufdllpitufd} 
\log \ell^{\beta}\big(t,X(t)\big) = \log \Bigg( \frac{p^{\beta}\big(t,X(t)\big)}{q\big(X(t)\big)} \Bigg) = \log p^{\beta}\big(t,X(t)\big) + 2 \, \Psi\big(X(t)\big) \, , \qquad  t \geqslant t_{0}.
\end{equation}

\begin{theorem} \label{thetthretv} Under the \textnormal{\hyperref[sosaojkoia]{Assumptions \ref*{sosaojkoia}}}, we let $t_{0} \geqslant 0$ and $T > t_{0}$. We define the perturbed cumulative relative Fisher information process, accumulated from the right, as
\begin{equation} \label{fispdaftrps}
F^{\beta}(T-s) \vcentcolon = \int_{0}^{s} \Bigg( \, \frac{1}{2}\frac{\big\vert \nabla \ell^{\beta}\big(T-u,X(T-u)\big) \big\vert^{2}}{\ell^{\beta}\big(T-u,X(T-u)\big)^{2}}   + \Big( \big\langle \beta \, , \, 2 \, \nabla \Psi \big\rangle - \operatorname{div} \beta \Big) \big(X(T-u)\big) \Bigg) \, \textnormal{d}u
\end{equation}
for $0 \leqslant s \leqslant T - t_{0}$. Then $\mathds{E}_{\mathds{P}^{\beta}}\big[F^{\beta}(t_{0})\big]< \infty$, and the process
\begin{equation} \label{dollafipimps}
M^{\beta}(T-s) \vcentcolon = \Big( \log \ell^{\beta}\big(T-s,X(T-s)\big) - \log \ell^{\beta}\big(T,X(T)\big) \Big) - F^{\beta}(T-s)  
\end{equation}
for $0 \leqslant s \leqslant T - t_{0}$, is a square-integrable martingale of the backwards filtration $(\mathcal{G}(T-s))_{0 \leqslant s \leqslant T - t_{0}}$ under the probability measure $\mathds{P}^{\beta}$. More explicitly, the martingale \textnormal{\hyperref[dollafipimps]{(\ref*{dollafipimps})}} can be represented as 
\begin{equation} \label{erdollafipimps}
M^{\beta}(T-s) = \int_{0}^{s} \Bigg\langle \frac{\nabla \ell^{\beta}\big(T-u,X(T-u)\big)}{\ell^{\beta}\big(T-u,X(T-u)\big)} \, , \, \textnormal{d}\overline{W}^{\mathds{P}^{\beta}}(T-u) \Bigg\rangle \, , \qquad 0 \leqslant s \leqslant T-t_{0},
\end{equation}
for a $\mathds{P}^{\beta}$-Brownian motion $\big(\overline{W}^{\mathds{P}^{\beta}}(T-s)\big)_{0 \leqslant s \leqslant T-t_{0}}$ of the backwards filtration $(\mathcal{G}(T-s))_{0 \leqslant s \leqslant T-t_{0}}$. 
\end{theorem}

\subsection{Consequences of the trajectorial results} \label{subimpconq}

Before tackling the proofs of \hyperref[thetsix]{Theorems \ref*{thetsix}} and \hyperref[thetthretv]{\ref*{thetthretv}}, we state several important consequences of these two basic results. In particular, we indicate how the corresponding assertions in the earlier \hyperref[thetone]{Theorems \ref*{thetone}} and \hyperref[thetthre]{\ref*{thetthre}} follow directly from these results by taking expectations.

\begin{corollary}[\textsf{Dissipation of relative entropy}] \label{thetsixcorao} Under the \textnormal{\hyperref[sosaojkoianoew]{Assumptions \ref*{sosaojkoianoew}}}, we have for all $t, t_{0} \geqslant 0$ the relative entropy identity
\begin{equation} \label{stewrtpefitre}
H\big( P(t) \, \vert \, \mathrm{Q} \big) - H\big( P(t_{0}) \, \vert \, \mathrm{Q}\big) = \mathds{E}_{\mathds{P}}\Bigg[ \log \Bigg( \frac{\ell\big(t,X(t)\big)}{\ell\big(t_{0},X(t_{0})\big)} \Bigg) \Bigg]  
= \mathds{E}_{\mathds{P}}\Bigg[\int_{t_{0}}^{t} \Bigg(- \frac{1}{2}\frac{\big\vert \nabla \ell\big(u,X(u)\big) \big\vert^{2}}{\ell\big(u,X(u)\big)^{2}} \, \Bigg) \, \textnormal{d}u\Bigg].
\end{equation}
Furthermore, we have for Lebesgue-a.e.\ $t_{0} \geqslant 0$ the \textnormal{\textsf{generalized de Bruijn identity}}
\begin{equation} \label{flfffl}
\lim_{t \rightarrow t_{0}} \, \frac{H\big( P(t) \, \vert \, \mathrm{Q} \big) - H\big( P(t_{0}) \, \vert \, \mathrm{Q}\big)}{t-t_{0}} 
= - \tfrac{1}{2} \, \mathds{E}_{\mathds{P}}\Bigg[ \ \frac{\big\vert \nabla \ell\big(t_{0},X(t_{0})\big) \big\vert^{2}}{\ell\big(t_{0},X(t_{0})\big)^{2}} \ \Bigg],
\end{equation}
as well as the limiting behavior of the quadratic Wasserstein distance
\begin{equation} \label{agswtuffasih}
\lim_{t \rightarrow t_{0}} \, \frac{W_{2}\big(P(t),P(t_{0})\big)}{\vert t - t_{0} \vert} 
= \tfrac{1}{2} \, \Bigg( \, \mathds{E}_{\mathds{P}}\Bigg[ \ \frac{\big\vert \nabla \ell\big(t_{0},X(t_{0})\big) \big\vert^{2}}{\ell\big(t_{0},X(t_{0})\big)^{2}} \ \Bigg] \, \Bigg)^{1/2}.
\end{equation}
If $t_{0} \geqslant 0$ is chosen so that the generalized de Bruijn identity \textnormal{\hyperref[flfffl]{(\ref*{flfffl})}} holds, then the limiting assertion \textnormal{\hyperref[agswtuffasih]{(\ref*{agswtuffasih})}} pertaining to the Wasserstein distance is also valid.
\begin{proof}[Proof of \texorpdfstring{\hyperref[thetsixcorao]{Corollary \ref*{thetsixcorao}}}{} from \texorpdfstring{\hyperref[thetsix]{Theorem \ref*{thetsix}}}{}:] The identity \hyperref[stewrtpefitre]{(\ref*{stewrtpefitre})} follows by taking expectations in \hyperref[erdollafipim]{(\ref*{erdollafipim})} with respect to the probability measure $\mathds{P}$, recalling the definitions \hyperref[fispdaftr]{(\ref*{fispdaftr})}, \hyperref[dollafipim]{(\ref*{dollafipim})}, and invoking the martingale property of the process in \hyperref[dollafipim]{(\ref*{dollafipim})} for $T \geqslant \max\{t_{0},t\}$. In particular, \hyperref[stewrtpefitre]{(\ref*{stewrtpefitre})} shows that the relative entropy function $t \mapsto H( P(t) \, \vert \, \mathrm{Q})$ from \hyperref[doref]{(\ref*{doref})}, thus also the free energy function $t \mapsto \mathscr{F}(p(t, \, \cdot \, ))$ from \hyperref[reeef]{(\ref*{reeef})}, are strictly decreasing provided $\ell(t, \, \cdot \,)$ is not constant. 

\smallskip

According to the Lebesgue differentiation theorem, the monotone function $t \mapsto H(P(t) \, \vert \, \mathrm{Q})$ is differentiable for Lebesgue-a.e.\ $t_{0} \geqslant 0$, in which case \hyperref[stewrtpefitre]{(\ref*{stewrtpefitre})} leads to the identity \hyperref[flfffl]{(\ref*{flfffl})}.

\smallskip

The limiting behavior \hyperref[agswtuffasih]{(\ref*{agswtuffasih})} of the Wasserstein distance, for Lebesgue-a.e.\ $t_{0} \geqslant 0$, is well known and worked out in \cite{AGS08}; \hyperref[stwt]{Section \ref*{stwt}} below provides details. \hyperref[agswt]{Theorem \ref*{agswt}} establishes the important, novel aspect of \hyperref[thetsixcorao]{Corollary \ref*{thetsixcorao}}; namely, its last assertion, that the validity of \hyperref[flfffl]{(\ref*{flfffl})} for some $t_{0} \geqslant 0$ implies that the limiting assertion \hyperref[agswtuffasih]{(\ref*{agswtuffasih})} also holds for the same point $t_{0}$. This seemingly harmless issue is actually quite delicate, and will be of crucial importance for our gradient flow analysis; it is here that we shall have to rely on condition \hyperref[nalwstasas]{\ref*{nalwstasas}} of \hyperref[sosaojkoianoew]{Assumptions \ref*{sosaojkoianoew}}.
\end{proof}
\end{corollary}

\begin{proof}[Proof of \texorpdfstring{\hyperref[thetone]{Theorem \ref*{thetone}}}{} from \texorpdfstring{\hyperref[thetsix]{Theorem \ref*{thetsix}}}{}:] This is a direct consequence of \hyperref[thetsixcorao]{Corollary \ref*{thetsixcorao}}. 
\end{proof}

In a manner similar to the derivation of \hyperref[thetsixcorao]{Corollary \ref*{thetsixcorao}} from \hyperref[thetsix]{Theorem \ref*{thetsix}}, we deduce now from \hyperref[thetthretv]{Theorem \ref*{thetthretv}} the following \hyperref[thetsixcoraopv]{Corollary \ref*{thetsixcoraopv}}. Its first identity \hyperref[stewrtpefitrepv]{(\ref*{stewrtpefitrepv})} shows, in particular, that the relative entropy $H( P^{\beta}(t) \, \vert \, \mathrm{Q})$ is finite for all $t \geqslant t_{0}$.  

\begin{corollary}[\textsf{Dissipation of relative entropy under perturbations}] \label{thetsixcoraopv} Under the \textnormal{\hyperref[sosaojkoianoew]{Assumptions \ref*{sosaojkoianoew}}}, we have, for all $t \geqslant t_{0} \geqslant 0$, the relative entropy identity
\begingroup
\addtolength{\jot}{0.7em}
\begin{equation} \label{stewrtpefitrepv}
\begin{aligned}
& H\big( P^{\beta}(t) \, \vert \, \mathrm{Q} \big) - H\big( P^{\beta}(t_{0}) \, \vert \, \mathrm{Q}\big) = \mathds{E}_{\mathds{P}^{\beta}}\Bigg[ \log \Bigg( \frac{\ell^{\beta}\big(t,X(t)\big)}{\ell^{\beta}\big(t_{0},X(t_{0})\big)}\Bigg)\Bigg]   \\
& \qquad = \mathds{E}_{\mathds{P}^{\beta}}\Bigg[\int_{t_{0}}^{t} \Bigg( - \frac{1}{2}\frac{\big\vert \nabla \ell^{\beta}\big(u,X(u)\big) \big\vert^{2}}{\ell^{\beta}\big(u,X(u)\big)^{2}} + \Big(\operatorname{div} \beta -  \big\langle \beta \, , \, 2 \, \nabla \Psi \big\rangle\Big)\big(X(u)\big) \Bigg) \textnormal{d}u \Bigg].
\end{aligned}
\end{equation}
\endgroup
Furthermore, for every point $t_{0} \in \mathds{R}_{+} \setminus N$ \textnormal{(}at which the right-sided limiting assertion \textnormal{\hyperref[rgtsflffflnv]{(\ref*{rgtsflffflnv})}} is valid\textnormal{)}, we have also the limiting identities
\begingroup
\addtolength{\jot}{0.7em}
\begin{align}
\lim_{t \downarrow t_{0}} \, \frac{H\big( P^{\beta}(t) \, \vert \, \mathrm{Q} \big) - H\big( P^{\beta}(t_{0}) \, \vert \, \mathrm{Q}\big)}{t-t_{0}} 
&= \mathds{E}_{\mathds{P}}\Bigg[ - \frac{1}{2} \frac{\big\vert \nabla \ell\big(t_{0},X(t_{0})\big) \big\vert^{2}}{\ell\big(t_{0},X(t_{0})\big)^{2}} + \Big( \operatorname{div} \beta -  \big\langle \beta \, , \, 2 \, \nabla \Psi \big\rangle\Big)\big(X(t_{0})\big)  \Bigg], \label{flffflpv} \\
\lim_{t \downarrow t_{0}} \, \frac{W_{2}\big( P^{\beta}(t),P^{\beta}(t_{0})\big)}{t-t_{0}} 
&= \tfrac{1}{2} \, \Bigg( \, \mathds{E}_{\mathds{P}}\Bigg[ \ \bigg\vert \frac{\nabla \ell\big(t_{0},X(t_{0})\big)}{\ell\big(t_{0},X(t_{0})\big)} + 2 \, \beta\big(X(t_{0})\big) \bigg\vert^{2} \ \Bigg] \, \Bigg)^{1/2}. \label{svpvompvv}
\end{align}
\endgroup
\begin{proof}[Proof of \texorpdfstring{\hyperref[thetsixcoraopv]{Corollary \ref*{thetsixcoraopv}}}{} from \texorpdfstring{\hyperref[thetthretv]{Theorem \ref*{thetthretv}}}{}:] Taking expectations in \hyperref[erdollafipimps]{(\ref*{erdollafipimps})} under the probability measure $\mathds{P}^{\beta}$, recalling the definitions \hyperref[fispdaftrps]{(\ref*{fispdaftrps})}, \hyperref[dollafipimps]{(\ref*{dollafipimps})}, and using the martingale property of the process in \hyperref[dollafipimps]{(\ref*{dollafipimps})} for $T \geqslant t \geqslant t_{0}$, leads to the identity \hyperref[stewrtpefitrepv]{(\ref*{stewrtpefitrepv})}. In order to derive from \hyperref[stewrtpefitrepv]{(\ref*{stewrtpefitrepv})} the limiting identity \hyperref[flffflpv]{(\ref*{flffflpv})}, extra care is needed to show that \hyperref[flffflpv]{(\ref*{flffflpv})} is valid for every time $t_{0} \in \mathds{R}_{+} \setminus N$. 

\smallskip

We shall verify in \hyperref[hctclwittmeitpocasot]{Lemma \ref*{hctclwittmeitpocasot}} of \hyperref[subsomusefullem]{Subsection \ref*{subsomusefullem}} below the following estimates on the ratio between the probability density function $p(t, \, \cdot \, )$ and its perturbed version $p^{\beta}(t, \, \cdot \, )$: For every $t_{0} \geqslant 0$ and $T > t_{0}$ there is a constant $C > 0$ such that 
\begin{equation} \label{ilpdepvhfv}
\bigg\vert \frac{\ell^{\beta}(t,x)}{\ell(t,x)} - 1 \bigg\vert = \bigg\vert \frac{p^{\beta}(t,x)}{p(t,x)} - 1 \bigg\vert \leqslant C \, (t-t_{0}) \, , \qquad (t,x) \in [t_{0},T] \times \mathds{R}^{n}
\end{equation}
as well as
\begin{equation} \label{ilpdepvhfvsv}
\mathds{E}_{\mathds{P}}\Bigg[\int_{t_{0}}^{t} \ \Bigg\vert \nabla \log \Bigg( \frac{\ell^{\beta}\big(u,X(u)\big)}{\ell\big(u,X(u)\big)} \Bigg) \Bigg\vert^{2} \, \textnormal{d}u \Bigg] 
\leqslant C \, (t-t_{0})^{2} \, , \qquad t_{0} \leqslant t \leqslant T.
\end{equation}

\smallskip

We turn now to the derivation of \hyperref[flffflpv]{(\ref*{flffflpv})} from \hyperref[stewrtpefitrepv]{(\ref*{stewrtpefitrepv})}. First, since the perturbation $\beta$ is smooth and compactly supported, and the paths of the canonical coordinate process $(X(t))_{t \geqslant 0}$ are continuous, we have 
\begin{equation} \label{llethar}
\lim_{t \downarrow t_{0}} \, \frac{1}{t-t_{0}} \, \mathds{E}_{\mathds{P}^{\beta}}\Bigg[\int_{t_{0}}^{t} \Big(  \operatorname{div} \beta -  \big\langle \beta \, , \, 2 \, \nabla \Psi \big\rangle\Big)\big(X(u)\big) \, \textnormal{d}u \Bigg] 
= \mathds{E}_{\mathds{P}^{\beta}}\Big[ \Big(\operatorname{div} \beta -  \big\langle \beta \, , \, 2 \, \nabla \Psi \big\rangle\Big)\big(X(t_{0})\big) \Big] 
\end{equation}
\textit{for every} $t_{0} \geqslant 0$. Secondly, the random variable $X(t_{0})$ has the same distribution under $\mathds{P}$, as it does under $\mathds{P}^{\beta}$, so it is immaterial whether we express the expectation on the right-hand side of \hyperref[llethar]{(\ref*{llethar})} with respect to the probability measure $\mathds{P}$ or $\mathds{P}^{\beta}$. Hence this expression equals the corresponding term on the right-hand side of \hyperref[flffflpv]{(\ref*{flffflpv})}.

\smallskip

Regarding the remaining term on the right-hand side of \hyperref[flffflpv]{(\ref*{flffflpv})}, the equality
\begin{equation} \label{tciwsolafophvs}
\lim_{t \downarrow t_{0}} \, \frac{1}{t-t_{0}} \, \mathds{E}_{\mathds{P}^{\beta}}\Bigg[\int_{t_{0}}^{t} \Bigg( -\frac{1}{2}\frac{\big\vert \nabla \ell^{\beta}\big(u,X(u)\big) \big\vert^{2}}{\ell^{\beta}\big(u,X(u)\big)^{2}} \, \Bigg) \, \textnormal{d}u \Bigg] 
= \lim_{t \downarrow t_{0}} \, \frac{1}{t-t_{0}} \, \mathds{E}_{\mathds{P}}\Bigg[\int_{t_{0}}^{t} \Bigg( -\frac{1}{2}\frac{\big\vert \nabla \ell\big(u,X(u)\big) \big\vert^{2}}{\ell\big(u,X(u)\big)^{2}} \, \Bigg) \, \textnormal{d}u \Bigg]
\end{equation}
holds as long as $t_{0} \geqslant 0$ is chosen so that one of the limits exists. Indeed, the equality
\begin{equation} \label{tciwsonlafophvs}
\lim_{t \downarrow t_{0}} \, \frac{1}{t-t_{0}} \, \mathds{E}_{\mathds{P}}\Bigg[\int_{t_{0}}^{t} \Bigg( -\frac{1}{2}\frac{\big\vert \nabla \ell^{\beta}\big(u,X(u)\big) \big\vert^{2}}{\ell^{\beta}\big(u,X(u)\big)^{2}} \, \Bigg) \, \textnormal{d}u \Bigg] 
= \lim_{t \downarrow t_{0}} \, \frac{1}{t-t_{0}} \, \mathds{E}_{\mathds{P}}\Bigg[\int_{t_{0}}^{t} \Bigg( -\frac{1}{2}\frac{\big\vert \nabla \ell\big(u,X(u)\big) \big\vert^{2}}{\ell\big(u,X(u)\big)^{2}} \, \Bigg) \, \textnormal{d}u \Bigg]
\end{equation}
follows from \hyperref[ilpdepvhfvsv]{(\ref*{ilpdepvhfvsv})}, and \hyperref[ilpdepvhfv]{(\ref*{ilpdepvhfv})} implies that it is immaterial whether we take expectations with respect to $\mathds{P}$ or $\mathds{P}^{\beta}$ in the two limits appearing in \hyperref[tciwsonlafophvs]{(\ref*{tciwsonlafophvs})}. Summing up, existence and equality of the limits in \hyperref[tciwsolafophvs]{(\ref*{tciwsolafophvs})} are guaranteed if and only if $t_{0} \in \mathds{R}_{+} \setminus N$. It develops that both limits in \hyperref[tciwsolafophvs]{(\ref*{tciwsolafophvs})} exist if $t_{0} \geqslant 0$ is not in the exceptional set $N$ of zero Lebesgue measure, and their common value is 
\begin{equation} \label{tcigwwsolafophvs}
- \tfrac{1}{2} \, I\big( P(t_{0}) \, \vert \, \mathrm{Q}\big) = - \tfrac{1}{2} \, \mathds{E}_{\mathds{P}}\Bigg[ \ \frac{\big\vert \nabla \ell\big(t_{0},X(t_{0})\big) \big\vert^{2}}{\ell\big(t_{0},X(t_{0})\big)^{2}} \ \Bigg].
\end{equation}
In conjunction with \hyperref[llethar]{(\ref*{llethar})}, which is valid for every $t_{0} \geqslant 0$, this establishes the limiting identity \hyperref[flffflpv]{(\ref*{flffflpv})} for every $t_{0} \in \mathds{R}_{+} \setminus N$. Therefore, the right-sided limiting assertion \hyperref[rgtsflffflnv]{(\ref*{rgtsflffflnv})}, and the similar perturbed limiting assertion in \hyperref[flffflpv]{(\ref*{flffflpv})}, fail on precisely the same set of exceptional points $N$. 

\medskip

As regards the final assertion we note that, by analogy with \hyperref[agswtuffasih]{(\ref*{agswtuffasih})}, the limiting behavior of the Wasserstein distance \hyperref[svpvompvv]{(\ref*{svpvompvv})}, for Lebesgue-a.e.\ $t_{0} \geqslant 0$, is well known \cite{AGS08}; details are in \hyperref[stwt]{Section \ref*{stwt}} below. More precisely, \hyperref[bvagswt]{Theorem \ref*{bvagswt}} establishes the novel and very crucial aspect, that the limiting assertion
\begin{equation} \label{osagswtuffasihnv}
\lim_{t \downarrow t_{0}} \, \frac{W_{2}\big(P(t),P(t_{0})\big)}{t - t_{0}} 
= \tfrac{1}{2} \, \sqrt{I\big( P(t_{0}) \, \vert \, \mathrm{Q}\big)}
\end{equation}
is valid for every $t_{0} \in \mathds{R}_{+} \setminus N$. Once again, concerning the relation between the limits in \hyperref[osagswtuffasihnv]{(\ref*{osagswtuffasihnv})} and \hyperref[svpvompvv]{(\ref*{svpvompvv})} pertaining to the Wasserstein distance, we discern a similar pattern as in the case of the generalized de Bruijn identity. In fact, \hyperref[bvagswt]{Theorem \ref*{bvagswt}} will tell us that the perturbed Wasserstein limit \hyperref[svpvompvv]{(\ref*{svpvompvv})} also holds for every $t_{0} \in \mathds{R}_{+} \setminus N$. 
\end{proof}
\end{corollary}

\begin{proof}[Proof of \texorpdfstring{\hyperref[thetthre]{Theorem \ref*{thetthre}}}{} from \texorpdfstring{\hyperref[thetsix]{Theorems \ref*{thetsix}}}{}, \texorpdfstring{\hyperref[thetthretv]{\ref*{thetthretv}}}{} and \texorpdfstring{\hyperref[thetsixcorao]{Corollaries \ref*{thetsixcorao}}}{}, \texorpdfstring{\hyperref[thetsixcoraopv]{\ref*{thetsixcoraopv}}}{}:] Let $t_{0} \in \mathds{R}_{+} \setminus N$, so that the limiting identities \hyperref[flffflpv]{(\ref*{flffflpv})} and \hyperref[svpvompvv]{(\ref*{svpvompvv})} from \hyperref[thetsixcoraopv]{Corollary \ref*{thetsixcoraopv}} are valid. Recalling the abbreviations in \hyperref[ttrvzo]{(\ref*{ttrvzo})}, we summarize now the identities just mentioned as
\begingroup
\addtolength{\jot}{1em}
\begin{alignat}{3}
&\lim_{t \downarrow t_{0}} \, \frac{H\big( P(t) \, \vert \, \mathrm{Q} \big) - H\big( P(t_{0}) \, \vert \, \mathrm{Q}\big)}{t-t_{0}} 
&&= - && \tfrac{1}{2} \, \| a \|_{L^{2}(\mathds{P})}^{2}, \label{nlwpthtt1} \\
&\lim_{t \downarrow t_{0}} \, \frac{W_{2}\big( P(t),P(t_{0})\big)}{t-t_{0}}
&&= && \tfrac{1}{2} \, \| a \|_{L^{2}(\mathds{P})}, \label{nlwpthtt2} \\
&\lim_{t \downarrow t_{0}} \, \frac{H\big( P^{\beta}(t) \, \vert \, \mathrm{Q} \big) - H\big( P^{\beta}(t_{0}) \, \vert \, \mathrm{Q}\big)}{t-t_{0}} 
&&= - && \tfrac{1}{2}  \, \big\langle a , a + 2b \big\rangle_{L^{2}(\mathds{P})}, \label{nlwpthtt3} \\
&\lim_{t \downarrow t_{0}} \, \frac{W_{2}\big( P^{\beta}(t),P^{\beta}(t_{0})\big)}{t-t_{0}} 
&&= && \tfrac{1}{2}  \, \| a + 2 b\|_{L^{2}(\mathds{P})}. \label{nlwpthtt4}
\end{alignat}
\endgroup
Indeed, the equations \hyperref[nlwpthtt1]{(\ref*{nlwpthtt1})}, \hyperref[nlwpthtt2]{(\ref*{nlwpthtt2})}, and \hyperref[nlwpthtt4]{(\ref*{nlwpthtt4})} correspond to \hyperref[rgtsflffflnv]{(\ref*{rgtsflffflnv})}, \hyperref[osagswtuffasihnv]{(\ref*{osagswtuffasihnv})}, and \hyperref[svpvompvv]{(\ref*{svpvompvv})}, respectively. As for \hyperref[nlwpthtt3]{(\ref*{nlwpthtt3})}, we note that, according to equation \hyperref[flffflpv]{(\ref*{flffflpv})} of \hyperref[thetsixcoraopv]{Corollary \ref*{thetsixcoraopv}}, the limit in \hyperref[nlwpthtt3]{(\ref*{nlwpthtt3})} equals 
\begin{equation}
- \tfrac{1}{2} \, \| a \|_{L^{2}(\mathds{P})}^{2} + \mathds{E}_{\mathds{P}}\Big[ \Big(\operatorname{div} \beta - 2 \, \big\langle \beta , \nabla \Psi \big\rangle\Big)\big(X(t_{0})\big) \Big].    
\end{equation}
Therefore, in view of the right-hand side of \hyperref[nlwpthtt3]{(\ref*{nlwpthtt3})}, we have to show the identity
\begin{equation} \label{telimitthetthretv}
\mathds{E}_{\mathds{P}}\Big[ \Big(\operatorname{div} \beta -  \big\langle \beta \, , \, 2 \, \nabla \Psi \big\rangle\Big)\big(X(t_{0})\big) \Big] = - \langle a,b \rangle_{L^{2}(\mathds{P})}.
\end{equation}
In order to do this, we write the left-hand side of \hyperref[telimitthetthretv]{(\ref*{telimitthetthretv})} as
\begin{equation} \label{dftelimitthetthretv}
\int_{\mathds{R}^{n}} \Big( \operatorname{div} \beta(x) - \big\langle \beta(x) \, , \, 2 \, \nabla \Psi(x) \big\rangle \Big) \, p(t_{0},x) \, \textnormal{d}x.
\end{equation}
Using --- for the first time, and only in order to show the identity \hyperref[telimitthetthretv]{(\ref*{telimitthetthretv})} --- integration by parts, and the fact that the perturbation $\beta$ is assumed to be smooth and compactly supported, we see that the expression \hyperref[dftelimitthetthretv]{(\ref*{dftelimitthetthretv})} becomes
\begin{equation}
- \int_{\mathds{R}^{n}} \Big\langle \beta(x) \, , \, \nabla \log p(t_{0},x) + 2 \, \nabla \Psi(x) \Big\rangle \, p(t_{0},x) \, \textnormal{d}x,
\end{equation}
which is the same as $- \big\langle \beta(X(t_{0})), \nabla \log \ell(t_{0},X(t_{0})) \big\rangle_{L^{2}(\mathds{P})} = - \langle b, a\rangle_{L^{2}(\mathds{P})}$. 

\smallskip

The limiting identities \hyperref[nlwpthtt1]{(\ref*{nlwpthtt1})} -- \hyperref[nlwpthtt4]{(\ref*{nlwpthtt4})} now imply the assertions of \hyperref[thetthre]{Theorem \ref*{thetthre}}. 
\end{proof}

\medskip

The following \hyperref[thetsixcor]{Propositions \ref*{thetsixcor}} and \hyperref[thetthretvcor]{\ref*{thetthretvcor}} are trajectorial versions of \hyperref[thetsixcorao]{Corollaries \ref*{thetsixcorao}} and \hyperref[thetsixcoraopv]{\ref*{thetsixcoraopv}}, respectively. They compute the rate of temporal change of relative entropy for the equation \hyperref[sdeids]{(\ref*{sdeids})} and for its perturbed version \hyperref[wpsdeids]{(\ref*{wpsdeids})}, respectively, in the more precise trajectorial manner of \hyperref[thetsix]{Theorems \ref*{thetsix}}, \hyperref[thetthretv]{\ref*{thetthretv}}. 

\begin{proposition}[\textsf{Trajectorial rate of relative entropy dissipation}] \label{thetsixcor} Under the \textnormal{\hyperref[sosaojkoia]{Assumptions \ref*{sosaojkoia}}}, we let $t_{0} \in \mathds{R}_{+} \setminus N$ and $T > t_{0}$. Then the relative entropy process \textnormal{\hyperref[stlrpd]{(\ref*{stlrpd})}} satisfies the trajectorial relation
\begin{equation} \label{thetsixcorfe} 
\lim_{s \uparrow T-t_{0}} \, \frac{\mathds{E}_{\mathds{P}}\Big[ \log \ell\big(t_{0},X(t_{0})\big) \ \big\vert \ \mathcal{G}(T-s) \Big] - \log \ell \big( T-s,X(T-s)\big)}{T-t_{0}-s} 
=  \frac{1}{2} \frac{\big\vert \nabla \ell\big(t_{0},X(t_{0})\big) \big\vert^{2}}{\ell\big(t_{0},X(t_{0})\big)^{2}},
\end{equation}
where the limit exists in $L^{1}(\mathds{P})$.
\end{proposition}

\begin{remark} The limiting assertion \hyperref[thetsixcorfe]{(\ref*{thetsixcorfe})} of \hyperref[thetsixcor]{Proposition \ref*{thetsixcor}} is the conditional trajectorial version of the generalized de Bruijn identity \hyperref[flfffl]{(\ref*{flfffl})}.
\end{remark}

\begin{proof}[Proof of \texorpdfstring{\hyperref[thetsixcor]{Proposition \ref*{thetsixcor}}}{} from \texorpdfstring{\hyperref[thetsix]{Theorem \ref*{thetsix}}}{}:] Let $t_{0} \in \mathds{R}_{+} \setminus N$, i.e., so that the right-sided limiting assertion \hyperref[rgtsflffflnv]{(\ref*{rgtsflffflnv})} is valid, and select $T > t_{0}$. The martingale property of the process in \hyperref[dollafipim]{(\ref*{dollafipim})} allows us to write the numerator in \hyperref[thetsixcorfe]{(\ref*{thetsixcorfe})} as
\begin{equation}
\mathds{E}_{\mathds{P}}\Big[ F(t_{0})-F(T-s) \ \big\vert \ \mathcal{G}(T-s) \Big], \qquad 0 \leqslant s \leqslant T-t_{0}
\end{equation}
in the notation of \hyperref[fispdaftr]{(\ref*{fispdaftr})}, which expresses the process $(F(T-s))_{0 \leqslant s \leqslant T}$ as the primitive of 
\begin{equation} \label{choofbfi}
B(u) = \frac{1}{2} \frac{\big\vert \nabla \ell\big(T-u,X(T-u)\big) \big\vert^{2}}{\ell\big(T-u,X(T-u)\big)^{2}} \, , \qquad 0 \leqslant u \leqslant T.
\end{equation}
By analogy with the derivation of \hyperref[flfffl]{(\ref*{flfffl})} from \hyperref[stewrtpefitre]{(\ref*{stewrtpefitre})}, where we calculated real-valued expectations, we rely on the Lebesgue differentiation theorem to obtain the corresponding result \hyperref[thetsixcorfe]{(\ref*{thetsixcorfe})} for conditional expectations. Using the left-continuity of the backwards filtration $(\mathcal{G}(T-s))_{0 \leqslant s \leqslant T}$, we can invoke the measure-theoretic result in \hyperref[probamtr]{Proposition \ref*{probamtr}} of \hyperref[apsecamtr]{Appendix \ref*{apsecamtr}}, with the choice of the process $B$ as in \hyperref[choofbfi]{(\ref*{choofbfi})} and $C \equiv 0$. This establishes the claim \hyperref[thetsixcorfe]{(\ref*{thetsixcorfe})}.
\end{proof}

\begin{proposition}[\textsf{Trajectorial rate of relative entropy dissipation under perturbations}] \label{thetthretvcor} Under the \textnormal{\hyperref[sosaojkoia]{Assumptions \ref*{sosaojkoia}}}, we let $t_{0} \in \mathds{R}_{+} \setminus N$ and $T > t_{0}$. Then the perturbed relative entropy process \textnormal{\hyperref[prerepitufdllpitufd]{(\ref*{prerepitufdllpitufd})}} satisfies the trajectorial relations
\begingroup
\addtolength{\jot}{0.7em}
\begin{alignat}{1}
&\begin{aligned} \label{thetsixcorfes}
& \lim_{s \uparrow T-t_{0}} \, \frac{\mathds{E}_{\mathds{P}^{\beta}}\Big[ \log \ell^{\beta}\big(t_{0},X(t_{0})\big) \ \big\vert \ \mathcal{G}(T-s) \Big] - \log \ell^{\beta} \big( T-s,X(T-s)\big)}{T-t_{0}-s}  \\
& \qquad = \frac{1}{2}\frac{\big\vert \nabla \ell\big(t_{0},X(t_{0})\big) \big\vert^{2}}{\ell\big(t_{0},X(t_{0})\big)^{2}} - \operatorname{div} \beta\big(X(t_{0})\big)  
+ \Big\langle \beta\big(X(t_{0})\big) \, , \, 2 \, \nabla \Psi\big(X(t_{0})\big) \Big\rangle, 
\end{aligned} \\[11pt]
&\begin{aligned} \label{thetsixcorfessl}
& \lim_{s \uparrow T-t_{0}} \, \frac{\mathds{E}_{\mathds{P}}\Big[ \log \ell^{\beta}\big(t_{0},X(t_{0})\big) \ \big\vert \ \mathcal{G}(T-s) \Big] - \log \ell^{\beta} \big( T-s,X(T-s)\big)}{T-t_{0}-s} \\
& \qquad = \frac{1}{2}\frac{\big\vert \nabla \ell\big(t_{0},X(t_{0})\big) \big\vert^{2}}{\ell\big(t_{0},X(t_{0})\big)^{2}} -  \operatorname{div} \beta\big(X(t_{0})\big)  
- \Big\langle \beta\big(X(t_{0})\big) \, , \, \nabla  \log p \big(t_{0},X(t_{0})\big) \Big\rangle , 
\end{aligned} \\[11pt]
&\begin{aligned} \label{titmlwhtcfs}
& \lim_{s \uparrow T-t_{0}} \, \frac{\log \ell^{\beta}\big(T-s,X(T-s)\big) - \log \ell\big(T-s,X(T-s)\big)}{T-t_{0}-s}  \\
& \qquad = \operatorname{div} \beta\big(X(t_{0})\big) + \Big\langle \beta\big(X(t_{0})\big) \, , \, \nabla  \log p\big(t_{0},X(t_{0})\big) \Big\rangle, 
\end{aligned}
\end{alignat}
\endgroup
where the limits in \textnormal{\hyperref[thetsixcorfes]{(\ref*{thetsixcorfes})}} -- \textnormal{\hyperref[titmlwhtcfs]{(\ref*{titmlwhtcfs})}} exist in both $L^{1}(\mathds{P})$ and $L^{1}(\mathds{P}^{\beta})$.
\end{proposition}

\begin{remark} It is noteworthy that the three limiting expressions in \hyperref[thetsixcorfes]{(\ref*{thetsixcorfes})}, \hyperref[thetsixcorfessl]{(\ref*{thetsixcorfessl})} and \hyperref[titmlwhtcfs]{(\ref*{titmlwhtcfs})} are quite different from each other. The first limiting assertion \hyperref[thetsixcorfes]{(\ref*{thetsixcorfes})} of \hyperref[thetthretvcor]{Proposition \ref*{thetthretvcor}} is the conditional trajectorial version of the perturbed de Bruijn identity \hyperref[flffflpv]{(\ref*{flffflpv})}. We also note that in fact the third limiting assertion \hyperref[titmlwhtcfs]{(\ref*{titmlwhtcfs})} is valid for all $t_{0} > 0$.
\end{remark}

\begin{proof}[Proof of \texorpdfstring{\hyperref[thetsixcorfes]{\textnormal{(\ref*{thetsixcorfes})}}}{} from \texorpdfstring{\hyperref[thetthretv]{Theorem \ref*{thetthretv}}}{}:] Let $t_{0} \in \mathds{R}_{+} \setminus N$, i.e., so that the right-sided limiting assertion \hyperref[rgtsflffflnv]{(\ref*{rgtsflffflnv})} is valid, and select $T > t_{0}$. In \hyperref[tciwsolafophvs]{(\ref*{tciwsolafophvs})} from \hyperref[thetsixcoraopv]{Corollary \ref*{thetsixcoraopv}} of \hyperref[thetthretv]{Theorem \ref*{thetthretv}} we have seen that the limits in \hyperref[rgtsflffflnv]{(\ref*{rgtsflffflnv})} and \hyperref[flffflpv]{(\ref*{flffflpv})} have the same exceptional sets, hence the limiting identity \hyperref[flffflpv]{(\ref*{flffflpv})} also holds. Now, for such $t_{0} \in \mathds{R}_{+} \setminus N$, we show the limiting assertion \hyperref[thetsixcorfes]{(\ref*{thetsixcorfes})} in the same way as the assertion \hyperref[thetsixcorfe]{(\ref*{thetsixcorfe})} in the proof of \hyperref[thetsixcor]{Proposition \ref*{thetsixcor}} above. Indeed, this time we invoke the $\mathds{P}^{\beta}$-martingale property of the process in \hyperref[dollafipimps]{(\ref*{dollafipimps})}, and write the numerator on the first line of \hyperref[thetsixcorfes]{(\ref*{thetsixcorfes})} as $\mathds{E}_{\mathds{P}^{\beta}}\big[ F^{\beta}(t_{0})-F^{\beta}(T-s) \ \big\vert \ \mathcal{G}(T-s) \big]$, $0 \leqslant s \leqslant T-t_{0}$, in the notation of \hyperref[fispdaftrps]{(\ref*{fispdaftrps})}, which expresses the process $(F^{\beta}(T-s))_{0 \leqslant s \leqslant T-t_{0}}$ as the primitive of $(B(u)+C(u))_{0 \leqslant s \leqslant T-t_{0}}$, with
\begin{equation} \label{clno}
B(u) = \frac{1}{2}\frac{\big\vert \nabla \ell^{\beta}\big(T-u,X(T-u)\big) \big\vert^{2}}{\ell^{\beta}\big(T-u,X(T-u)\big)^{2}}, \qquad C(u) = \Big( \big\langle \beta \, , \, 2 \, \nabla \Psi \big\rangle - \operatorname{div} \beta \Big) \big(X(T-u)\big).
\end{equation}
\smallskip

Applying \hyperref[probamtr]{Proposition \ref*{probamtr}} of \hyperref[apsecamtr]{Appendix \ref*{apsecamtr}} in this situation proves the limiting identity \hyperref[thetsixcorfes]{(\ref*{thetsixcorfes})} in $L^{1}(\mathds{P}^{\beta})$. As we shall see in \hyperref[hctclwittmeitpocasotpv]{Lemma \ref*{hctclwittmeitpocasotpv}} of \hyperref[subsomusefullem]{Subsection \ref*{subsomusefullem}} below, the probability measures $\mathds{P}$ and $\mathds{P}^{\beta}$ are equivalent, and the mutual Radon--Nikod\'{y}m derivatives $\frac{\textnormal{d}\mathds{P}^{\beta}}{\textnormal{d}\mathds{P}}$ and $\frac{\textnormal{d}\mathds{P}}{\textnormal{d}\mathds{P}^{\beta}}$ are bounded on the $\sigma$-algebra $\mathcal{F}(T) = \mathcal{G}(0)$ (recall, in this vein, the claims of \hyperref[ilpdepvhfv]{(\ref*{ilpdepvhfv})}). Hence, convergence in $L^{1}(\mathds{P})$ is equivalent to convergence in $L^{1}(\mathds{P}^{\beta})$. This readily proves assertion \hyperref[thetsixcorfes]{(\ref*{thetsixcorfes})}.

\smallskip

The proofs of the limiting assertions \hyperref[thetsixcorfessl]{(\ref*{thetsixcorfessl})} and \hyperref[titmlwhtcfs]{(\ref*{titmlwhtcfs})} are postponed to \hyperref[nspofthetthretvcor]{Subsection \ref*{nspofthetthretvcor}}.
\end{proof}


\subsection{A trajectorial proof of the HWI inequality} \label{ramifications}


The aim of this section is to provide a proof of the celebrated \textit{HWI inequality} due to Otto and Villani \cite{OV00} by applying trajectorial arguments similar to those in \hyperref[thetsix]{Theorem \ref*{thetsix}}, in fact quite easier. We thus obtain an intuitive geometric picture and deduce the sharpened form of the HWI inequality; see also \cite{CE02}, \cite{OV00} and \cite[p.\ 650]{Vil09}).
 
\smallskip

The goal is to compare the relative entropies $H(P_{0} \, \vert \, \mathrm{Q})$ and $H(P_{1} \, \vert \, \mathrm{Q})$ for arbitrary probability measures $P_{0}, P_{1} \in \mathscr{P}_{2}(\mathds{R}^{n})$. Using \textit{Brenier's theorem} \cite{Bre91}, we first define the constant speed geodesic $(P_{t})_{0 \leqslant t \leqslant 1}$ between $P_{0}$ and $P_{1}$ with respect to the Wasserstein distance $W_{2}$ (details are given below). We remark, that we have chosen the subscript notation for $P_{t}$ in order to avoid confusion with the probability measure $P(t)$ from our \hyperref[snaas]{Section \ref*{snaas}} here. With $p_{t}(\, \cdot \,)$ the density function of the probability measure $P_{t}$, we define the likelihood ratio function 
\begin{equation} \label{nlrfitramfcc}
\ell_{t}(x) \vcentcolon = \frac{p_{t}(x)}{q(x)}, \qquad (t,x) \in [0,1] \times \mathds{R}^{n}.
\end{equation}

\smallskip

We shall investigate the behavior of the relative entropy function $t \mapsto f(t) \vcentcolon = H( P_{t} \, \vert \, \mathrm{Q})$ along the constant speed geodesic $(P_{t})_{0 \leqslant t \leqslant 1}$ by estimating two quantities: First, we want a lower bound on the first derivative $f'(0^{+})$. Secondly, we want a lower bound on the second derivative $(f''(t))_{0 \leqslant t \leqslant 1}$. It should be geometrically obvious (and will be spelled out in the proof of \hyperref[hwiai]{Theorem \ref*{hwiai}} below) that information on these two lower bounds leads to a lower bound on $f(1) - f(0)$. The latter is the content of the HWI inequality. As regards the second derivative $(f''(t))_{0 \leqslant t \leqslant 1}$, we shall rely on a fundamental result on displacement convexity due to McCann \cite{McC97} and have no novel contribution. As regards $f'(0^{+})$, however, we shall obtain a sharp estimate for this quantity by applying a trajectorial reasoning similar to that deployed in the proof of \hyperref[thetsix]{Theorem \ref*{thetsix}}. 

\smallskip

We will define an $\mathds{R}^{n}$-valued stochastic process $(X_{t})_{0 \leqslant t \leqslant 1}$, with marginal distributions $(P_{t})_{0 \leqslant t \leqslant 1}$ moving along straight lines in $\mathds{R}^{n}$, and calculate the relevant quantities of this finite variation process along every trajectory, by analogy with the proof of \hyperref[thetsix]{Theorem \ref*{thetsix}}. This gives the desired bound (and actually an equality) for the derivative $f'(0^{+})$. 

\medskip

We now cast these ideas into formal terms. The first step is to calculate the decay of the relative entropy function $t \mapsto H( P_{t} \, \vert \, \mathrm{Q})$ along the ``straight line'' $(P_{t})_{0 \leqslant t \leqslant 1}$ joining the elements $P_{0}$ and $P_{1}$ in $\mathscr{P}_{2}(\mathds{R}^{n})$. To this end, we impose temporarily the following strong regularity conditions. In the proof of \hyperref[hwiai]{Theorem \ref*{hwiai}} we shall see that these will not restrict the generality of the argument.

\begin{assumptions}[\textsf{Regularity assumptions of \texorpdfstring{\hyperref[hlotl]{Lemma \ref*{hlotl}}}{}}] \label{hwisosaojkoia} We impose that $P_{0}$ and $P_{1}$ are probability measures in $\mathscr{P}_{2}(\mathds{R}^{n})$ with smooth densities, which are compactly supported and strictly positive in the interior of their respective supports. Hence there exists a map $\gamma \colon \mathds{R}^{n} \rightarrow \mathds{R}^{n}$ of the form $\gamma(x) = \nabla(G(x) - \vert x \vert^{2}/2)$ for some convex function $G \colon \mathds{R}^{n} \rightarrow \mathds{R}$, uniquely defined on and supported by the support of $P_{0}$, and smooth in the interior of this set, such that $\gamma$ induces the optimal quadratic Wasserstein transport from $P_{0}$ to $P_{1}$ via
\begin{equation} \label{hlotlse}
T_{t}^{\gamma}(x) \vcentcolon = x + t \, \gamma(x) = (1-t) \, x + t \, \nabla G(x) \qquad \textnormal{and} \qquad P_{t} \vcentcolon = (T_{t}^{\gamma})_{\#}(P_{0}) = P_{0} \circ (T_{t}^{\gamma})^{-1}
\end{equation}
for $0 \leqslant t \leqslant 1$; to wit, the curve $(P_{t})_{0 \leqslant t \leqslant 1}$ is the displacement interpolation (constant speed geodesic) between $P_{0}$ and $P_{1}$, and we have along it the linear growth of the quadratic Wasserstein distance
\begin{equation} \label{otmitwsnp}
W_{2}(P_{0},P_{t}) = t \, \sqrt{\int_{\mathds{R}^{n}} \vert x - \nabla G(x) \vert^{2} \, \textnormal{d} P_{0}(x)} = t \, \| \gamma \|_{L^{2}(P_{0})}, \qquad 0 \leqslant t \leqslant 1.
\end{equation}
For existence and uniqueness of the optimal transport map $\gamma$ we refer to \cite[Theorem 2.12]{Vil03}, and for its smoothness to \cite[Theorem 4.14]{Vil03} as well as \cite[Remarks 4.15]{Vil03}. These results are known collectively under the rubric of Brenier's theorem \cite{Bre91}. 
\end{assumptions}

Next we compute the slope of the function $t \mapsto H( P_{t} \, \vert \, \mathrm{Q})$ along the straight line $(P_{t})_{0 \leqslant t \leqslant 1}$.

\begin{lemma} \label{hlotl} Under the \textnormal{\hyperref[hwisosaojkoia]{Assumptions \ref*{hwisosaojkoia}}}, let $X_{0} \colon S \rightarrow \mathds{R}^{n}$ be a random variable with probability distribution $P_{0} \in \mathscr{P}_{2}(\mathds{R}^{n})$, defined on some probability space $(S,\mathcal{S},\nu)$. Then we have
\begin{equation} \label{hlotlte}
\lim_{t \downarrow 0} \frac{H(P_{t} \, \vert \, \mathrm{Q}) - H(P_{0} \, \vert \, \mathrm{Q})}{t} = \big\langle \nabla \log \ell_{0}(X_{0}) \, , \, \gamma(X_{0}) \big\rangle_{L^{2}(\nu)}.
\end{equation}
\end{lemma}

We relegate to \hyperref[polhlotl]{Appendix \ref*{polhlotl}} the proof of \hyperref[hlotl]{Lemma \ref*{hlotl}}, which follows a similar (but considerably simpler) trajectorial line of reasoning as the proof of \hyperref[thetthre]{Theorem \ref*{thetthre}}. Combining \hyperref[hlotl]{Lemma \ref*{hlotl}} with well-known arguments, in particular, with a fundamental result on displacement convexity due to McCann \cite{McC97}, we derive now the HWI inequality of Otto and Villani \cite{OV00}. 

\begin{theorem}[\textsf{HWI inequality \cite{OV00}}] \label{hwiai} We fix $P_{0}, P_{1} \in \mathscr{P}_{2}(\mathds{R}^{n})$ and assume that the relative entropy $H( P_{1} \, \vert \, \mathrm{Q})$ is finite. We suppose in addition that the potential $\Psi \in \mathcal{C}^{\infty}(\mathds{R}^{n};[0,\infty))$ satisfies a curvature lower bound 
\begin{equation} \label{sndcbe}
\textnormal{Hess}(\Psi) \geqslant \kappa \, I_{n},
\end{equation}
for some $\kappa \in \mathds{R}$. Then we have
\begin{equation} \label{sivothwii}
H( P_{0} \, \vert \, \mathrm{Q} ) - H( P_{1} \, \vert \, \mathrm{Q} ) \leqslant - \big\langle \nabla \log \ell_{0}(X_{0}) \, , \, \gamma(X_{0}) \big\rangle_{L^{2}(\nu)} - \tfrac{\kappa}{2} \, W_{2}^{2}(P_{0},P_{1}),
\end{equation}
where the likelihood ratio function $\ell_{0}$, the random variable $X_{0}$, the optimal transport map $\gamma$, and the probability measure $\nu$, are as in \textnormal{\hyperref[hlotl]{Lemma \ref*{hlotl}}}.
\end{theorem}

We stress that \hyperref[hwiai]{Theorem \ref*{hwiai}} does not require the measure $\mathrm{Q}$ with density $q(x) = \mathrm{e}^{-2 \Psi(x)}$ to be a finite measure in the formulation of the HWI inequality \hyperref[sivothwii]{(\ref*{sivothwii})}. 

\smallskip

On the strength of the Cauchy--Schwarz inequality, we have
\begin{equation}
- \big\langle \nabla \log \ell_{0}(X_{0}) \, , \, \gamma(X_{0}) \big\rangle_{L^{2}(\nu)} \leqslant \| \nabla \log \ell_{0}(X_{0}) \|_{L^{2}(\nu)} \ \| \gamma(X_{0}) \|_{L^{2}(\nu)},
\end{equation}
with equality if and only if the functions $\nabla \log \ell_{0}(\, \cdot \,)$ and $\gamma(\, \cdot \,)$ are negatively collinear. The relative Fisher information of $P_{0}$ with respect to $\mathrm{Q}$ equals 
\begin{equation} 
I( P_{0} \, \vert \, \mathrm{Q}) = \mathds{E}_{\nu}\Big[ \vert \nabla \log \ell_{0}(X_{0}) \vert^{2} \Big] = \| \nabla \log \ell_{0}(X_{0}) \|_{L^{2}(\nu)}^{2},
\end{equation}
and by Brenier's theorem \cite[Theorem 2.12]{Vil03} we deduce
\begin{equation}
\| \gamma(X_{0}) \|_{L^{2}(\nu)} = W_{2}(P_{0},P_{1})
\end{equation}
as in \hyperref[otmitwsnp]{(\ref*{otmitwsnp})}, along with the inequality
\begin{equation} \label{csfhwi} 
- \big\langle \nabla \log \ell_{0}(X_{0}) \, , \, \gamma(X_{0}) \big\rangle_{L^{2}(\nu)} \leqslant \sqrt{I( P_{0} \,  \vert \, \mathrm{Q} )} \ W_{2}(P_{0},P_{1}).
\end{equation}
Inserting \hyperref[csfhwi]{(\ref*{csfhwi})} into \hyperref[sivothwii]{(\ref*{sivothwii})} we obtain the usual form of the HWI inequality 
\begin{equation} \label{hwiast}
H( P_{0} \, \vert \, \mathrm{Q} ) - H( P_{1} \, \vert \, \mathrm{Q} ) \leqslant W_{2}(P_{0},P_{1}) \ \sqrt{I( P_{0} \,  \vert \, \mathrm{Q} )} - \tfrac{\kappa}{2} \, W_{2}^{2}(P_{0},P_{1}).
\end{equation}
When there is a non-trivial angle between $- \nabla \log \ell_{0}(X_{0})$ and $\gamma(X_{0})$ in $L^{2}(\nu)$, the inequality \hyperref[sivothwii]{(\ref*{sivothwii})} gives a sharper bound than \hyperref[hwiast]{(\ref*{hwiast})}. We refer to the original paper \cite{OV00}, as well as to \cite{CE02}, \cite[Chapter 5]{Vil03}, \cite[p.\ 650]{Vil09} and the recent papers \cite{GLRT20,KMS20} for detailed discussions of the HWI inequality in several contexts. For a good survey on transport inequalities, see \cite{GL10}. 

\begin{proof}[Proof of \texorpdfstring{\hyperref[hwiai]{Theorem \ref*{hwiai}}}] As elaborated in \cite[Section 9.4]{Vil03} we may assume without loss of generality that $P_{0}$ and $P_{1}$ satisfy the strong regularity \hyperref[hwisosaojkoia]{Assumptions \ref*{hwisosaojkoia}}, guaranteeing existence and smoothness of the optimal transport map $\gamma$.

We consider now the relative entropy with respect to $\mathrm{Q}$ along the constant-speed geodesic $(P_{t})_{0 \leqslant t \leqslant 1}$, namely, the function $f(t) \vcentcolon = H( P_{t} \, \vert \, \mathrm{Q})$, for $0 \leqslant t \leqslant 1$. The displacement convexity results of McCann \cite{McC97} imply 
\begin{equation} \label{dcromcc}
f''(t) \geqslant \kappa \, W_{2}^{2}(P_{0},P_{1}), \qquad 0 \leqslant t \leqslant 1.
\end{equation}

Indeed, under the condition \hyperref[sndcbe]{(\ref*{sndcbe})}, the potential $\Psi$ is $\kappa$-uniformly convex. Consequently, by items (i) and (ii) of \cite[Theorem 5.15]{Vil03}, the internal and potential energies
\begin{equation}
g(t) \vcentcolon = \int_{\mathds{R}^{n}} p_{t}(x) \log p_{t}(x) \, \textnormal{d}x, \qquad h(t) \vcentcolon = 2 \int_{\mathds{R}^{n}} \Psi(x) \, p_{t}(x) \, \textnormal{d}x, \qquad 0 \leqslant t \leqslant 1,
\end{equation}
are displacement convex and $\kappa$-uniformly displacement convex, respectively; i.e.,
\begin{equation} 
g''(t) \geqslant 0, \qquad h''(t) \geqslant \kappa \, W_{2}^{2}(P_{0},P_{1}), \qquad 0 \leqslant t \leqslant 1.
\end{equation}
As we have $f = g + h$, we conclude that the relative entropy function $f$ is $\kappa$-uniformly displacement convex, i.e., its second derivative satisfies \hyperref[dcromcc]{(\ref*{dcromcc})}. We appeal now to \hyperref[hlotl]{Lemma \ref*{hlotl}}, according to which 
\begin{equation} \label{irf}
f'(0^{+}) = \lim_{t \downarrow 0} \, \frac{f(t)-f(0)}{t} = \big\langle \nabla \log \ell_{0}(X_{0}) \, , \, \gamma(X_{0}) \big\rangle_{L^{2}(\nu)}.
\end{equation}
In conjunction with \hyperref[dcromcc]{(\ref*{dcromcc})} and \hyperref[irf]{(\ref*{irf})}, the Taylor formula $f(1) = f(0) + f'(0^{+}) + \int_{0}^{1} (1-t) f''(t) \, \textnormal{d}t$ now yields \hyperref[sivothwii]{(\ref*{sivothwii})}.
\end{proof}


\section{Details and proofs} \label{ssgrodwudm}


In this section we complete the proofs of \hyperref[thetsixcoraopv]{Corollary \ref*{thetsixcoraopv}} and \hyperref[thetthretvcor]{Proposition \ref*{thetthretvcor}}, and provide the proofs of our main results, \hyperref[thetsix]{Theorems \ref*{thetsix}} and \hyperref[thetthretv]{\ref*{thetthretv}}. What we have to do in order to prove these latter theorems is to apply It\^{o}'s formula so as to calculate the dynamics, i.e., the stochastic differentials, of the ``pure'' and ``perturbed'' relative entropy processes of \hyperref[stlrpd]{(\ref*{stlrpd})} and \hyperref[prerepitufdllpitufd]{(\ref*{prerepitufdllpitufd})} under the measures $\mathds{P}$ and $\mathds{P}^{\beta}$, respectively. As already discussed, we shall do this in the backward direction of time.


\subsection{The proof of \texorpdfstring{\hyperref[thetsix]{Theorem \ref*{thetsix}}}{Theorem 4.1}} \label{cotpottpf}


We start by calculating the stochastic differential of the time-reversed canonical coordinate process $(X(T-s))_{0 \leqslant s \leqslant T}$ under $\mathds{P}$, a well-known and classical theme; see e.g.\ \cite{Foe85,Foe86}, \cite{HP86}, \cite{Mey94}, \cite{Nel01}, and \cite{Par86}. The reader may consult Appendix G of \cite{KST20} for an extensive presentation of the relevant facts regarding the theory of time reversal for diffusion processes. The idea of time reversal goes back to Boltzmann \cite{Bol96,Bol98a,Bol98b} and Schr\"odinger \cite{Sch31,Sch32}, as well as Kolmogorov \cite{Kol37}. In fact, the relation between time reversal of a Brownian motion and the quadratic Wasserstein distance may \textit{in nuce} be traced back to an insight of Bachelier in his thesis \cite{Bac00,Bac06} from 1900. This theme is discussed in Appendix A of \cite{KST20}.

\smallskip

Recall that the probability measure $\mathds{P}$ was defined on path space $\Omega = \mathcal{C}(\mathds{R}_{+};\mathds{R}^{n})$ so that the canonical coordinate process $(X(t,\omega))_{t \geqslant 0} = (\omega(t))_{t \geqslant 0}$ satisfies the stochastic differential equation \hyperref[sdeids]{(\ref*{sdeids})} with initial probability distribution $P(0)$ for $X(0)$ under $\mathds{P}$. In other words, the process
\begin{equation} \label{dobmwopsrss}
W(t) = X(t) - X(0) + \int_{0}^{t} \nabla \Psi\big(X(u)\big) \, \textnormal{d}u, \qquad t \geqslant 0
\end{equation}
is a Brownian motion of the forward filtration $(\mathcal{F}(t))_{t \geqslant 0}$ under the probability measure $\mathds{P}$. Passing to the reverse direction of time, the following classical result is well known to hold under the present assumptions. For proof and references we refer to Theorems G.2 and G.5 of Appendix G in \cite{KST20}.

\begin{proposition} \label{ptra} Under \textnormal{\hyperref[sosaojkoia]{Assumptions \ref*{sosaojkoia}}}, fix $T > 0$. The process 
\begin{equation} \label{dotowptmtpbm}
\overline{W}^\mathds{P}(T-s) \vcentcolon = W(T-s) - W(T) - \int_{0}^{s} \nabla \log p\big(T-u,X(T-u)\big) \, \textnormal{d}u \, , \qquad 0 \leqslant s \leqslant T
\end{equation}
is a Brownian motion of the backwards filtration $(\mathcal{G}(T-s))_{0 \leqslant s \leqslant T}$ under the probability measure $\mathds{P}$. Moreover, the time-reversed canonical coordinate process $(X(T-s))_{0 \leqslant s \leqslant T}$ satisfies the stochastic differential equation
\begingroup
\addtolength{\jot}{0.7em}
\begin{align}
\textnormal{d} X(T-s) &= \Big( \nabla \log p\big(T-s,X(T-s)\big) + \nabla \Psi\big(X(T-s)\big) \Big) \, \textnormal{d}s + \textnormal{d}\overline{W}^{\mathds{P}}(T-s) \label{poortdfttrpp} \\    
 &= \Big( \nabla \log \ell\big(T-s,X(T-s)\big) - \nabla \Psi\big(X(T-s)\big) \Big) \, \textnormal{d}s + \textnormal{d}\overline{W}^{\mathds{P}}(T-s), \label{oortdfttrpp}
\end{align}
\endgroup
for $0 \leqslant s \leqslant T$, with respect to the backwards filtration $(\mathcal{G}(T-s))_{0 \leqslant s \leqslant T}$. 
\end{proposition}

The following result computes the forward dynamics of the likelihood ratio process $(\ell(t,X(t)))_{t \geqslant 0}$ of \hyperref[rndlr]{(\ref*{rndlr})} and compares it with the stochastic differential of the time-reversed likelihood ratio process
\begin{equation} \label{ttrreplik}
\ell\big(T-s,X(T-s)\big) = \frac{p\big(T-s,X(T-s)\big)}{q\big(X(T-s)\big)}  \, , \qquad 0 \leqslant s \leqslant T,
\end{equation} 
as well as its logarithmic differential.

\begin{proposition} \label{ppropp} Under the \textnormal{\hyperref[sosaojkoia]{Assumptions \ref*{sosaojkoia}}}, the likelihood ratio process \textnormal{\hyperref[rndlr]{(\ref*{rndlr})}} is a continuous semimartingale with respect to the forward filtration $(\mathcal{F}(t))_{t \geqslant 0}$ and satisfies, for $t \geqslant 0$, the stochastic differential equation
\begingroup
\addtolength{\jot}{0.7em}
\begin{equation} \label{updotlrpfu}
\textnormal{d} \ell\big(t,X(t)\big) 
= \Big\langle \nabla \ell\big(t,X(t)\big) \, , \, \textnormal{d}W(t) \Big\rangle
+ \Big( \Delta \ell\big(t,X(t)\big) 
- \Big\langle \nabla \ell\big(t,X(t)\big) \, , \,   2 \, \nabla \Psi\big(X(t)\big)  \Big\rangle \Big) \, \textnormal{d}t.  
\end{equation} 
\endgroup
Furthermore, the time-reversed likelihood ratio process \textnormal{\hyperref[ttrreplik]{(\ref*{ttrreplik})}} is a continuous semimartingale with respect to the backwards filtration $(\mathcal{G}(T-s))_{0 \leqslant s \leqslant T}$ and satisfies, for $0 \leqslant s \leqslant T$, the stochastic differential equations
\begingroup
\addtolength{\jot}{0.7em}
\begin{align}
\textnormal{d} \ell\big(T-s,X(T-s)\big) 
&= \Big\langle \nabla \ell\big(T-s,X(T-s)\big) \, , \, \textnormal{d}\overline{W}^\mathds{P}(T-s) \Big\rangle  + \frac{\big\vert \nabla \ell\big(T-s,X(T-s)\big) \big\vert^{2}}{\ell\big(T-s,X(T-s)\big)} \, \textnormal{d}s, \label{qubmcfrsdeftpmpaei} \\
\frac{\textnormal{d} \ell\big(T-s,X(T-s)\big)}{\ell\big(T-s,X(T-s)\big)} 
&= \Bigg\langle \frac{\nabla \ell\big(T-s,X(T-s)\big)}{\ell\big(T-s,X(T-s)\big)} \, , \, \textnormal{d}\overline{W}^\mathds{P}(T-s) \Bigg\rangle  + \frac{\big\vert \nabla \ell\big(T-s,X(T-s)\big) \big\vert^{2}}{\ell\big(T-s,X(T-s)\big)^{2}} \, \textnormal{d}s, \label{qubmcfrsdeftpmpaeie} \\
\textnormal{d}\log \ell\big(T-s,X(T-s)\big)
&= \Bigg\langle \frac{\nabla \ell\big(T-s,X(T-s)\big)}{\ell\big(T-s,X(T-s)\big)} \, , \, \textnormal{d}\overline{W}^\mathds{P}(T-s) \Bigg\rangle  + \frac{1}{2} \frac{\big\vert \nabla \ell\big(T-s,X(T-s)\big) \big\vert^{2}}{\ell\big(T-s,X(T-s)\big)^{2}} \, \textnormal{d}s. \label{sdotpttrrep}
\end{align}
\endgroup
\begin{proof} We start with the following observation. Writing the Fokker--Planck equation \hyperref[fpeqnwfp]{(\ref*{fpeqnwfp})} as
\begin{equation} \label{pdeftupdpp}
\partial_{t} p(t,x) = \tfrac{1}{2} \Delta p(t,x) + \big\langle \nabla p(t,x) \, , \nabla \Psi(x)\big\rangle + p(t,x) \, \Delta \Psi(x), \qquad t > 0
\end{equation}
and substituting the expression
\begin{equation} \label{tpdfptxcbritf}
p(t,x) = \ell(t,x) \, q(x) = \ell(t,x) \, \mathrm{e}^{ - 2 \Psi(x)}, \qquad t \geqslant 0
\end{equation}
into this equation, we find that the likelihood ratio function $(t,x) \mapsto \ell(t,x)$ solves the \textit{backwards} Kolmogorov equation
\begin{equation} \label{fpdefflrf}
\partial_{t} \ell(t,x) = \tfrac{1}{2} \Delta \ell(t,x) - \big\langle \nabla \ell(t,x) \, , \nabla \Psi(x) \big\rangle, \qquad t > 0.
\end{equation}

\smallskip

Now we turn to the proofs of \hyperref[updotlrpfu]{(\ref*{updotlrpfu})} -- \hyperref[sdotpttrrep]{(\ref*{sdotpttrrep})}. By \hyperref[sosaojkoia]{Assumptions \ref*{sosaojkoia}}, the likelihood ratio function $(t,x) \mapsto \ell(t,x)$ is sufficiently smooth to allow an application of It\^{o}'s formula. Together with the Langevin--Smoluchowski dynamics \hyperref[sdeids]{(\ref*{sdeids})} and the backwards Kolmogorov equation \hyperref[pdeftupdpp]{(\ref*{pdeftupdpp})}, we obtain \hyperref[updotlrpfu]{(\ref*{updotlrpfu})} by direct calculation. A similar calculation, this time relying on the backwards dynamics \hyperref[oortdfttrpp]{(\ref*{oortdfttrpp})}, shows \hyperref[qubmcfrsdeftpmpaei]{(\ref*{qubmcfrsdeftpmpaei})}. Finally, the equations \hyperref[qubmcfrsdeftpmpaeie]{(\ref*{qubmcfrsdeftpmpaeie})} and \hyperref[sdotpttrrep]{(\ref*{sdotpttrrep})} follow from \hyperref[qubmcfrsdeftpmpaei]{(\ref*{qubmcfrsdeftpmpaei})} and It\^{o}'s formula.
\end{proof}
\end{proposition}

The crucial feature of the stochastic differentials \hyperref[updotlrpfu]{(\ref*{updotlrpfu})} -- \hyperref[sdotpttrrep]{(\ref*{sdotpttrrep})} is that, after passing to time reversal, the finite-variation term $\Delta \ell - \langle \nabla \ell \, ,  \, 2 \, \nabla \Psi \rangle$ in \hyperref[updotlrpfu]{(\ref*{updotlrpfu})}, involving the Laplacian $\Delta \ell$, gets replaced by a term involving only the likelihood ratio function $\ell$ and its gradient $\nabla \ell$. We owe this crucial insight to the work of Fontbona and Jourdain \cite{FJ16}; see Theorem 4.2 and Appendix E in \cite{KST20} for an extensive discussion and a proof of the Fontbona--Jourdain theorem.

\smallskip

For another application of time reversal in a similar context, see \cite{Leo17}.

\begin{proof}[\bfseries \upshape Proof of \texorpdfstring{\hyperref[thetsix]{Theorem \ref*{thetsix}}}{}] On a formal level, the expressions \hyperref[fispdaftr]{(\ref*{fispdaftr})}, \hyperref[erdollafipim]{(\ref*{erdollafipim})} are just integral versions of the It\^{o} differential \hyperref[sdotpttrrep]{(\ref*{sdotpttrrep})}. What remains to check is that the integrals in \hyperref[fispdaftr]{(\ref*{fispdaftr})} and \hyperref[erdollafipim]{(\ref*{erdollafipim})} indeed make rigorous sense and satisfy the claimed integrability conditions.

\smallskip

By condition \hyperref[tsaosaojko]{\ref*{tsaosaojko}} of \hyperref[sosaojkoia]{Assumptions \ref*{sosaojkoia}} the function $(t,x) \mapsto \nabla \log \ell(t,x)$ is continuous. Together with the continuity of the paths of the canonical coordinate process $(X(t))_{t \geqslant 0}$, this implies 
\begin{equation} \label{twtcotpotccptitttfitisifoas}
\int_{0}^{T-\varepsilon} \frac{\big\vert \nabla \ell\big(T-u,X(T-u)\big) \big\vert^{2}}{\ell\big(T-u,X(T-u)\big)^{2}} \, \textnormal{d}u < \infty, \qquad \mathds{P}\textnormal{-a.s.}
\end{equation}
for every $0 < \varepsilon \leqslant T$. On account of \hyperref[twtcotpotccptitttfitisifoas]{(\ref*{twtcotpotccptitttfitisifoas})}, the sequence of stopping times (with respect to the backwards filtration)
\begin{equation} \label{twtcdostptitttfitisifoas}
\tau_{n} \vcentcolon = \inf \Bigg\{ t \geqslant 0 \colon \, \int_{0}^{t} \frac{\big\vert \nabla \ell\big(T-u,X(T-u)\big) \big\vert^{2}}{\ell\big(T-u,X(T-u)\big)^{2}} \, \textnormal{d}u \, \geqslant n \,  \Bigg\} \wedge T, \qquad n \in \mathds{N}_{0}
\end{equation}
is non-decreasing and converges $\mathds{P}$-a.s.\ to $T$. Defining $M$ via \hyperref[dollafipim]{(\ref*{dollafipim})}, each stopped process $M^{\tau_{n}}$ is bounded in $L^{2}(\mathds{P})$ and satisfies the stopped version of \hyperref[erdollafipim]{(\ref*{erdollafipim})}, i.e.,
\begin{equation} 
M^{\tau_{n}}(T-s) = M\big(T-(s \wedge \tau_{n})\big) = \int_{0}^{s \wedge \tau_{n}} \Bigg\langle \frac{\nabla \ell\big(T-u,X(T-u)\big)}{\ell\big(T-u,X(T-u)\big)} \, , \, \textnormal{d}\overline{W}^\mathds{P}(T-u) \Bigg\rangle \, , \qquad 0 \leqslant s \leqslant T.
\end{equation}
To show that, in fact, the process $M$ is a true $\mathds{P}$-martingale, we have to rely on condition \hyperref[ffecaoo]{(\ref*{ffecaoo})}, which asserts that the initial relative entropy $H( P(0) \, \vert \, \mathrm{Q} )$ is finite. 

\smallskip

We consider the process 
\begin{equation} \label{ttrreplikrev}
\ell^{-1}\big(T-s,X(T-s)\big) = \frac{q\big(X(T-s)\big)}{p\big(T-s,X(T-s)\big)}  \, , \qquad 0 \leqslant s \leqslant T,
\end{equation} 
where $\ell^{-1}(t, \, \cdot \,) = \frac{1}{\ell(t, \, \cdot \,)}$ is the likelihood ratio function of $\frac{\textnormal{d}\mathrm{Q}}{\textnormal{d}P(t)}( \, \cdot \,)$. Applying It\^{o}'s formula and using \hyperref[qubmcfrsdeftpmpaei]{(\ref*{qubmcfrsdeftpmpaei})}, we find the stochastic differential
\begin{equation} \label{qubmcfrsinvtpmpaei}
\textnormal{d} \ell^{-1}\big(T-s,X(T-s)\big)
= - \Bigg\langle \frac{\nabla \ell\big(T-s,X(T-s)\big)}{\ell\big(T-s,X(T-s)\big)^{2}} \, , \, \textnormal{d}\overline{W}^\mathds{P}(T-s) \Bigg\rangle,
\end{equation}
revealing that the locally bounded process \hyperref[ttrreplikrev]{(\ref*{ttrreplikrev})} is a local martingale under $\mathds{P}$. In fact, this result does not come as a surprise: it is a consequence of the eye-opening result of Fontbona and Jourdain \cite{FJ16} mentioned above, at least when $\mathrm{Q}$ is a finite measure. We refer to Subsection 4.2 of \cite{KST20} for more information on this theme, and for a more direct proof of \hyperref[thetsix]{Theorem \ref*{thetsix}} in the case when $\mathrm{Q}$ is a finite measure on $\mathds{R}^{n}$. 

\smallskip

From \hyperref[qubmcfrsinvtpmpaei]{(\ref*{qubmcfrsinvtpmpaei})}, we deduce the stochastic differential of the logarithm of the process \hyperref[ttrreplikrev]{(\ref*{ttrreplikrev})} and obtain in accordance with \hyperref[sdotpttrrep]{(\ref*{sdotpttrrep})} its form
\begin{equation} \label{sdotpttrrepinv}
\textnormal{d}\log \ell^{-1} \big(T-s,X(T-s)\big)
= - \Bigg\langle \frac{\nabla \ell\big(T-s,X(T-s)\big)}{\ell\big(T-s,X(T-s)\big)} \, , \, \textnormal{d}\overline{W}^\mathds{P}(T-s) \Bigg\rangle  -  \frac{1}{2} \frac{\big\vert \nabla \ell\big(T-s,X(T-s)\big) \big\vert^{2}}{\ell\big(T-s,X(T-s)\big)^{2}} \, \textnormal{d}s.
\end{equation}
We know that the terminal value $\log \ell^{-1}(0,X(0))$ is $\mathds{P}$-integrable, with 
\begin{equation}
\mathds{E}_{\mathds{P}}\big[ \log \ell^{-1}\big(0,X(0)\big) \big] = - H\big( P(0) \, \vert \, \mathrm{Q} \big) \in (-\infty,\infty).
\end{equation}
On the other hand, the initial value
\begin{equation}
\mathds{E}_{\mathds{P}}\big[ \log \ell^{-1}\big(T,X(T)\big) \big] = - H\big( P(T) \, \vert \, \mathrm{Q} \big) \in [-\infty,\infty)
\end{equation}
cannot take the value $\infty$, as mentioned after the definition \hyperref[doref]{(\ref*{doref})} of relative entropy. Hence we can apply \hyperref[sprobamtr]{Proposition \ref*{sprobamtr}} in \hyperref[apsecamtr]{Appendix \ref*{apsecamtr}} to the local martingale \hyperref[ttrreplikrev]{(\ref*{ttrreplikrev})} (in the reverse direction of time) and the deterministic stopping time $\tau = T$, to conclude that 
\begin{equation} \label{gfdmlmpapitupc}
\mathds{E}_{\mathds{P}}\big[ \log \ell^{-1}\big(0,X(0)\big) \big]  - \mathds{E}_{\mathds{P}}\big[ \log \ell^{-1}\big(T,X(T)\big) \big] = - \mathds{E}_{\mathds{P}}\Bigg[\int_{0}^{T} \frac{1}{2} \frac{\big\vert \nabla \ell\big(T-u,X(T-u)\big) \big\vert^{2}}{\ell\big(T-u,X(T-u)\big)^{2}} \,  \textnormal{d}u\Bigg],
\end{equation}
where all terms are well-defined and finite. This shows that the local martingale $M$ is bounded in $L^{2}(\mathds{P})$, with
\begin{equation} \label{gfdmlmequpapitupc}
\| M(0)\|_{L^{2}(\mathds{P})}^{2} = H\big( P(0) \, \vert \, \mathrm{Q} \big) - H\big( P(T) \, \vert \, \mathrm{Q} \big) = \tfrac{1}{2} \int_{0}^{T} I\big( P(t) \, \vert \, \mathrm{Q}\big) \, \textnormal{d}t < \infty,
\end{equation}
completing the proof of \hyperref[thetsix]{Theorem \ref*{thetsix}}.
\end{proof}


\subsection{The proof of \texorpdfstring{\hyperref[thetthretv]{Theorem \ref*{thetthretv}}}{Theorem 4.2}} 


The first step in the proof of \hyperref[thetthretv]{Theorem \ref*{thetthretv}} is to compute the stochastic differentials of the time-reversed perturbed likelihood ratio process
\begin{equation} \label{trplrp}
\ell^{\beta}\big(T-s,X(T-s)\big) = \frac{p^{\beta}\big(T-s,X(T-s)\big)}{q\big(X(T-s)\big)} \, , \qquad 0 \leqslant s \leqslant T - t_{0},
\end{equation}
and its logarithm. By analogy with \hyperref[ptra]{Proposition \ref*{ptra}}, the following result is well known (see, e.g., Theorems G.2 and G.5 in Appendix G of \cite{KST20}) to hold under suitable regularity conditions, such as \hyperref[sosaojkoia]{Assumptions \ref*{sosaojkoia}}. Recall that $(W^{\beta}(t))_{t \geqslant t_{0}}$ denotes the $\mathds{P}^{\beta}$-Brownian motion (in the forward direction of time) defined in \hyperref[wpsdeids]{(\ref*{wpsdeids})}.

\begin{proposition} Under the \textnormal{\hyperref[sosaojkoia]{Assumptions \ref*{sosaojkoia}}}, we let $t_{0} \geqslant 0$ and $T > t_{0}$. The process 
\begin{equation} \label{dotowpbtmtpbm}
\overline{W}^{\mathds{P}^{\beta}}(T-s) \vcentcolon = W^{\beta}(T-s) - W^{\beta}(T) - \int_{0}^{s} \nabla \log p^{\beta}\big(T-u,X(T-u)\big) \, \textnormal{d}u
\end{equation}
for $0 \leqslant s \leqslant T-t_{0}$, is a Brownian motion of the backwards filtration $(\mathcal{G}(T-s))_{0 \leqslant s \leqslant T-t_{0}}$ under the probability measure $\mathds{P}^{\beta}$. Furthermore, the semimartingale decomposition of the time-reversed canonical coordinate process $(X(T-s))_{0 \leqslant s \leqslant T-t_{0}}$ is given by
\begingroup
\addtolength{\jot}{0.7em}
\begin{align}
\textnormal{d} X(T-s) 
&= \Big( \nabla \log p^{\beta}\big(T-s,X(T-s)\big) + \big(\nabla \Psi + \beta\big)\big(X(T-s)\big) \Big) \, \textnormal{d}s + \textnormal{d}\overline{W}^{\mathds{P}^{\beta}}(T-s) \label{rtdfttrppitop} \\
&= \Big( \nabla \log \ell^{\beta} \big(T-s,X(T-s)\big) - \big(\nabla \Psi - \beta\big)\big(X(T-s)\big) \Big) \, \textnormal{d}s + \textnormal{d}\overline{W}^{\mathds{P}^{\beta}}(T-s), \label{rtdfttrpp}
\end{align}
\endgroup
for $0 \leqslant s \leqslant T - t_{0}$, with respect to the backwards filtration $(\mathcal{G}(T-s))_{0 \leqslant s \leqslant T - t_{0}}$. 
\end{proposition}

Comparing the equation \hyperref[poortdfttrpp]{(\ref*{poortdfttrpp})} with \hyperref[rtdfttrppitop]{(\ref*{rtdfttrppitop})}, we see that the reverse-time Brownian motions $\overline{W}^{\mathds{P}^{\beta}}$ and $\overline{W}^{\mathds{P}}$ are related as follows.

\begin{lemma} \label{cotttrbmpapb} Under the \textnormal{\hyperref[sosaojkoia]{Assumptions \ref*{sosaojkoia}}}, we let $t_{0} \geqslant 0$ and $T > t_{0}$. For $0 \leqslant s \leqslant T-t_{0}$, we have
\begingroup
\addtolength{\jot}{0.7em}
\begin{align}
\textnormal{d}\big( \overline{W}^{\mathds{P}} - \overline{W}^{\mathds{P}^{\beta}}\big)(T-s) &=  \Bigg( \beta\big(X(T-s)\big) 
 + \nabla \log \Bigg( \frac{p^{\beta}\big(T-s,X(T-s)\big)}{p\big(T-s,X(T-s)\big)} \Bigg) \Bigg) \, \textnormal{d}s \label{dbowptmtowpbtmtfp} \\
 &=  \Bigg( \beta\big(X(T-s)\big) 
 + \nabla \log \Bigg( \frac{\ell^{\beta}\big(T-s,X(T-s)\big)}{\ell\big(T-s,X(T-s)\big)} \Bigg) \Bigg) \, \textnormal{d}s. \label{dbowptmtowpbtmt}
\end{align}
\endgroup
\end{lemma}

\begin{remark} We shall apply \hyperref[cotttrbmpapb]{Lemma \ref*{cotttrbmpapb}} down the road, when $s$ is close to $T-t_{0}$. In this case the logarithmic gradients in \hyperref[dbowptmtowpbtmtfp]{(\ref*{dbowptmtowpbtmtfp})} and \hyperref[dbowptmtowpbtmt]{(\ref*{dbowptmtowpbtmt})} will become small in view of $p^{\beta}(t_{0},\, \cdot \,) = p(t_{0},\, \cdot \,)$, so that these logarithmic gradients will disappear in the limit $s \uparrow T-t_{0}$; see also \hyperref[hctclwittmeitpocasot]{Lemma \ref*{hctclwittmeitpocasot}} below. By contrast, the term $\beta(X(T-s))$ will not go away in the limit $s \uparrow T-t_{0}$. Rather, it will tend to the random variable $\beta(X(t_{0}))$, which plays an important role in distinguishing between \hyperref[thetsixcorfes]{(\ref*{thetsixcorfes})} and \hyperref[thetsixcorfessl]{(\ref*{thetsixcorfessl})} in \hyperref[thetthretvcor]{Proposition \ref*{thetthretvcor}}.
\end{remark} 

\smallskip

By analogy with the proof of \hyperref[ppropp]{Proposition \ref*{ppropp}}, for $t > t_{0}$, we write now the perturbed Fokker--Planck equation \hyperref[pfpeq]{(\ref*{pfpeq})} as 
\begin{equation} \label{pfpeef}
\partial_{t} p^{\beta}(t,x) = \tfrac{1}{2} \Delta p^{\beta}(t,x) + \big\langle \nabla p^{\beta}(t,x) \, , \nabla \Psi(x) + \beta(x) \big\rangle + p^{\beta}(t,x) \, \big( \Delta \Psi(x) + \operatorname{div} \beta(x) \big).
\end{equation}
Using the relation
\begin{equation} \label{rppbaeb}
p^{\beta}(t,x) = \ell^{\beta}(t,x) \, q(x) = \ell^{\beta}(t,x) \, \mathrm{e}^{ - 2 \Psi(x)}, \qquad t \geqslant t_{0},
\end{equation}
determined computation shows that the perturbed likelihood ratio function $\ell^{\beta}(t,x)$ satisfies 
\begingroup
\addtolength{\jot}{0.7em}
\begin{equation} \label{pfpeeffe}
\begin{aligned}
\partial_{t} \ell^{\beta}(t,x) = \tfrac{1}{2} \Delta \ell^{\beta}(t,x) &+ \big\langle \nabla \ell^{\beta}(t,x) \, , \, \beta(x) - \nabla \Psi(x) \big\rangle \\
&+ \ell^{\beta}(t,x) \, \Big( \operatorname{div} \beta(x) -   \big\langle \beta(x) \, , \, 2 \, \nabla \Psi(x) \big\rangle \Big), \qquad t > t_{0};
\end{aligned}
\end{equation}
\endgroup
this is the analogue of the backwards Kolmogorov equation \hyperref[fpdefflrf]{(\ref*{fpdefflrf})} in this ``perturbed'' context, and reduces to \hyperref[fpdefflrf]{(\ref*{fpdefflrf})} when $\beta \equiv 0$.

\smallskip

With these preparations, we obtain the following stochastic differentials for our objects of interest.

\begin{lemma} \label{trplrpdald} Under the \textnormal{\hyperref[sosaojkoia]{Assumptions \ref*{sosaojkoia}}}, we let $t_{0} \geqslant 0$ and $T > t_{0}$. The time-reversed perturbed likelihood ratio process \textnormal{\hyperref[trplrp]{(\ref*{trplrp})}} and its logarithm satisfy the stochastic differential equations
\begingroup
\addtolength{\jot}{0.7em}
\begin{equation} \label{sdefttrplrp}
\begin{aligned}
&\frac{\textnormal{d} \ell^{\beta}\big(T-s,X(T-s)\big)}{\ell^{\beta}\big(T-s,X(T-s)\big)} 
= \Big( \big\langle \beta \, , \, 2 \, \nabla \Psi \big\rangle - \operatorname{div} \beta \Big)\big(X(T-s)\big) \, \textnormal{d}s \\
& \qquad \quad + \frac{\big\vert \nabla \ell^{\beta}\big(T-s,X(T-s)\big) \big\vert^{2}}{\ell^{\beta}\big(T-s,X(T-s)\big)^{2}} \, \textnormal{d}s \, + \, \Bigg\langle \frac{\nabla \ell^{\beta}\big(T-s,X(T-s)\big)}{\ell^{\beta}\big(T-s,X(T-s)\big)} \, , \, \textnormal{d}\overline{W}^{\mathds{P}^{\beta}}(T-s)\Bigg\rangle
\end{aligned}
\end{equation}
\endgroup
and
\begingroup
\addtolength{\jot}{0.7em}
\begin{equation} \label{sdefttrplrpail}
\begin{aligned}
&\textnormal{d} \log \ell^{\beta}\big(T-s,X(T-s)\big)
= \Big( \big\langle \beta \, , \, 2 \, \nabla \Psi \big\rangle  - \operatorname{div} \beta \Big)\big(X(T-s)\big) \, \textnormal{d}s  \\
&  \qquad \quad + \frac{1}{2} \frac{\big\vert \nabla \ell^{\beta}\big(T-s,X(T-s)\big) \big\vert^{2}}{\ell^{\beta}\big(T-s,X(T-s)\big)^{2}} \, \textnormal{d}s \, + \, \Bigg\langle \frac{\nabla \ell^{\beta}\big(T-s,X(T-s)\big)}{\ell^{\beta}\big(T-s,X(T-s)\big)} \, , \, \textnormal{d}\overline{W}^{\mathds{P}^{\beta}}(T-s) \Bigg\rangle,
\end{aligned}
\end{equation}
\endgroup
respectively, for $0 \leqslant s \leqslant T - t_{0}$, with respect to the backwards filtration $(\mathcal{G}(T-s))_{0 \leqslant s \leqslant T - t_{0}}$.
\begin{proof} The equations \hyperref[sdefttrplrp]{(\ref*{sdefttrplrp})}, \hyperref[sdefttrplrpail]{(\ref*{sdefttrplrpail})} follow from It\^{o}'s formula together with \hyperref[rtdfttrpp]{(\ref*{rtdfttrpp})}, \hyperref[pfpeeffe]{(\ref*{pfpeeffe})}.
\end{proof}
\end{lemma}

We have assembled now all the ingredients needed for the proof of \hyperref[thetthretv]{Theorem \ref*{thetthretv}}. 

\begin{proof}[\bfseries \upshape Proof of \texorpdfstring{\hyperref[thetthretv]{Theorem \ref*{thetthretv}}}] Formally, the stochastic differential in \hyperref[sdefttrplrpail]{(\ref*{sdefttrplrpail})} amounts to the conclusions \hyperref[fispdaftrps]{(\ref*{fispdaftrps})} -- \hyperref[erdollafipimps]{(\ref*{erdollafipimps})} of \hyperref[thetthretv]{Theorem \ref*{thetthretv}}. But as in the proof of \hyperref[thetsix]{Theorem \ref*{thetsix}}, we still have to substantiate the claim, that the stochastic process $M^{\beta}$ defined in \hyperref[dollafipimps]{(\ref*{dollafipimps})} with representation \hyperref[erdollafipimps]{(\ref*{erdollafipimps})} is indeed a $\mathds{P}^{\beta}$-martingale of the backwards filtration $(\mathcal{G}(T-s))_{0 \leqslant s \leqslant T - t_{0}}$, and is bounded in $L^{2}(\mathds{P}^{\beta})$. 

\smallskip

By \hyperref[sdefttrplrpail]{(\ref*{sdefttrplrpail})} and the same stopping argument as in the proof of \hyperref[thetsix]{Theorem \ref*{thetsix}}, the process $M^{\beta}$ is a local $\mathds{P}^{\beta}$-martingale. We have to show that $\mathds{E}_{\mathds{P}^{\beta}}\big[F^{\beta}(t_{0})\big]< \infty$.

\smallskip

We recall that $\beta = \nabla B$ and define the density
\begin{equation}
q^{\beta}(x) \vcentcolon = \mathrm{e}^{-2(\Psi + B)(x)}, \qquad x \in \mathds{R}^{n}.
\end{equation}
This density function solves the stationary version of the perturbed Fokker--Planck equation \hyperref[pfpeq]{(\ref*{pfpeq})}. Equivalently, it induces an invariant measure for the stochastic differential equation \hyperref[wpsdeids]{(\ref*{wpsdeids})}. We now consider the ``doubly perturbed'' likelihood ratio function
\begin{equation} \label{ellbb}
\ell_{\beta}^{\beta}(t,x) \vcentcolon = \frac{p^{\beta}(t,x)}{q^{\beta}(x)} \, , \qquad (t,x) \in [t_{0},\infty) \times \mathds{R}^{n}.
\end{equation}

\smallskip

The \hyperref[sosaojkoia]{Assumptions \ref*{sosaojkoia}} are invariant under the passage from the potential $\Psi$ to $\Psi + B$, so we can apply \hyperref[thetsix]{Theorem \ref*{thetsix}} to the potential $\Psi + B$ and obtain that the process (cf.\ \hyperref[fispdaftr]{(\ref*{fispdaftr})})
\begin{equation} \label{ttwow}
F_{\beta}^{\beta}(T-s) \vcentcolon = \int_{0}^{s} \frac{1}{2} \frac{\big\vert \nabla \ell_{\beta}^{\beta}\big(T-u,X(T-u)\big) \big\vert^{2}}{\ell_{\beta}^{\beta}\big(T-u,X(T-u)\big)^{2}} \, \textnormal{d}u \, , \qquad 0 \leqslant s \leqslant T-t_{0}
\end{equation}
satisfies $\mathds{E}_{\mathds{P}^{\beta}}\big[F_{\beta}^{\beta}(t_{0})\big]< \infty$. This latter condition implies also $\mathds{E}_{\mathds{P}^{\beta}}\big[F^{\beta}(t_{0})\big]< \infty$, where the process $F^{\beta}$ is defined in \hyperref[fispdaftrps]{(\ref*{fispdaftrps})}. Indeed, the function $\big\langle \beta \, , \, 2 \, \nabla \Psi \big\rangle - \operatorname{div} \beta$ in \hyperref[fispdaftrps]{(\ref*{fispdaftrps})} is bounded, so that
\begin{equation} \label{tfeabb}
\mathds{E}_{\mathds{P}^{\beta}}\Bigg[\int_{0}^{T-t_{0}} \Big\vert \big\langle \beta \, , \, 2 \, \nabla \Psi \big\rangle - \operatorname{div} \beta \Big\vert \big(X(T-u)\big) \, \textnormal{d}u \Bigg] < \infty.
\end{equation}
As regards the remaining difference between \hyperref[ttwow]{(\ref*{ttwow})} and \hyperref[fispdaftrps]{(\ref*{fispdaftrps})}, note that $\ell^{\beta}(t,x) / \ell_{\beta}^{\beta}(t,x) = \mathrm{e}^{2B(x)}$ and consequently $\nabla \log \ell^{\beta}(t,x) - \nabla \log \ell_{\beta}^{\beta}(t,x) = 2 \, \nabla B(x)$, which again is a bounded function.

\smallskip

In conclusion, we obtain that $\mathds{E}_{\mathds{P}^{\beta}}\big[F^{\beta}(t_{0})\big]< \infty$, finishing the proof of \hyperref[thetthretv]{Theorem \ref*{thetthretv}}.
\end{proof}


\subsection{Some useful lemmas} \label{subsomusefullem}


In this subsection we collect some useful results needed in order to justify the claims \hyperref[ilpdepvhfv]{(\ref*{ilpdepvhfv})}, \hyperref[ilpdepvhfvsv]{(\ref*{ilpdepvhfvsv})} made in the course of the proof of \hyperref[thetsixcoraopv]{Corollary \ref*{thetsixcoraopv}}, and to complete the proof of \hyperref[thetthretvcor]{Proposition \ref*{thetthretvcor}} in \hyperref[nspofthetthretvcor]{Subsection \ref*{nspofthetthretvcor}}.

\medskip

First, let us introduce the ``perturbed-to-unperturbed'' ratio
\begin{equation} \label{witpetounper}
Y^{\beta}(t,x) \vcentcolon = \frac{\ell^{\beta}(t,x)}{\ell(t,x)}  = \frac{p^{\beta}(t,x)}{p(t,x)}, \qquad (t,x) \in [t_{0},\infty) \times \mathds{R}^{n}
\end{equation}
and recall the backwards Kolmogorov-type equations \hyperref[fpdefflrf]{(\ref*{fpdefflrf})}, \hyperref[pfpeeffe]{(\ref*{pfpeeffe})}. These lead to the equation 
\begingroup
\addtolength{\jot}{0.7em}
\begin{equation} \label{pdefoybptupr}
\begin{aligned}
\partial_{t} Y^{\beta}(t,x) = \tfrac{1}{2} \Delta Y^{\beta}(t,x) 
& + \big\langle \nabla Y^{\beta}(t,x) \, , \, \beta(x) + \nabla \log p(t,x) + \nabla \Psi(x) \big\rangle \\
& + Y^{\beta}(t,x) \, \Big( \operatorname{div} \beta(x) + \big\langle \beta(x) \, , \, \nabla \log p(t,x) \big\rangle \Big), \qquad t > t_{0},
\end{aligned}
\end{equation}
\endgroup
with $Y^{\beta}(t_{0}, \, \cdot \,) = 1$, for the ratio in \hyperref[witpetounper]{(\ref*{witpetounper})}. In conjunction with \hyperref[poortdfttrpp]{(\ref*{poortdfttrpp})}, this equation leads by direct calculation to the following backward dynamics.

\begin{lemma} \label{wsublnsibvn} Under the \textnormal{\hyperref[sosaojkoia]{Assumptions \ref*{sosaojkoia}}}, we let $t_{0} \geqslant 0$ and $T > t_{0}$. The time-reversed ratio process $\big(Y^{\beta}(T-s,X(T-s))\big)_{0 \leqslant s \leqslant T-t_{0}}$ and its logarithm satisfy the stochastic differential equations
\begingroup
\addtolength{\jot}{0.7em}
\begin{equation} \label{logybdbyb}
\begin{aligned}
&\frac{\textnormal{d}Y^{\beta}\big(T-s,X(T-s)\big)}{Y^{\beta}\big(T-s,X(T-s)\big)} = \Bigg\langle \frac{\nabla Y^{\beta}\big(T-s,X(T-s)\big)}{Y^{\beta}\big(T-s,X(T-s)\big)} \, , \, \textnormal{d}\overline{W}^{\mathds{P}}(T-s) - \beta\big(X(T-s)\big) \, \textnormal{d}s \Bigg\rangle \\
& \qquad \qquad \qquad - \bigg( \operatorname{div} \beta \big(X(T-s)\big) + \Big\langle \beta\big(X(T-s)\big) \, , \nabla \log p\big(T-s,X(T-s)\big) \Big\rangle  \bigg) \, \textnormal{d}s
\end{aligned}
\end{equation}
\endgroup
and
\begingroup
\addtolength{\jot}{0.7em}
\begin{equation} \label{logybetapro}
\begin{aligned}
&\textnormal{d} \log Y^{\beta}\big(T-s,X(T-s)\big) = \Bigg\langle \frac{\nabla Y^{\beta}\big(T-s,X(T-s)\big)}{Y^{\beta}\big(T-s,X(T-s)\big)} \, , \, \textnormal{d}\overline{W}^{\mathds{P}}(T-s) - \beta\big(X(T-s)\big) \, \textnormal{d}s \Bigg\rangle \\
& \qquad \qquad \qquad - \bigg( \operatorname{div} \beta \big(X(T-s)\big) + \Big\langle \beta\big(X(T-s)\big) \, , \nabla \log p\big(T-s,X(T-s)\big) \Big\rangle  \bigg) \, \textnormal{d}s \\
& \qquad \qquad \qquad \qquad \qquad \qquad  \qquad  \qquad \qquad \qquad - \frac{1}{2} \frac{\big\vert \nabla Y^{\beta}\big(T-s,X(T-s)\big) \big\vert^{2}}{Y^{\beta}\big(T-s,X(T-s)\big)^{2}} \, \textnormal{d}s,
\end{aligned}
\end{equation}
\endgroup
respectively, for $0 \leqslant s \leqslant T - t_{0}$, relative to the backwards filtration $(\mathcal{G}(T-s))_{0 \leqslant s \leqslant T - t_{0}}$.
\end{lemma}

We first establish a preliminary control on $Y^{\beta}(\, \cdot \, , \, \cdot \,)$, which will be refined in \hyperref[hctclwittmeitpocasot]{Lemma \ref*{hctclwittmeitpocasot}} below.

\begin{lemma} \label{hctclwittmeitpocasotpv} Under the \textnormal{\hyperref[sosaojkoia]{Assumptions \ref*{sosaojkoia}}}, we let $t_{0} \geqslant 0$ and $T > t_{0}$. There is a real constant $C > 1$ such that 
\begin{equation} \label{flaarftpottpsfe}
\frac{1}{C} \leqslant Y^{\beta}(t,x) \leqslant C \, , \qquad (t,x) \in [t_{0},T] \times \mathds{R}^{n}.
\end{equation}
\begin{proof} In the forward direction of time, the canonical coordinate process $(X(t))_{t_{0} \leqslant t \leqslant T}$ on path space $\Omega = \mathcal{C}([t_{0},T];\mathds{R}^{n})$ satisfies the equations \hyperref[sdeids]{(\ref*{sdeids})} and \hyperref[wpsdeids]{(\ref*{wpsdeids})} with initial distribution $P(t_{0})$ under the probability measures $\mathds{P}$ and $\mathds{P}^{\beta}$, respectively. Hence, the $\mathds{P}$-Brownian motion $(W(t))_{t_{0} \leqslant t \leqslant T}$ from \hyperref[sdeids]{(\ref*{sdeids})} can be represented as
\begin{equation}
W(t) - W(t_{0}) = W^{\beta}(t) - W^{\beta}(t_{0}) - \int_{t_{0}}^{t} \beta\big(X(u)\big) \, \textnormal{d}u, \qquad t_{0} \leqslant t \leqslant T,
\end{equation}
where $(W^{\beta}(t))_{t_{0} \leqslant t \leqslant T}$ is the $\mathds{P}^{\beta}$-Brownian motion appearing in \hyperref[wpsdeids]{(\ref*{wpsdeids})}. By the Girsanov theorem, this amounts, for $t_{0} \leqslant t \leqslant T$, to the likelihood ratio computation
\begin{equation} \label{tgtae}
Z(t) \vcentcolon = \frac{\textnormal{d}\mathds{P}^{\beta}}{\textnormal{d}\mathds{P}} \bigg \vert_{\mathcal{F}(t)} 
= \exp \Bigg( - \int_{t_{0}}^{t} \Big\langle \beta\big(X(u)\big) \, , \, \textnormal{d}W(u) \Big\rangle - \tfrac{1}{2} \int_{t_{0}}^{t} \big\vert \beta\big(X(u)\big) \big\vert^{2} \, \textnormal{d}u \Bigg).
\end{equation} 

\smallskip

Now, for each $(t,x) \in [t_{0},T] \times \mathds{R}^{n}$, the ratio $Y^{\beta}(t,x) = p^{\beta}(t,x) / p(t,x)$ equals the conditional expectation of the random variable \hyperref[tgtae]{(\ref*{tgtae})} with respect to the probability measure $\mathds{P}$, where we condition on $X(t) = x$; to wit,
\begin{equation}
Y^{\beta}(t,x) = \mathds{E}_{\mathds{P}}\big[ Z(t) \, \vert \, X(t) = x\big] \, , \qquad (t,x) \in [t_{0},T] \times \mathds{R}^{n}.
\end{equation}
Therefore, in order to obtain the estimate \hyperref[flaarftpottpsfe]{(\ref*{flaarftpottpsfe})}, it suffices to show that the log-density process $( \log Z(t))_{t_{0} \leqslant t \leqslant T}$ is uniformly bounded. Since the perturbation $\beta$ is smooth and has compact support, the Lebesgue integral inside the exponential of \hyperref[tgtae]{(\ref*{tgtae})} is uniformly bounded, as required. 

\smallskip

In order to handle the stochastic integral with respect to the $\mathds{P}$-Brownian motion $(W(u))_{t_{0} \leqslant u \leqslant t}$ inside the exponential \hyperref[tgtae]{(\ref*{tgtae})}, we invoke the assumption that the vector field $\beta$ equals the gradient of a potential $B \colon \mathds{R}^{n} \rightarrow \mathds{R}$, which is of class $\mathcal{C}^{\infty}(\mathds{R}^{n};\mathds{R})$ and has compact support. According to It\^{o}'s formula and \hyperref[sdeids]{(\ref*{sdeids})}, we can express the stochastic integral appearing in \hyperref[tgtae]{(\ref*{tgtae})} as
\begin{equation} \label{urotsiarl}
\int_{t_{0}}^{t} \Big\langle \beta\big(X(u)\big) \, , \, \textnormal{d}W(u) \Big\rangle = B\big(X(t)) - B\big(X(t_{0})\big) + \int_{t_{0}}^{t} \Big(\big\langle \beta \, , \, \nabla \Psi \big\rangle - \tfrac{1}{2} \operatorname{div} \beta \Big)\big(X(u)\big) \, \textnormal{d}u
\end{equation}
for $t_{0} \leqslant t \leqslant T$. At this stage it becomes obvious that the expression of \hyperref[urotsiarl]{(\ref*{urotsiarl})} is uniformly bounded. This completes the proof of \hyperref[hctclwittmeitpocasotpv]{Lemma \ref*{hctclwittmeitpocasotpv}}.
\end{proof}
\end{lemma}

The following \hyperref[hctclwittmeitpocasot]{Lemma \ref*{hctclwittmeitpocasot}} provides the crucial estimates \hyperref[ilpdepvhfv]{(\ref*{ilpdepvhfv})} and \hyperref[ilpdepvhfvsv]{(\ref*{ilpdepvhfvsv})}, needed in the proofs of \hyperref[thetsixcoraopv]{Corollary \ref*{thetsixcoraopv}} and \hyperref[thetthretvcor]{Proposition \ref*{thetthretvcor}}.

\begin{lemma} \label{hctclwittmeitpocasot} Under the \textnormal{\hyperref[sosaojkoia]{Assumptions \ref*{sosaojkoia}}}, we let $t_{0} \geqslant 0$ and $T > t_{0}$. There is a constant $C > 0$ such that 
\begin{equation} \label{ilpdepvhfvwp}
\big\vert Y^{\beta}(T-s,x) - 1 \big\vert \leqslant C \, (T-t_{0}-s),
\end{equation}
as well as
\begin{equation} \label{ilpdepvhfvsvwp}
\mathds{E}_{\mathds{P}}\Bigg[\int_{s}^{T-t_{0}} \Big\vert \nabla \log Y^{\beta}\big(T-u,X(T-u)\big) \Big\vert^{2} \, \textnormal{d}u  \ \bigg\vert \ X(T-s) = x \Bigg] 
\leqslant C \, (T - t_{0} - s)^{2}, 
\end{equation}
hold for all $0 \leqslant s \leqslant T-t_{0}$ and $x \in \mathds{R}^{n}$. Furthermore, for every $t_{0} > 0$ and $x \in \mathds{R}^{n}$ we have the pointwise limiting assertion
\begin{equation} \label{taflawhplftw}
\lim_{s \uparrow T-t_{0}} \frac{\log Y^{\beta}(T-s,x)}{T-t_{0}-s}
= \operatorname{div} \beta(x) + \Big\langle \beta(x) \, , \, \nabla  \log p(t_{0},x) \Big\rangle,
\end{equation}
where the fraction on the left-hand side of \textnormal{\hyperref[taflawhplftw]{(\ref*{taflawhplftw})}} is uniformly bounded on $[0,T-t_{0}] \times \mathds{R}^{n}$.
\begin{remark} The pointwise limiting assertion \hyperref[taflawhplftw]{(\ref*{taflawhplftw})} is the deterministic analogue of the trajectorial relation \hyperref[titmlwhtcfs]{(\ref*{titmlwhtcfs})} from \hyperref[thetthretvcor]{Proposition \ref*{thetthretvcor}}. In \hyperref[nspofthetthretvcor]{Subsection \ref*{nspofthetthretvcor}} below we will prove that the limiting assertion \hyperref[titmlwhtcfs]{(\ref*{titmlwhtcfs})} holds in $L^{1}$ under both $\mathds{P}$ and $\mathds{P}^{\beta}$, and is valid for all $t_{0} > 0$.
\end{remark}
\begin{proof} As $\log Y^{\beta} = \log \ell^{\beta} - \log \ell$, we obtain from \hyperref[thetsix]{Theorems \ref*{thetsix}}, \hyperref[thetthretv]{\ref*{thetthretv}} and \hyperref[flaarftpottpsfe]{(\ref*{flaarftpottpsfe})} that the martingale part of the process in \hyperref[logybetapro]{(\ref*{logybetapro})} is bounded in $L^{2}(\mathds{P})$, i.e.,
\begin{equation} \label{martpoflogybibilt}
\mathds{E}_{\mathds{P}}\Bigg[\int_{0}^{T - t_{0}} \frac{\big\vert \nabla Y^{\beta}\big(T-u,X(T-u)\big) \big\vert^{2}}{Y^{\beta}\big(T-u,X(T-u)\big)^{2}} \, \textnormal{d}u\Bigg] < \infty.
\end{equation}
Once again using \hyperref[flaarftpottpsfe]{(\ref*{flaarftpottpsfe})}, we compare $\nabla Y^{\beta} / Y^{\beta}$ with $\nabla Y^{\beta}$ to see that \hyperref[martpoflogybibilt]{(\ref*{martpoflogybibilt})} also implies
\begin{equation} \label{martpoflogybibiltnyb}
\mathds{E}_{\mathds{P}}\Bigg[\int_{0}^{T - t_{0}} \Big\vert \nabla Y^{\beta}\big(T-u,X(T-u)\big) \Big\vert^{2} \, \textnormal{d}u\Bigg] < \infty.
\end{equation}
According to \hyperref[logybdbyb]{(\ref*{logybdbyb})}, the time-reversed ratio process $\big(Y^{\beta}(T-s,X(T-s))\big)_{0 \leqslant s \leqslant T-t_{0}}$ satisfies with respect to the backwards filtration $(\mathcal{G}(T-s))_{0 \leqslant s \leqslant T - t_{0}}$ the stochastic differential equation
\begingroup
\addtolength{\jot}{0.7em}
\begin{equation} \label{tpdbdoybe}
\begin{aligned}
&\textnormal{d}Y^{\beta}\big(T-s,X(T-s)\big) = \Big\langle \nabla Y^{\beta}\big(T-s,X(T-s)\big) \, , \, \textnormal{d}\overline{W}^{\mathds{P}}(T-s) - \beta\big(X(T-s)\big) \, \textnormal{d}s \Big\rangle \\
& \ - Y^{\beta}\big(T-s,X(T-s)\big)  \bigg( \operatorname{div} \beta \big(X(t-s)\big) + \Big\langle \beta\big(X(T-s)\big) \, , \nabla \log p\big(T-s,X(T-s)\big) \Big\rangle  \bigg) \, \textnormal{d}s.
\end{aligned}
\end{equation}
\endgroup

\smallskip

In view of \hyperref[martpoflogybibiltnyb]{(\ref*{martpoflogybibiltnyb})}, the martingale part in \hyperref[tpdbdoybe]{(\ref*{tpdbdoybe})} is bounded in $L^{2}(\mathds{P})$. As regards the drift term of this equation, we observe that it vanishes when $X(T-s)$ takes values outside the compact support of the smooth vector field $\beta$. Consequently, the drift term is bounded, i.e., the constant
\begin{equation} 
C_{1} \vcentcolon = \sup_{\substack{t_{0} \leqslant t \leqslant T \\ y \in \mathds{R}^{n}}} \Bigg\vert - Y^{\beta}(t,y)  \Bigg( \operatorname{div} \beta(y) + \Bigg\langle \beta(y) \, , \nabla \log p(t,y) + \frac{\nabla Y^{\beta}(t,y)}{Y^{\beta}(t,y)} \Bigg\rangle \, \Bigg)  \Bigg\vert
\end{equation}
is finite, and the processes
\begin{equation}
Y^{\beta}\big(T-s,X(T-s)\big) + C_{1} \, s \qquad \textnormal{ and } \qquad Y^{\beta}\big(T-s,X(T-s)\big) - C_{1} \, s 
\end{equation}
for $0 \leqslant s \leqslant T-t_{0}$, are a sub- and a supermartingale, respectively. We conclude that
\begin{equation} \label{tjpefftvotp}
\Big\vert \, Y^{\beta}(T-s,x) - \mathds{E}_{\mathds{P}}\Big[ Y^{\beta}\big(t_{0},X(t_{0})\big) \ \big\vert \ X(T-s) = x \Big] \, \Big\vert \leqslant  C_{1} \, (T-t_{0}-s)
\end{equation}
holds for all $0 \leqslant s \leqslant T-t_{0}$ and $x \in \mathds{R}^{n}$. Since $Y^{\beta}(t_{0}, \, \cdot \,) = 1$, this establishes the first estimate 
\begin{equation} \label{ilpdepvhfvwphiwtrt}
\big\vert Y^{\beta}(T-s,x) - 1 \big\vert \leqslant C_{1} \, (T-t_{0}-s).
\end{equation}

\smallskip

Now we turn our attention to the second estimate \hyperref[ilpdepvhfvsvwp]{(\ref*{ilpdepvhfvsvwp})}. We fix $0 \leqslant s \leqslant T-t_{0}$ and $x \in \mathds{R}^{n}$. By means of the stochastic differentials in \hyperref[logybetapro]{(\ref*{logybetapro})} and \hyperref[tpdbdoybe]{(\ref*{tpdbdoybe})}, we find that the expression
\begin{equation} \label{sodacebtbbsewiet}
\tfrac{1}{2} \, \mathds{E}_{\mathds{P}}\Bigg[\int_{s}^{T-t_{0}} \Big\vert \nabla \log Y^{\beta}\big(T-u,X(T-u)\big) \Big\vert^{2} \, \textnormal{d}u  \ \bigg\vert \ X(T-s) = x \Bigg] 
\end{equation}
is equal to
\begin{equation} \label{sodacebtbbs}
\log Y^{\beta}(T-s,x) - Y^{\beta}(T-s,x) + 1 + \mathds{E}_{\mathds{P}}\Bigg[\int_{s}^{T-t_{0}} G\big(T-u,X(T-u)\big) \, \textnormal{d}u  \ \bigg\vert \ X(T-s) = x \Bigg],
\end{equation}
where we have set
\begin{equation}
G(t,y) \vcentcolon = \big( Y^{\beta}(t,y)-1\big) \Bigg( \operatorname{div} \beta(y) + \Bigg\langle \beta(y) \, , \nabla \log p(t,y) + \frac{\nabla Y^{\beta}(t,y)}{Y^{\beta}(t,y)} \Bigg\rangle \, \Bigg)
\end{equation}
for $t_{0} \leqslant t \leqslant T$ and $y \in \mathds{R}^{n}$. Introducing the finite constant
\begin{equation} 
C_{2} \vcentcolon = \sup_{\substack{t_{0} \leqslant t \leqslant T \\ y \in \mathds{R}^{n}}} \Bigg\vert \operatorname{div} \beta(y) + \Bigg\langle \beta(y) \, , \nabla \log p(t,y) + \frac{\nabla Y^{\beta}(t,y)}{Y^{\beta}(t,y)} \Bigg\rangle \, \Bigg\vert
\end{equation}
and using the just proved estimate \hyperref[ilpdepvhfvwphiwtrt]{(\ref*{ilpdepvhfvwphiwtrt})}, we see that the absolute value of the conditional expectation appearing in \hyperref[sodacebtbbs]{(\ref*{sodacebtbbs})} can be bounded by $C_{1} \, C_{2} \, (T-t_{0}-s)^{2}$. In order to handle the remaining terms of \hyperref[sodacebtbbs]{(\ref*{sodacebtbbs})}, we apply the elementary inequality $\log p \leqslant p-1$, which is valid for all $p > 0$, and obtain
\begin{equation} \label{teiuftfa}
\log Y^{\beta}(T-s,x) - Y^{\beta}(T-s,x) + 1 \leqslant 0.
\end{equation}
This implies that the expression of \hyperref[sodacebtbbsewiet]{(\ref*{sodacebtbbsewiet})} is bounded by $C_{1} \, C_{2} \, (T-t_{0}-s)^{2}$, which establishes the second estimate \hyperref[ilpdepvhfvsvwp]{(\ref*{ilpdepvhfvsvwp})}. We also note that the elementary inequality \hyperref[teiuftfa]{(\ref*{teiuftfa})} in conjunction with the estimate \hyperref[ilpdepvhfvwphiwtrt]{(\ref*{ilpdepvhfvwphiwtrt})} shows that
\begin{equation}
\log Y^{\beta}(T-s,x) \leqslant C_{1} \, (T-t_{0}-s)   
\end{equation}
for all $0 \leqslant s \leqslant T-t_{0}$ and $x \in \mathds{R}^{n}$; this implies that the fraction on the left-hand side of \hyperref[taflawhplftw]{(\ref*{taflawhplftw})} is uniformly bounded on $[0,T-t_{0}] \times \mathds{R}^{n}$. 

\smallskip

Regarding the limiting assertion \hyperref[taflawhplftw]{(\ref*{taflawhplftw})}, we fix $t_{0} > 0$, $x \in \mathds{R}^{n}$ and $0 \leqslant s \leqslant T-t_{0}$, and take conditional expectations with respect to $X(T-s) = x$ in the integral version of the stochastic differential \hyperref[logybetapro]{(\ref*{logybetapro})}. On account of \hyperref[martpoflogybibilt]{(\ref*{martpoflogybibilt})}, the stochastic integral with respect to the $\mathds{P}$-Brownian motion $(\overline{W}^{\mathds{P}}(T-s))_{0 \leqslant s \leqslant T}$ in \hyperref[logybetapro]{(\ref*{logybetapro})} vanishes. Dividing by $T-t_{0}-s$ and passing to the limit as $s \uparrow T-t_{0}$, we can use the estimate \hyperref[ilpdepvhfvsvwp]{(\ref*{ilpdepvhfvsvwp})} to deduce that the expression in the third line of \hyperref[logybetapro]{(\ref*{logybetapro})} vanishes in the limit. After applying the Cauchy--Schwarz inequality, we see that the normalized integral involving the perturbation $\beta$ appearing in the first line of \hyperref[logybetapro]{(\ref*{logybetapro})} can be bounded by
\begin{equation} \label{trtotflitfube}
\frac{1}{T-t_{0}-s} \int_{s}^{T-t_{0}} \Big\vert \nabla \log Y^{\beta}\big(T-u,X(T-u)\big) \Big\vert \cdot \big\vert \beta\big(X(T-u)\big) \big\vert \, \textnormal{d}u.
\end{equation}
By conditions \hyperref[tsaosaojko]{\ref*{tsaosaojko}}, \hyperref[naltsaosaojkos]{\ref*{naltsaosaojkos}} of \hyperref[sosaojkoia]{Assumptions \ref*{sosaojkoia}}, the function $(t,x) \mapsto \nabla \log Y^{\beta}(t,x)$ is continuous on $(0,\infty) \times \mathds{R}^{n}$, thus the expression in \hyperref[trtotflitfube]{(\ref*{trtotflitfube})} is uniformly bounded on the rectangle $[0,T-t_{0}] \times \operatorname{supp} \beta$. As $\log Y^{\beta}(t_{0}, \, \cdot \,) = 0$, it converges $\mathds{P}$-a.s.\ to zero, hence also
\begin{equation} 
\lim_{s \uparrow T-t_{0}} \, \mathds{E}_{\mathds{P}}\Bigg[\frac{1}{T-t_{0}-s} \int_{s}^{T-t_{0}} \Big\vert \nabla \log Y^{\beta}\big(T-u,X(T-u)\big) \Big\vert \cdot \big\vert \beta\big(X(T-u)\big) \big\vert \, \textnormal{d}u  \ \bigg\vert \ X(T-s) = x \Bigg] 
= 0.
\end{equation}
Finally, continuity and uniform boundedness imply that the conditional expectations of the normalized integrals over the second line of \hyperref[logybetapro]{(\ref*{logybetapro})} converge to the right-hand side of \hyperref[taflawhplftw]{(\ref*{taflawhplftw})}, as claimed.
\end{proof}
\end{lemma}


\subsection{Completing the proof of \texorpdfstring{\hyperref[thetthretvcor]{Proposition \ref*{thetthretvcor}}}{Proposition 4.7}} \label{nspofthetthretvcor}


With the preparations of \hyperref[subsomusefullem]{Subsection \ref*{subsomusefullem}}, we are now able to complete the proof of \hyperref[thetthretvcor]{Proposition \ref*{thetthretvcor}} by establishing the remaining limiting assertions \hyperref[titmlwhtcfs]{(\ref*{titmlwhtcfs})} and \hyperref[thetsixcorfessl]{(\ref*{thetsixcorfessl})} therein. 

\begin{proof}[Proof of \texorpdfstring{\textnormal{\hyperref[titmlwhtcfs]{(\ref*{titmlwhtcfs})}}}{} in \texorpdfstring{\hyperref[thetthretvcor]{Proposition \ref*{thetthretvcor}}}{}:] Let $t_{0} > 0$ and select $T > t_{0}$. Using the notation of \hyperref[witpetounper]{(\ref*{witpetounper})} above, we have to calculate the limit 
\begin{equation} \label{tfwhtccftbsav}
\lim_{s \uparrow T-t_{0}} \, \frac{\log Y^{\beta}\big(T-s,X(T-s)\big)}{T-t_{0}-s}.
\end{equation}
Fix $0 \leqslant s \leqslant T-t_{0}$. According to the integral version of the stochastic differential \hyperref[logybetapro]{(\ref*{logybetapro})}, the fraction in \hyperref[tfwhtccftbsav]{(\ref*{tfwhtccftbsav})} is equal to the sum of the following four normalized integral terms \hyperref[tfwhtccftbsavfi]{(\ref*{tfwhtccftbsavfi})} -- \hyperref[tfwhtccftbsavfith]{(\ref*{tfwhtccftbsavfith})} and \hyperref[tfwhtccftbsavfifo]{(\ref*{tfwhtccftbsavfifo})}, whose behavior as $s \uparrow T-t_{0}$ we will study separately below. By conditions \hyperref[tsaosaojko]{\ref*{tsaosaojko}}, \hyperref[naltsaosaojkos]{\ref*{naltsaosaojkos}} of \hyperref[sosaojkoia]{Assumptions \ref*{sosaojkoia}}, the function $(t,x) \mapsto \nabla \log Y^{\beta}(t,x)$ is continuous on $(0,\infty) \times \mathds{R}^{n}$, thus the first expression 
\begin{equation} \label{tfwhtccftbsavfi}
\frac{1}{T-t_{0}-s} \int_{s}^{T-t_{0}} \bigg( \operatorname{div} \beta \big(X(T-u)\big) + \Big\langle \beta\big(X(T-u)\big) \, , \nabla \log p\big(T-u,X(T-u)\big) \Big\rangle  \bigg) \, \textnormal{d}u
\end{equation}
is uniformly bounded on $[0,T-t_{0}] \times \operatorname{supp} \beta$. Using continuity and uniform boundedness, we conclude that \hyperref[tfwhtccftbsavfi]{(\ref*{tfwhtccftbsavfi})} converges $\mathds{P}$-a.s.\ as well as in $L^{1}(\mathds{P})$ to the right-hand side of \hyperref[titmlwhtcfs]{(\ref*{titmlwhtcfs})}, as required. Thus it remains to show that the three remaining terms converge to zero. Using continuity and uniform boundedness once again, we deduce from $\log Y^{\beta}(t_{0}, \, \cdot \,) = 0$ that the second integral term 
\begin{equation} \label{tfwhtccftbsavfis}
\frac{1}{T-t_{0}-s} \int_{s}^{T-t_{0}} \Bigg\langle \frac{\nabla Y^{\beta}\big(T-u,X(T-u)\big)}{Y^{\beta}\big(T-u,X(T-u)\big)} \, , \, \beta\big(X(T-u)\big)  \Bigg\rangle \, \textnormal{d}u
\end{equation}
converges to zero $\mathds{P}$-a.s.\ and in $L^{1}(\mathds{P})$. Since $\log Y^{\beta}(t_{0}, \, \cdot \,) = 0$ and because the integrand is continuous, we see that the third expression
\begin{equation} \label{tfwhtccftbsavfith}
\frac{1}{T-t_{0}-s} \int_{s}^{T-t_{0}} \frac{1}{2} \frac{\big\vert \nabla Y^{\beta}\big(T-u,X(T-u)\big) \big\vert^{2}}{Y^{\beta}\big(T-u,X(T-u)\big)^{2}} \, \textnormal{d}u
\end{equation}
converges $\mathds{P}$-a.s.\ to zero. Furthermore, owing to \hyperref[hctclwittmeitpocasot]{Lemma \ref*{hctclwittmeitpocasot}}, there is a constant $C > 0$ such that 
\begin{equation} \label{ubfiswwwu}
\mathds{E}_{\mathds{P}}\Bigg[\frac{1}{T - t_{0} - s}\int_{s}^{T-t_{0}} \frac{\big\vert \nabla Y^{\beta}\big(T-u,X(T-u)\big) \big\vert^{2}}{Y^{\beta}\big(T-u,X(T-u)\big)^{2}} \, \textnormal{d}u \Bigg] 
\leqslant C \, (T - t_{0} - s)
\end{equation}
holds for all $0 \leqslant s \leqslant T-t_{0}$, which implies that \hyperref[tfwhtccftbsavfith]{(\ref*{tfwhtccftbsavfith})} converges to zero also in $L^{1}(\mathds{P})$. 

\smallskip

The fourth and last term is the stochastic integral
\begin{equation} \label{tfwhtccftbsavfifo}
- \frac{1}{T-t_{0}-s} \int_{s}^{T-t_{0}} \Bigg\langle \frac{\nabla Y^{\beta}\big(T-u,X(T-u)\big)}{Y^{\beta}\big(T-u,X(T-u)\big)} \, , \, \textnormal{d}\overline{W}^{\mathds{P}}(T-u) \Bigg\rangle.
\end{equation}
The expression \hyperref[tfwhtccftbsavfith]{(\ref*{tfwhtccftbsavfith})} converges to zero $\mathds{P}$-a.s.\ and according to \hyperref[ubfiswwwu]{(\ref*{ubfiswwwu})} we have
\begin{equation} 
\mathds{E}_{\mathds{P}}\Bigg[\frac{1}{(T - t_{0} - s)^{2}}\int_{s}^{T-t_{0}} \frac{\big\vert \nabla Y^{\beta}\big(T-u,X(T-u)\big) \big\vert^{2}}{Y^{\beta}\big(T-u,X(T-u)\big)^{2}} \, \textnormal{d}u \Bigg] \leqslant C.
\end{equation}
By means of the It\^{o} isometry, we deduce that
\begin{equation} 
\lim_{s \uparrow T-t_{0}} \, \mathds{E}_{\mathds{P}}\Bigg[ \, \Bigg( \frac{1}{T-t_{0}-s} \int_{s}^{T-t_{0}} \Bigg\langle \frac{\nabla Y^{\beta}\big(T-u,X(T-u)\big)}{Y^{\beta}\big(T-u,X(T-u)\big)} \, , \, \textnormal{d}\overline{W}^{\mathds{P}}(T-u) \Bigg\rangle  \, \Bigg)^{2} \, \Bigg] = 0.
\end{equation}
In other words, the normalized stochastic integral of \hyperref[tfwhtccftbsavfifo]{(\ref*{tfwhtccftbsavfifo})} converges to zero in $L^{2}(\mathds{P})$.

\smallskip

Summing up, we have shown that the limiting assertion \hyperref[titmlwhtcfs]{(\ref*{titmlwhtcfs})} holds in $L^{1}(\mathds{P})$ for every $t_{0} > 0$. As we have seen in \hyperref[hctclwittmeitpocasotpv]{Lemma \ref*{hctclwittmeitpocasotpv}}, the probability measures $\mathds{P}$ and $\mathds{P}^{\beta}$ are equivalent, the Radon--Nikod\'{y}m derivatives $\frac{\textnormal{d}\mathds{P}^{\beta}}{\textnormal{d}\mathds{P}}$ and $\frac{\textnormal{d}\mathds{P}}{\textnormal{d}\mathds{P}^{\beta}}$ are bounded on the $\sigma$-algebra $\mathcal{F}(T) = \mathcal{G}(0)$, and therefore convergence in $L^{1}(\mathds{P})$ is equivalent to convergence in $L^{1}(\mathds{P}^{\beta})$. This completes the proof of \hyperref[titmlwhtcfs]{(\ref*{titmlwhtcfs})}.
\end{proof}

\begin{proof}[Proof of \texorpdfstring{\textnormal{\hyperref[thetsixcorfessl]{(\ref*{thetsixcorfessl})}}}{} in \texorpdfstring{\hyperref[thetthretvcor]{Proposition \ref*{thetthretvcor}}}{}:] This is proved in very much the same way, as \hyperref[thetsixcorfes]{(\ref*{thetsixcorfes})}, \hyperref[titmlwhtcfs]{(\ref*{titmlwhtcfs})}. The only novelty here is the use of \hyperref[dbowptmtowpbtmtfp]{(\ref*{dbowptmtowpbtmtfp})} to pass to the $\mathds{P}$-Brownian motion $\overline{W}^{\mathds{P}}$ from the $\mathds{P}^{\beta}$-Brownian motion $\overline{W}^{\mathds{P}^{\beta}}$, and the reliance on $\mathds{E}_{\mathds{P}^{\beta}}\big[F^{\beta}(t_{0})\big]< \infty$ to ensure that the resulting stochastic integral is a (square-integrable) $\mathds{P}$-martingale. We leave the details to the diligent reader.
\end{proof}


\section{The rate of growth for the Wasserstein distance} \label{stwt} 


Let us recapitulate the message of \hyperref[thetsixcorao]{Corollaries \ref*{thetsixcorao}} and \hyperref[thetsixcoraopv]{\ref*{thetsixcoraopv}}: in these results we compare the rate of decay for the relative entropy with the rate of growth for the quadratic Wasserstein distance $W_{2}$ along the curves $(P(t))_{t \geqslant 0}$ and $(P^{\beta}(t))_{t \geqslant t_{0}}$ in $\mathscr{P}_{2}(\mathds{R}^{n})$. This is the essence of the gradient flow property formalized in \hyperref[thetthre]{Theorem \ref*{thetthre}}.

\smallskip

In order to complete the proofs of \hyperref[thetsixcorao]{Corollaries \ref*{thetsixcorao}} and \hyperref[thetsixcoraopv]{\ref*{thetsixcoraopv}}, we have to establish the limits \hyperref[agswtuffasih]{(\ref*{agswtuffasih})} and \hyperref[svpvompvv]{(\ref*{svpvompvv})}. The limit \hyperref[agswtuffasih]{(\ref*{agswtuffasih})} is well known (see \cite{AGS08}) to exist, under suitable regularity assumptions, for \textit{Lebesgue-a.e.\ $t_{0} \geqslant 0$}. A similar remark pertains to the ``perturbed'' limit \hyperref[svpvompvv]{(\ref*{svpvompvv})}: if we replace $t_{0}$ by $s_{0}$ in \hyperref[svpvompvv]{(\ref*{svpvompvv})}, it is well known that this limit exists for \textit{Lebesgue-a.e.\ $s_{0} \geqslant t_{0}$}. But this is not what we need. We have to prove the validity of \hyperref[svpvompvv]{(\ref*{svpvompvv})} for \textit{the point $t_{0}$ itself}, in order to calculate the slope of the function $(H(P^{\beta}(t) \, \vert \, \mathrm{Q}))_{t \geqslant t_{0}}$ with respect to the Wasserstein distance \textit{at time $t_{0}$}. After all, the deviation of $P^{\beta}(t)$ from $P(t)$ takes place at time $t_{0}$.

\smallskip

This technical aspect turns out to be quite delicate. We already needed a careful analysis (recall the estimates \hyperref[ilpdepvhfv]{(\ref*{ilpdepvhfv})}, \hyperref[ilpdepvhfvsv]{(\ref*{ilpdepvhfvsv})}) to show that the exceptional set $N$ of \hyperref[rgtsflffflnv]{(\ref*{rgtsflffflnv})}, defined in terms of the decay of entropy of the unperturbed curve $(P(t))_{t \geqslant 0}$, does not change when passing to the perturbed curve $(P^{\beta}(t))_{t \geqslant t_{0}}$. In addition, we have to show that this set $N$ also cannot increase when passing from the unperturbed Wasserstein limit \hyperref[agswtuffasih]{(\ref*{agswtuffasih})} to its perturbed counterpart \hyperref[svpvompvv]{(\ref*{svpvompvv})}. In order to do this, we have to rely here (and only here) on condition \hyperref[nalwstasas]{\ref*{nalwstasas}} of \hyperref[sosaojkoianoew]{Assumptions \ref*{sosaojkoianoew}}.

\smallskip

For a detailed discussion of metric measure spaces and in particular Wasserstein spaces, we refer to \cite{AG13,AGS08} and \cite{Stu06a,Stu06b}. We also refer to Section 5 in \cite{KST20}, where some results on quadratic Wasserstein transport are reviewed for the convenience of the reader.

\medskip

For fixed $T \in (0,\infty)$, we define now the time-dependent velocity field 
\begin{equation} \label{tdvfvtx}
[0,T] \times \mathds{R}^{n} \ni (t,x) \longmapsto v(t,x) \vcentcolon = - \bigg( \frac{1}{2} \frac{\nabla p(t,x)}{p(t,x)} + \nabla \Psi(x) \bigg) = - \frac{1}{2} \frac{\nabla \ell(t,x)}{\ell(t,x)} \in \mathds{R}^{n}.
\end{equation}
According to condition \hyperref[nalwstasas]{\ref*{nalwstasas}} in \hyperref[sosaojkoianoew]{Assumptions \ref*{sosaojkoianoew}}, this gradient vector field $v(t, \, \cdot \, )$ is an element of the tangent space (see Definition 8.4.1 in \cite{AGS08}) of $\mathscr{P}_{2}(\mathds{R}^{n})$ at the point $P(t) \in \mathscr{P}_{2}(\mathds{R}^{n})$, i.e.,
\begin{equation} \label{tanpcvecup}
v(t, \, \cdot \, ) \in \textnormal{Tan}_{P(t)} \mathscr{P}_{2}(\mathds{R}^{n}) \vcentcolon = \overline{\big\{ \nabla \varphi \colon \ \varphi \in \mathcal{C}_{c}^{\infty}(\mathds{R}^{n};\mathds{R})\big\}}^{L^{2}(P(t))}.
\end{equation}
We can now formulate the ``unperturbed'' version of our desired result.

\begin{theorem}[\textsf{Limiting behavior of the quadratic Wasserstein distance}] \label{agswt} Under the \textnormal{\hyperref[sosaojkoianoew]{Assumptions \ref*{sosaojkoianoew}}}, let $t_{0} \geqslant 0$ be such that the generalized de Bruijn identity \textnormal{\hyperref[flffflnv]{(\ref*{flffflnv})}}, \textnormal{\hyperref[flfffl]{(\ref*{flfffl})}} is valid. Then we have the two-sided limit
\begin{equation} \label{agswtuff}
\lim_{t \rightarrow t_{0}} \, \frac{W_{2}\big(P(t),P(t_{0})\big)}{\vert t - t_{0} \vert} 
= \bigg( \mathds{E}_{\mathds{P}}\Big[ \, \big\vert v\big(t_{0},X(t_{0})\big) \big\vert^{2} \, \Big] \bigg)^{1/2} 
= \tfrac{1}{2} \, \sqrt{I\big( P(t_{0}) \, \vert \, \mathrm{Q}\big)}.
\end{equation}
\end{theorem}

\bigskip

Before dealing with \hyperref[agswt]{Theorem \ref*{agswt}}, we will prove the more general \hyperref[bvagswt]{Theorem \ref*{bvagswt}} below which amounts to the perturbed version of \hyperref[agswt]{Theorem \ref*{agswt}}. For right-derivatives, the latter then simply follows by setting $\beta \equiv 0$ in the statement of \hyperref[bvagswt]{Theorem \ref*{bvagswt}}.

\smallskip

We consider the ``perturbed'' curve $(P^{\beta}(t))_{t \geqslant t_{0}}$ in $\mathscr{P}_{2}(\mathds{R}^{n})$, as defined in \hyperref[pfpeq]{(\ref*{pfpeq})} -- \hyperref[wpsdeids]{(\ref*{wpsdeids})}, and define the time-dependent perturbed velocity field 
\begin{equation} \label{pbtdvfvb}
[t_{0},T] \times \mathds{R}^{n} \ni (t,x) \longmapsto v^{\beta}(t,x) \vcentcolon = - \bigg( \frac{1}{2} \frac{\nabla p^{\beta}(t,x)}{p^{\beta}(t,x)} + \nabla \Psi(x) + \beta(x) \bigg) \in \mathds{R}^{n}.
\end{equation} 
At this point, we recall that the perturbation $\beta \colon \mathds{R}^{n} \rightarrow \mathds{R}^{n}$ is a \textit{gradient vector field}, i.e., of the form $\beta = \nabla B$ for some smooth potential $B \colon \mathds{R}^{n} \rightarrow \mathds{R}$ with compact support. Since $p(t_{0}, \, \cdot \,) = p^{\beta}(t_{0}, \, \cdot \,)$, at time $t_{0}$ the vector fields of \hyperref[tdvfvtx]{(\ref*{tdvfvtx})} and \hyperref[pbtdvfvb]{(\ref*{pbtdvfvb})} are related via
\begin{equation} \label{rbvbavwitob}
v^{\beta}(t_{0},x) = v(t_{0},x) - \nabla B(x) = - \nabla \Big( \tfrac{1}{2} \log \ell(t_{0},x) + B(x) \Big), \qquad x \in \mathds{R}^{n}.
\end{equation}
Using the regularity assumption that the potential $B$ is of class $\mathcal{C}_{c}^{\infty}(\mathds{R}^{n};\mathds{R})$, we conclude from \hyperref[tanpcvecup]{(\ref*{tanpcvecup})} and \hyperref[rbvbavwitob]{(\ref*{rbvbavwitob})} that the perturbed vector field $v^{\beta}(t_{0}, \, \cdot \,)$ is also an element of the tangent space of $\mathscr{P}_{2}(\mathds{R}^{n})$ at the point $P^{\beta}(t_{0}) = P(t_{0}) \in \mathscr{P}_{2}(\mathds{R}^{n})$, i.e.,
\begin{equation} \label{perttaspaver}
v^{\beta}(t_{0}, \, \cdot \,) \in \textnormal{Tan}_{P^{\beta}(t_{0})} \mathscr{P}_{2}(\mathds{R}^{n}) 
= \overline{\big\{ \nabla \varphi^{\beta} \colon \ \varphi^{\beta} \in \mathcal{C}_{c}^{\infty}(\mathds{R}^{n};\mathds{R})\big\}}^{L^{2}(P^{\beta}(t_{0}))}.
\end{equation}

\begin{theorem}[\textsf{Limiting behavior of the quadratic Wasserstein distance under perturbations}] \label{bvagswt} Under the \textnormal{\hyperref[sosaojkoianoew]{Assumptions \ref*{sosaojkoianoew}}}, for every point $t_{0} \in \mathds{R}_{+} \setminus N$ \textnormal{(}at which the right-sided limiting identity \textnormal{\hyperref[rgtsflffflnv]{(\ref*{rgtsflffflnv})}} is valid\textnormal{)}, we have the one-sided limit
\begin{equation} \label{pvosagswtuff}
\lim_{t \downarrow t_{0}} \, \frac{W_{2}\big( P^{\beta}(t),P^{\beta}(t_{0})\big)}{t-t_{0}} 
= \bigg( \mathds{E}_{\mathds{P}}\Big[ \, \big\vert v^{\beta}\big(t_{0},X(t_{0})\big) \big\vert^{2} \, \Big] \bigg)^{1/2} 
= \tfrac{1}{2}  \, \| a + 2 b\|_{L^{2}(\mathds{P})}.
\end{equation}
Here $a = \nabla \log \ell(t_{0},X(t_{0}))$ and $b = \beta(X(t_{0}))$ as in \textnormal{\hyperref[ttrvzo]{(\ref*{ttrvzo})}}.
\begin{proof}[Proof of \texorpdfstring{\hyperref[bvagswt]{Theorem \ref*{bvagswt}}}{}] The second equality in \hyperref[pvosagswtuff]{(\ref*{pvosagswtuff})} is apparent from the definition of the time-dependent perturbed velocity field $(v^{\beta}(t, \, \cdot \,))_{t \geqslant t_{0}}$ from \hyperref[pbtdvfvb]{(\ref*{pbtdvfvb})} above. The delicate point is to show that the limiting assertion \hyperref[pvosagswtuff]{(\ref*{pvosagswtuff})} is valid for every $t_{0} \in \mathds{R}_{+} \setminus N$. 
\smallskip

In order to see this, let us fix some $t_{0} \in \mathds{R}_{+} \setminus N$ so that the limiting identity \hyperref[rgtsflffflnv]{(\ref*{rgtsflffflnv})} is valid. In the following steps we prove that then the limiting assertion \hyperref[pvosagswtuff]{(\ref*{pvosagswtuff})} also holds.

\medskip

\noindent \fbox{\textsf{Step 1.}} The gradient vector field $v^{\beta}(t_{0}, \, \cdot \,)$ induces a family of \textit{linearized transport maps} 
\begin{equation} \label{dofaijkophi}
\mathcal{X}_{t}^{\beta}(x) \vcentcolon = x + (t-t_{0}) \cdot v^{\beta}(t_{0},x), \qquad x \in \mathds{R}^{n}
\end{equation}
for $t \geqslant t_{0}$ in the manner of \hyperref[hlotlse]{(\ref*{hlotlse})}, and we denote by $P_{\mathcal{X}}^{\beta}(t)$ the image measure of $P^{\beta}(t_{0}) = P(t_{0})$ under the transport map $\mathcal{X}_{t}^{\beta} \colon \mathds{R}^{n} \rightarrow \mathds{R}^{n}$; i.e.,
\begin{equation} 
P_{\mathcal{X}}^{\beta}(t) \vcentcolon = (\mathcal{X}_{t}^{\beta})_{\#} P^{\beta}(t_{0}), \qquad t \geqslant t_{0}.
\end{equation}
To motivate the arguments that follow, let us first pretend that, for all $t > t_{0}$ sufficiently close to $t_{0}$, the map $\mathcal{X}_{t}^{\beta}$ is the \textit{optimal quadratic Wasserstein transport} from $P^{\beta}(t_{0})$ to $P_{\mathcal{X}}^{\beta}(t)$; i.e.,
\begin{equation}
W_{2}^{2}\big(P_{\mathcal{X}}^{\beta}(t),P^{\beta}(t_{0})\big) = \mathds{E}_{\mathds{P}^{\beta}}\Big[ \, \big\vert \mathcal{X}_{t}^{\beta}\big(X(t_{0})\big) - X(t_{0}) \big\vert^{2} \, \Big] = \mathds{E}_{\mathds{P}}\Big[ \, \big\vert \mathcal{X}_{t}^{\beta}\big(X(t_{0})\big) - X(t_{0}) \big\vert^{2} \, \Big],
\end{equation}
where we have used in the last equality the fact that $X(t_{0})$ has the same distribution under $\mathds{P}^{\beta}$ as it does under $\mathds{P}$. Then, on account of \hyperref[dofaijkophi]{(\ref*{dofaijkophi})}, we could conclude that
\begin{equation} \label{svrlinfifat}
\lim_{t \downarrow t_{0}} \,\frac{W_{2}\big(P_{\mathcal{X}}^{\beta}(t),P^{\beta}(t_{0})\big)}{t-t_{0}} = \bigg( \mathds{E}_{\mathds{P}}\Big[ \, \big\vert v^{\beta}\big(t_{0},X(t_{0})\big) \big\vert^{2} \, \Big] \bigg)^{1/2}
= \tfrac{1}{2}  \, \| a + 2 b\|_{L^{2}(\mathds{P})}.
\end{equation} 
Furthermore, let us suppose that we can show the limiting identity
\begin{equation} \label{svrlinfifatsvwct}
\lim_{t \downarrow t_{0}} \, \frac{W_{2}\big(P^{\beta}(t),P_{\mathcal{X}}^{\beta}(t)\big)}{t-t_{0}} = 0,
\end{equation} 
which has the interpretation that ``the straight line $(P_{\mathcal{X}}^{\beta}(t))_{t \geqslant t_{0}}$ is \textit{tangential} to the curve $(P^{\beta}(t))_{t \geqslant t_{0}}$''. Using \hyperref[svrlinfifat]{(\ref*{svrlinfifat})} and \hyperref[svrlinfifatsvwct]{(\ref*{svrlinfifatsvwct})}, we could now derive the desired equality \hyperref[pvosagswtuff]{(\ref*{pvosagswtuff})}. Indeed, invoking the triangle inequality for the quadratic Wasserstein distance we obtain 
\begin{equation} \label{tiotwwsdfv}
\lim_{t \downarrow t_{0}} \, \frac{W_{2}\big(P_{\mathcal{X}}^{\beta}(t),P^{\beta}(t_{0})\big)}{t-t_{0}} \, \leqslant \, \lim_{t \downarrow t_{0}} \, \frac{W_{2}\big(P_{\mathcal{X}}^{\beta}(t),P^{\beta}(t)\big)}{t-t_{0}} \, + \, \liminf_{t \downarrow t_{0}} \, \frac{W_{2}\big(P^{\beta}(t),P^{\beta}(t_{0})\big)}{t-t_{0}},
\end{equation}
and one more application of the triangle inequality yields
\begin{equation} \label{tiotwwsdfvsv}
\limsup_{t \downarrow t_{0}} \,\frac{W_{2}\big(P^{\beta}(t),P^{\beta}(t_{0})\big)}{t-t_{0}} \, \leqslant \,  \lim_{t \downarrow t_{0}} \, \frac{W_{2}\big(P^{\beta}(t),P_{\mathcal{X}}^{\beta}(t)\big)}{t-t_{0}} \, + \, \lim_{t \downarrow t_{0}} \, \frac{W_{2}\big(P_{\mathcal{X}}^{\beta}(t),P^{\beta}(t_{0})\big)}{t-t_{0}}.
\end{equation}

\medskip

\noindent \fbox{\textsf{Step 2.}} The bad news at this point is that there is little reason why, for $t > t_{0}$ sufficiently close to $t_{0}$, the map $\mathcal{X}_{t}^{\beta}$ defined in \hyperref[dofaijkophi]{(\ref*{dofaijkophi})} of \textsf{Step 1} should be optimal with respect to quadratic Wasserstein transportation costs; i.e., by Brenier's theorem \cite{Bre91}, equal to the gradient of a \textit{convex} function. The good news is that we can reduce the general case to the situation of optimal transports $\mathcal{X}_{t}^{\beta}$ as in \textsf{Step 1} by localizing the vector field $v^{\beta}(t_{0}, \, \cdot \,)$ as well as the transport maps $(\mathcal{X}_{t}^{\beta})_{t \geqslant t_{0}}$ to compact subsets of $\mathds{R}^{n}$ (\textsf{Steps 2} -- \textsf{4}); and that, after these localizations have been carried out, an analogue of the equality \hyperref[svrlinfifatsvwct]{(\ref*{svrlinfifatsvwct})} also holds, allowing us to complete the argument (\textsf{Steps 5} -- \textsf{6}).

\smallskip

To this end, we recall that $v^{\beta}(t_{0}, \, \cdot \,)$ from \hyperref[rbvbavwitob]{(\ref*{rbvbavwitob})} is an element of the tangent space $\textnormal{Tan}_{P^{\beta}(t_{0})} \mathscr{P}_{2}(\mathds{R}^{n})$ of the quadratic Wasserstein space $\mathscr{P}_{2}(\mathds{R}^{n})$ at the point $P^{\beta}(t_{0}) \in \mathscr{P}_{2}(\mathds{R}^{n})$. Thus, we can choose a sequence of functions $(\varphi_{m}^{\beta}(t_{0}, \cdot \,))_{m \geqslant 1} \subseteq \mathcal{C}_{c}^{\infty}(\mathds{R}^{n};\mathds{R})$ such that
\begin{equation} \label{aotvfgosfvgosfwcs}
\lim_{m \rightarrow \infty} \, \mathds{E}_{\mathds{P}}\bigg[ \ \Big\vert v^{\beta}\big((t_{0},X(t_{0})\big) - \nabla \varphi_{m}^{\beta}\big(t_{0},X(t_{0})\big) \Big\vert^{2} \ \bigg] = 0.
\end{equation}

\smallskip

Next, for each $m \in \mathds{N}$, we define the \textit{localized gradient vector fields}
\begin{equation} \label{lvfvbrtox}
v_{m}^{\beta}(t_{0},x) \vcentcolon = \nabla \varphi_{m}^{\beta}(t_{0},x), \qquad x \in \mathds{R}^{n}.
\end{equation} 
These have compact support, approximate the gradient vector field $v^{\beta}(t_{0}, \, \cdot \,)$ in $L^{2}(P(t_{0}))$ as in \hyperref[aotvfgosfvgosfwcs]{(\ref*{aotvfgosfvgosfwcs})}, and induce a family of \textit{localized linear transports} $(\mathcal{X}_{t}^{\beta,m})_{t \geqslant t_{0}}$ defined by analogy with \hyperref[dofaijkophi]{(\ref*{dofaijkophi})} via
\begin{equation} \label{defphbrtaopws}
\mathcal{X}_{t}^{\beta,m}(x) \vcentcolon = x + (t-t_{0}) \cdot v_{m}^{\beta}(t_{0},x), \qquad x \in \mathds{R}^{n}.
\end{equation}
We denote by $P_{\mathcal{X}}^{\beta,m}(t)$ the image measure of $P^{\beta}(t_{0}) = P(t_{0})$ under this localized linear transport map $\mathcal{X}_{t}^{\beta,m} \colon \mathds{R}^{n} \rightarrow \mathds{R}^{n}$; i.e.,
\begin{equation} \label{delltmpphbrtf}
P_{\mathcal{X}}^{\beta,m}(t) \vcentcolon = (\mathcal{X}_{t}^{\beta,m})_{\#} P^{\beta}(t_{0}), \qquad t \geqslant t_{0}.
\end{equation} 

\medskip

\noindent \fbox{\textsf{Step 3.}} We claim that, for every $m \in \mathds{N}$, there exists some $\varepsilon_{m} > 0$ such that for all $t \in (t_{0},t_{0} + \varepsilon_{m})$, the localized linear transport map $\mathcal{X}_{t}^{\beta,m} \colon \mathds{R}^{n} \rightarrow \mathds{R}^{n}$ constructed in \textsf{Step 2} defines an optimal Wasserstein transport from $P^{\beta}(t_{0})$ to $P_{\mathcal{X}}^{\beta,m}(t)$. Hence, by Brenier's theorem (\cite{Bre91}, \cite[Theorem 2.12]{Vil03}), we have to show that $\mathcal{X}_{t}^{\beta,m}$ is the \textit{gradient of a convex function}, for all $t > t_{0}$ sufficiently close to $t_{0}$. 

\smallskip

Indeed, from the definitions in \hyperref[lvfvbrtox]{(\ref*{lvfvbrtox})}, \hyperref[defphbrtaopws]{(\ref*{defphbrtaopws})} we see that the functions $\mathcal{X}_{t}^{\beta,m}$ are gradients, for all $m \in \mathds{N}$ and $t \geqslant t_{0}$. More precisely, we have
\begin{equation} \label{bracket}
\mathcal{X}_{t}^{\beta,m}(x) = \nabla \Big( \tfrac{1}{2} \vert x \vert^{2} + (t-t_{0}) \cdot \varphi_{m}^{\beta}(t_{0},x) \Big), \qquad x \in \mathds{R}^{n}.
\end{equation}
As the Hessian matrix of $\varphi_{m}^{\beta}(t_{0}, \, \cdot \,)$ is uniformly bounded, the function in the bracket of \hyperref[bracket]{(\ref*{bracket})} is a convex function of $x$ for every $m \in \mathds{N}$ and $t \in (t_{0},t_{0} + \varepsilon_{m})$, for $\varepsilon_{m} > 0$ small enough. We also note for later use that $\mathcal{X}_{t}^{\beta,m}$ defines a Lipschitz bijection on $\mathds{R}^{n}$, again for every $m \in \mathds{N}$ and $t \in (t_{0},t_{0} + \varepsilon_{m})$.

\medskip

\noindent \fbox{\textsf{Step 4.}} From \textsf{Step 3} we know that, for every $m \in \mathds{N}$, there exists some $\varepsilon_{m} > 0$ such that for all $t \in (t_{0},t_{0} + \varepsilon_{m})$ the localized map $\mathcal{X}_{t}^{\beta,m}$ is the optimal transport from $P^{\beta}(t_{0})$ to $P_{\mathcal{X}}^{\beta,m}(t)$ with respect to quadratic Wasserstein costs. Therefore, we can apply the considerations of \textsf{Step 1} to the optimal map $\mathcal{X}_{t}^{\beta,m}$ in \hyperref[defphbrtaopws]{(\ref*{defphbrtaopws})}, and deduce that
\begin{equation} \label{svrlinfifatnv}
\lim_{t \downarrow t_{0}} \,\frac{W_{2}\big(P_{\mathcal{X}}^{\beta,m}(t),P^{\beta}(t_{0})\big)}{t-t_{0}} = \bigg( \mathds{E}_{\mathds{P}}\Big[ \, \big\vert v_{m}^{\beta}\big(t_{0},X(t_{0})\big) \big\vert^{2} \, \Big] \bigg)^{1/2}
\end{equation} 
holds for every $m \in \mathds{N}$. Invoking \hyperref[aotvfgosfvgosfwcs]{(\ref*{aotvfgosfvgosfwcs})} and \hyperref[lvfvbrtox]{(\ref*{lvfvbrtox})}, we obtain from this
\begin{equation} 
\lim_{m \rightarrow \infty} \, \lim_{t \downarrow t_{0}} \,\frac{W_{2}\big(P_{\mathcal{X}}^{\beta,m}(t),P^{\beta}(t_{0})\big)}{t-t_{0}} = \bigg( \mathds{E}_{\mathds{P}}\Big[ \, \big\vert v^{\beta}\big(t_{0},X(t_{0})\big) \big\vert^{2} \, \Big] \bigg)^{1/2} = \tfrac{1}{2}  \, \| a + 2 b\|_{L^{2}(\mathds{P})}.
\end{equation} 
From the inequalities \hyperref[tiotwwsdfv]{(\ref*{tiotwwsdfv})} and \hyperref[tiotwwsdfvsv]{(\ref*{tiotwwsdfvsv})} of \textsf{Step 1} (with $P_{\mathcal{X}}^{\beta,m}(t)$ instead of $P_{\mathcal{X}}^{\beta}(t)$) it follows that, in order to conclude \hyperref[pvosagswtuff]{(\ref*{pvosagswtuff})}, it remains to establish the analogue of the identity \hyperref[svrlinfifatsvwct]{(\ref*{svrlinfifatsvwct})}:
\begin{equation} \label{nvotlisvrlinfifatsvwct}
\lim_{m \rightarrow \infty} \, \lim_{t \downarrow t_{0}} \, \frac{W_{2}\big(P^{\beta}(t),P_{\mathcal{X}}^{\beta,m}(t)\big)}{t-t_{0}} = 0.
\end{equation} 

\medskip

\noindent \fbox{\textsf{Step 5.}} The time-dependent velocity field $(v^{\beta}(t, \, \cdot \,))_{t \geqslant t_{0}}$ induces a \textit{curved flow} $(\mathcal{Y}_{t}^{\beta})_{t \geqslant t_{0}}$, which is characterized by  
\begin{equation} \label{cotcfl}
\tfrac{\textnormal{d}}{\textnormal{d}t} \, \mathcal{Y}_{t}^{\beta} = v^{\beta}(t,\mathcal{Y}_{t}^{\beta}) \quad \textnormal{ for all } t \geqslant t_{0} \, , \qquad \mathcal{Y}_{t_{0}}^{\beta} = \textnormal{Id}_{\mathds{R}^{n}}.
\end{equation}
Then, for every $t \geqslant t_{0}$, the map $\mathcal{Y}_{t}^{\beta} \colon \mathds{R}^{n} \rightarrow \mathds{R}^{n}$ transports the measure $P^{\beta}(t_{0}) = P(t_{0})$ to $P^{\beta}(t)$, i.e., $(\mathcal{Y}_{t}^{\beta})_{\#} P^{\beta}(t_{0}) = P^{\beta}(t)$.

\smallskip

The localized linear mappings $\mathcal{X}_{t}^{\beta,m} \colon \mathds{R}^{n} \rightarrow \mathds{R}^{n}$ of \hyperref[defphbrtaopws]{(\ref*{defphbrtaopws})} transport $P^{\beta}(t_{0})$ to $P_{\mathcal{X}}^{\beta,m}(t)$, as in \hyperref[delltmpphbrtf]{(\ref*{delltmpphbrtf})}. As mentioned at the end of \textsf{Step 3}, the inverse mappings $(\mathcal{X}_{t}^{\beta,m})^{-1} \colon \mathds{R}^{n} \rightarrow \mathds{R}^{n}$ are well-defined for all $m \in \mathds{N}$ and $t \in (t_{0},t_{0} + \varepsilon_{m})$; they satisfy
\begin{equation} 
\big((\mathcal{X}_{t}^{\beta,m})^{-1}\big)_{\#} P_{\mathcal{X}}^{\beta,m}(t) = P^{\beta}(t_{0}), \qquad t \in (t_{0},t_{0} + \varepsilon_{m}).
\end{equation} 

\smallskip

From \textsf{Step 4}, our remaining task is to prove \hyperref[nvotlisvrlinfifatsvwct]{(\ref*{nvotlisvrlinfifatsvwct})}. To this end, we have to construct maps $\mathcal{Z}_{t}^{\beta,m} \colon \mathds{R}^{n} \rightarrow \mathds{R}^{n}$ that transport $P_{\mathcal{X}}^{\beta,m}(t)$ to $P^{\beta}(t)$, i.e., $(\mathcal{Z}_{t}^{\beta,m})_{\#} P_{\mathcal{X}}^{\beta,m}(t) = P^{\beta}(t)$, and satisfy
\begin{equation} \label{twhteaoncmne}
\lim_{m \rightarrow \infty} \, \lim_{t \downarrow t_{0}} \, \frac{1}{t-t_{0}} \, \Bigg( \, \mathds{E}_{\mathds{P}_{\mathcal{X}}^{\beta,m}}\bigg[ \ \Big\vert \mathcal{Z}_{t}^{\beta,m}\big(X(t)\big) - X(t) \Big\vert^{2} \ \bigg]  \, \Bigg)^{1/2} = 0  \, ,
\end{equation}
where $\mathds{P}_{\mathcal{X}}^{\beta,m}$ denotes a probability measure on path space under which the random variable $X(t)$ has distribution $P_{\mathcal{X}}^{\beta,m}(t)$ as in \hyperref[delltmpphbrtf]{(\ref*{delltmpphbrtf})}. We define for this job the candidate maps
\begin{equation} \label{defphitcamap}
\mathcal{Z}_{t}^{\beta,m} \vcentcolon = \mathcal{Y}_{t}^{\beta} \circ \big( \mathcal{X}_{t}^{\beta,m} \big)^{-1}, \qquad t \in (t_{0},t_{0} + \varepsilon_{m});
\end{equation}
recall that $(\mathcal{X}_{t}^{\beta,m})^{-1}$ transports $P_{\mathcal{X}}^{\beta,m}(t)$ to $P^{\beta}(t_{0})$ while $\mathcal{Y}_{t}^{\beta}$ transports $P^{\beta}(t_{0})$ to $P^{\beta}(t)$; and conclude that $\mathcal{Z}_{t}^{\beta,m}$ of \hyperref[defphitcamap]{(\ref*{defphitcamap})} transports $P_{\mathcal{X}}^{\beta,m}(t)$ to $P^{\beta}(t)$. Thus, we obtain 
\begin{equation} \label{olwtwhteaoncmne}
\mathds{E}_{\mathds{P}_{\mathcal{X}}^{\beta,m}}\bigg[ \ \Big\vert \mathcal{Z}_{t}^{\beta,m}\big(X(t)\big) - X(t) \Big\vert^{2} \ \bigg] 
= \mathds{E}_{\mathds{P}}\bigg[ \ \Big\vert \mathcal{Y}_{t}^{\beta}\big(X(t_{0})\big) - \mathcal{X}_{t}^{\beta,m}\big(X(t_{0})\big) \Big\vert^{2} \ \bigg].
\end{equation}
Combining \hyperref[twhteaoncmne]{(\ref*{twhteaoncmne})} with \hyperref[olwtwhteaoncmne]{(\ref*{olwtwhteaoncmne})}, we see that we have to establish
\begin{equation}  \label{olebaewtwhteaoncmne}
\lim_{m \rightarrow \infty} \, \lim_{t \downarrow t_{0}} \, \frac{1}{(t-t_{0})^{2}} \, \mathds{E}_{\mathds{P}}\bigg[ \ \Big\vert \mathcal{Y}_{t}^{\beta}\big(X(t_{0})\big) - \mathcal{X}_{t}^{\beta,m}\big(X(t_{0})\big) \Big\vert^{2} \ \bigg] = 0.
\end{equation}
Using \hyperref[defphbrtaopws]{(\ref*{defphbrtaopws})} and the elementary inequality $\vert x+y \vert^{2} \leqslant 2(\vert x \vert^{2}+\vert y \vert^{2})$, for $x,y \in \mathds{R}^{n}$, we derive the estimate
\begingroup
\addtolength{\jot}{0.7em}
\begin{align}
\tfrac{1}{2} \, \big\vert \mathcal{Y}_{t}^{\beta}(x) - \mathcal{X}_{t}^{\beta,m}(x) \big\vert^{2} \, \leqslant \, &(t-t_{0})^2 \cdot \vert v^{\beta}(t_{0},x) - v_{m}^{\beta}(t_{0},x)\vert^{2} \label{olwtwhteaoncmneb} \\
\, + \, &\Big\vert \big(\mathcal{Y}_{t}^{\beta}(x) - x \big) - (t-t_{0}) \cdot v^{\beta}(t_{0},x) \Big\vert^{2}. \label{olwtwhteaoncmnea} 
\end{align}
\endgroup

\smallskip

Therefore, in order to establish \hyperref[olebaewtwhteaoncmne]{(\ref*{olebaewtwhteaoncmne})}, it suffices to show the limiting assertions \hyperref[liidofstsetwahfo]{(\ref*{liidofstsetwahfo})} and \hyperref[bolwtwhteaoncmne]{(\ref*{bolwtwhteaoncmne})} below; these correspond to \hyperref[olwtwhteaoncmneb]{(\ref*{olwtwhteaoncmneb})} and \hyperref[olwtwhteaoncmnea]{(\ref*{olwtwhteaoncmnea})}, respectively. The first limiting identity we already have from \hyperref[aotvfgosfvgosfwcs]{(\ref*{aotvfgosfvgosfwcs})}, \hyperref[lvfvbrtox]{(\ref*{lvfvbrtox})} of \textsf{Step 2}, namely,
\begin{equation} \label{liidofstsetwahfo}
\lim_{m \rightarrow \infty} \, \mathds{E}_{\mathds{P}}\bigg[ \ \Big\vert v^{\beta}\big((t_{0},X(t_{0})\big) - v_{m}^{\beta}\big(t_{0},X(t_{0})\big) \Big\vert^{2} \ \bigg] = 0.
\end{equation}

\medskip

\noindent \fbox{\textsf{Step 6.}} Our final task is to justify that
\begin{equation} \label{bolwtwhteaoncmne}
\lim_{t \downarrow t_{0}} \, \mathds{E}_{\mathds{P}}\bigg[ \ \Big\vert \tfrac{1}{t-t_{0}} \Big(\mathcal{Y}_{t}^{\beta}\big(X(t_{0})\big) - X(t_{0}) \Big) - v^{\beta}\big(t_{0},X(t_{0})\big) \Big\vert^{2} \ \bigg] = 0.
\end{equation}
To this end, we first note that by \hyperref[cotcfl]{(\ref*{cotcfl})} we have for all $t \geqslant t_{0}$ the identity
\begin{equation} \label{iotcfitdfe}
\mathcal{Y}_{t}^{\beta}(x) = x + \int_{t_{0}}^{t} v^{\beta}\big(u,\mathcal{Y}_{u}^{\beta}(x)\big) \, \textnormal{d}u, \qquad x \in \mathds{R}^{n},
\end{equation}
on whose account the expectation in \hyperref[bolwtwhteaoncmne]{(\ref*{bolwtwhteaoncmne})} is equal to
\begin{equation} 
\mathds{E}_{\mathds{P}}\Bigg[ \ \bigg\vert \frac{1}{t-t_{0}} \int_{t_{0}}^{t} v^{\beta}\Big(u,\mathcal{Y}_{u}^{\beta}\big(X(t_{0})\big)\Big) \, \textnormal{d}u - v^{\beta}\big(t_{0},X(t_{0})\big) \bigg\vert^{2} \ \Bigg].
\end{equation}
As $\mathcal{Y}_{t}^{\beta}$ transports $P^{\beta}(t_{0})=P(t_{0})$ to $P^{\beta}(t)$, and because the random variable $X(t_{0})$ has the same distribution under $\mathds{P}^{\beta}$ as it does under $\mathds{P}$, this expectation can also be expressed with respect to the probability measure $\mathds{P}^{\beta}$, and it thus suffices to show the limiting assertion
\begin{equation} \label{tsctznuiv}
\lim_{t \downarrow t_{0}} \, \mathds{E}_{\mathds{P}^{\beta}}\Bigg[ \ \bigg\vert \frac{1}{t-t_{0}} \int_{t_{0}}^{t} v^{\beta}\big(u,X(u)\big) \, \textnormal{d}u - v^{\beta}\big(t_{0},X(t_{0})\big) \bigg\vert^{2} \ \Bigg] = 0.
\end{equation}
For this purpose, we first observe that by the continuity of the paths of the canonical coordinate process $(X(t))_{t \geqslant 0}$, the family of random variables 
\begin{equation} \label{uivrvone}
\Bigg( \ \bigg\vert \frac{1}{t-t_{0}} \int_{t_{0}}^{t} v^{\beta}\big(u,X(u)\big) \, \textnormal{d}u - v^{\beta}\big(t_{0},X(t_{0})\big) \bigg\vert^{2} \ \Bigg)_{t \geqslant t_{0}}
\end{equation}
converges $\mathds{P}^{\beta}$-a.s.\ to zero, as $t \downarrow t_{0}$. In order to show that their expectations also converge to zero, i.e., that \hyperref[tsctznuiv]{(\ref*{tsctznuiv})} does hold, we have to verify that the family of \hyperref[uivrvone]{(\ref*{uivrvone})} is uniformly integrable with respect to $\mathds{P}^{\beta}$. As the random variable $\vert v^{\beta}(t_{0},X(t_{0})) \vert^{2}$ belongs to $L^{1}(\mathds{P}^{\beta})$, and we have
\begin{equation}
\bigg\vert \frac{1}{t-t_{0}} \int_{t_{0}}^{t} v^{\beta}\big(u,X(u)\big) \, \textnormal{d}u \bigg\vert^{2} \, \leqslant \,
\frac{1}{t-t_{0}} \int_{t_{0}}^{t} \big\vert v^{\beta}\big(u,X(u)\big) \big\vert^{2} \, \textnormal{d}u, \qquad t \geqslant t_{0}
\end{equation}
by Jensen's inequality, it is sufficient to prove the uniform integrability of the family 
\begin{equation} \label{eqouofuindsf}
\Bigg( \, \frac{1}{t-t_{0}} \int_{t_{0}}^{t} \big\vert v^{\beta}\big(u,X(u)\big) \big\vert^{2} \, \textnormal{d}u  \, \Bigg)_{t \geqslant t_{0}}.
\end{equation}
Invoking the definition of the time-dependent velocity field $(v^{\beta}(t, \, \cdot \,))_{t \geqslant t_{0}}$ in \hyperref[pbtdvfvb]{(\ref*{pbtdvfvb})} and the fact that the perturbation $\beta$ is smooth and compactly supported, the uniform integrability of the family in \hyperref[eqouofuindsf]{(\ref*{eqouofuindsf})} above, is equivalent to the uniform integrability of the family
\begin{equation} \label{tfseqsbuitg}
\Bigg( \, \frac{1}{t-t_{0}} \int_{t_{0}}^{t} \frac{\big\vert \nabla \ell^{\beta}\big(u,X(u)\big) \big\vert^{2}}{\ell^{\beta}\big(u,X(u)\big)^{2}} \, \textnormal{d}u  \, \Bigg)_{t \geqslant t_{0}}.
\end{equation}
Now by continuity, the family of \hyperref[tfseqsbuitg]{(\ref*{tfseqsbuitg})} converges $\mathds{P}^{\beta}$-a.s.\ to $\vert \nabla \log \ell(t_{0},X(t_{0}))\vert^{2}$. Thus, to establish this uniform integrability, it suffices to show that the family of random variables in \hyperref[tfseqsbuitg]{(\ref*{tfseqsbuitg})} converges in $L^{1}(\mathds{P}^{\beta})$. Hence, in view of \textit{Scheff\'{e}'s lemma} (\hyperref[whaelsl]{Lemma \ref*{whaelsl}}), it remains to check that the corresponding expectations also converge. But at this point we use for the first time our choice of $t_{0} \in \mathds{R}_{+} \setminus N$ and recall \hyperref[tciwsolafophvs]{(\ref*{tciwsolafophvs})}, \hyperref[tcigwwsolafophvs]{(\ref*{tcigwwsolafophvs})} from the proof of \hyperref[thetsixcoraopv]{Corollary \ref*{thetsixcoraopv}}, which gives us 
\begin{equation}
\lim_{t \downarrow t_{0}}  \, \mathds{E}_{\mathds{P}^{\beta}}\Bigg[ \frac{1}{t-t_{0}} \int_{t_{0}}^{t} \frac{\big\vert \nabla \ell^{\beta}\big(u,X(u)\big) \big\vert^{2}}{\ell^{\beta}\big(u,X(u)\big)^{2}} \, \textnormal{d}u \Bigg] = 
\mathds{E}_{\mathds{P}}\Bigg[ \ \frac{\big\vert \nabla \ell\big(t_{0},X(t_{0})\big) \big\vert^{2}}{\ell\big(t_{0},X(t_{0})\big)^{2}} \ \Bigg],
\end{equation}
as required. This completes the proof of the claim made in the beginning of \textsf{Step 6}.

\smallskip

Summing up, in light of \hyperref[olwtwhteaoncmneb]{(\ref*{olwtwhteaoncmneb})}, \hyperref[olwtwhteaoncmnea]{(\ref*{olwtwhteaoncmnea})} from \textsf{Step 5}, the limiting assertions \hyperref[liidofstsetwahfo]{(\ref*{liidofstsetwahfo})} and \hyperref[bolwtwhteaoncmne]{(\ref*{bolwtwhteaoncmne})} imply the limiting behavior \hyperref[olebaewtwhteaoncmne]{(\ref*{olebaewtwhteaoncmne})}. According to the results of \textsf{Steps 4} and \textsf{5}, the latter also entails the validity of the limiting identity \hyperref[nvotlisvrlinfifatsvwct]{(\ref*{nvotlisvrlinfifatsvwct})}, which completes the proof of \hyperref[bvagswt]{Theorem \ref*{bvagswt}}.
\end{proof}
\end{theorem}

Equipped with \hyperref[bvagswt]{Theorem \ref*{bvagswt}}, we can now easily deduce \hyperref[agswt]{Theorem \ref*{agswt}}.

\begin{proof}[Proof of \texorpdfstring{\hyperref[agswt]{Theorem \ref*{agswt}}}{}] The second equality in \hyperref[agswtuff]{(\ref*{agswtuff})} follows from the representation of the relative Fisher information in \hyperref[merfi]{(\ref*{merfi})} and the definition of the time-dependent velocity field $(v(t, \, \cdot \,))_{t \geqslant t_{0}}$ in \hyperref[tdvfvtx]{(\ref*{tdvfvtx})}. The first equality in \hyperref[agswtuff]{(\ref*{agswtuff})} follows from \hyperref[bvagswt]{Theorem \ref*{bvagswt}} if we set $\beta \equiv 0$. However, the limit in \hyperref[pvosagswtuff]{(\ref*{pvosagswtuff})} is only from the right, while the limit in \hyperref[agswtuff]{(\ref*{agswtuff})} is two-sided. But the only reason for considering right-sided limits in \hyperref[bvagswt]{Theorem \ref*{bvagswt}}, was the presence of the perturbation $\beta$ at time $t \geqslant t_{0}$. If there is no such perturbation, one can replace all limits from the right by two-sided ones. This completes the proof of \hyperref[agswt]{Theorem \ref*{agswt}}.
\end{proof}



\setkomafont{section}{\large}
\setkomafont{subsection}{\normalsize}

\begin{appendices}



\section{Some measure-theoretic results} \label{apsecamtr} 


In the proofs of \hyperref[thetsixcor]{Propositions \ref*{thetsixcor}} and \hyperref[thetthretvcor]{\ref*{thetthretvcor}} we have used a result about conditional expectations, which we formulate below as \hyperref[probamtr]{Proposition \ref*{probamtr}}; we refer to Proposition D.2 in Appendix D of \cite{KST20} for its proof. We place ourselves on a probability space $(\Omega,\mathcal{F},\mathds{P})$ endowed with a left-continuous filtration $(\mathcal{F}(t))_{t \geqslant 0}$. We first recall the following result, which is well known under the name of \textit{Scheff\'{e}'s lemma} \cite[5.10]{Wil91}.

\begin{lemma}[\textsf{Scheff\'{e}'s lemma}] \label{whaelsl} For a sequence of integrable random variables $(X_{n})_{n \in \mathds{N}}$ which converges $\mathds{P}$-a.s.\ to another integrable random variable $X$, the convergence of the $L^{1}(\mathds{P})$-norms \textnormal{(}i.e., $\lim_{n \rightarrow \infty} \mathds{E}[\vert X_{n} \vert] = \mathds{E}[\vert X \vert ]$\textnormal{)} is equivalent to the convergence in $L^{1}(\mathds{P})$ \textnormal{(}i.e., $\lim_{n \rightarrow \infty} \mathds{E}[\vert X_{n} - X\vert] = 0$\textnormal{)}. 
\end{lemma}

\begin{proposition} \label{probamtr} Let $(B(t))_{0 \leqslant t \leqslant T}$ and $(C(t))_{0 \leqslant t \leqslant T}$ be adapted continuous processes, which are non-negative and uniformly bounded, respectively. Define the process 
\begin{equation}
A(t) \vcentcolon = \int_{0}^{t} \big( B(u) + C(u) \big) \, \textnormal{d}u, \qquad 0 \leqslant t \leqslant T
\end{equation}
and assume that $\mathds{E}\big[\int_{0}^{T}  B(u)  \, \textnormal{d}u\big]$ is finite. By the Lebesgue differentiation theorem, for Lebesgue-a.e.\ $t_{0} \in [0,T]$ we have
\begin{equation} \label{titlocotn}
\lim_{t \uparrow t_{0}} \,  \frac{\mathds{E}\big[A(t)-A(t_{0})\big]}{t-t_{0}} = \lim_{t \uparrow t_{0}} \, \frac{1}{t-t_{0}} \, \mathds{E}\Bigg[ \int_{t_{0}}^{t} \big( B(u) + C(u) \big) \, \textnormal{d}u \Bigg]  = \mathds{E}\big[ B(t_{0}) + C(t_{0})\big].
\end{equation}
Now fix a ``Lebesgue point'' $t_{0} \in [0,T]$ for which \textnormal{\hyperref[titlocotn]{(\ref*{titlocotn})}} does hold. Then we have the analogous limiting assertion for the conditional expectations, i.e.,
\begin{equation} \label{probamtre}
\lim_{t \uparrow t_{0}} \, \frac{\mathds{E}\big[A(t_{0})-A(t) \ \vert \ \mathcal{F}(t)\big]}{t_{0}-t} = 
B(t_{0}) + C(t_{0}),
\end{equation}
where the limit exists in $L^{1}(\mathds{P})$.
\end{proposition}

In the proof of \hyperref[thetsix]{Theorem \ref*{thetsix}} we invoked the following result. We refer for its proof to Lemma 2.48 in \cite{KK21}.

\begin{proposition} \label{sprobamtr} Suppose $(N(t))_{t \geqslant 0}$ is a strictly positive local martingale with continuous paths. Let $\tau$ be a $[0,\infty)$-valued stopping time such that $\mathds{E}[\log N(\tau)]$ is finite and $\mathds{E}[(\log N(0))^{+}] < \infty$. Then $\mathds{E}[\log N(0)]$ is finite and
\begin{equation}
\mathds{E}\big[\log N(\tau)\big] - \mathds{E}\big[\log N(0)\big]  = - \tfrac{1}{2} \, \mathds{E}\Big[ \big[\log N, \log N\big](\tau)\Big].
\end{equation}

\end{proposition}


\section{The proof of \texorpdfstring{\hyperref[hlotl]{Lemma \ref*{hlotl}}}{Lemma 4.10}} \label{polhlotl}


\begin{proof}[Proof of \texorpdfstring{\hyperref[hlotl]{Lemma \ref*{hlotl}}}{}] In order to show \hyperref[hlotlte]{(\ref*{hlotlte})}, we recall the notation of \hyperref[hlotlse]{(\ref*{hlotlse})} and consider the time-dependent velocity field 
\begin{equation}
[0,1] \times \mathds{R}^{n} \ni (t,\xi) \longmapsto v_{t}(\xi) \vcentcolon = \gamma\Big( \big(T_{t}^{\gamma}\big)^{-1}(\xi)\Big) \in \mathds{R}^{n},
\end{equation}
which is well-defined $P_{t}$-a.s.\ for every $t \in [0,1]$. Then $(v_{t})_{0 \leqslant t \leqslant 1}$ is the velocity field associated with $(T_{t}^{\gamma})_{0 \leqslant t \leqslant 1}$, i.e.,
\begin{equation}
T_{t}^{\gamma}(x) = x + \int_{0}^{t} v_{\theta}\big(T_{\theta}^{\gamma}(x)\big) \, \textnormal{d}\theta,
\end{equation}
on account of \hyperref[hlotlse]{(\ref*{hlotlse})}. Let $p_{t}(\, \cdot \,)$ be the probability density function of the probability measure $P_{t}$ in \hyperref[hlotlse]{(\ref*{hlotlse})}. Then, according to \cite[Theorem 5.34]{Vil03}, the function $p_{t}(\, \cdot \,)$ satisfies the continuity equation
\begin{equation} 
\partial_{t} p_{t}(x) + \operatorname{div} \big( v_{t}(x) \, p_{t}(x) \big) = 0, \qquad (t,x) \in (0,1) \times \mathds{R}^{n},
\end{equation}
which can be written equivalently as 
\begin{equation} \label{ceiikyjl}
- \partial_{t} p_{t}(x) = \operatorname{div} \big( v_{t}(x) \big) \, p_{t}(x) + \big\langle v_{t}(x) \, , \nabla p_{t}(x) \big\rangle, \qquad (t,x) \in (0,1) \times \mathds{R}^{n}.
\end{equation}
Recall that $X_{0}$ is a random variable with probability distribution $P_{0}$ on the probability space $(S,\mathcal{S},\nu)$. Then the integral equation
\begin{equation} \label{ieahwv}
X_{t} = X_{0} + \int_{0}^{t} v_{\theta}(X_{\theta}) \, \textnormal{d}\theta, \qquad 0 \leqslant t \leqslant 1
\end{equation}
defines random variables $X_{t}$ with probability distributions $P_{t} = (T_{t}^{\gamma})_{\#}(P_{0})$ for $t \in [0,1]$, as in \hyperref[hlotlse]{(\ref*{hlotlse})}. We have 
\begin{equation} 
\textnormal{d} p_{t}(X_{t}) 
= \partial_{t} p_{t}(X_{t}) \, \textnormal{d}t + \big\langle \nabla p_{t}(X_{t}) \, , \, \textnormal{d}X_{t} \big\rangle
= - p_{t}(X_{t}) \operatorname{div} \big( v_{t}(X_{t}) \big) \, \textnormal{d}t
\end{equation}
on account of \hyperref[ceiikyjl]{(\ref*{ceiikyjl})}, \hyperref[ieahwv]{(\ref*{ieahwv})}, thus also
\begin{equation} \label{dlogpd}
\textnormal{d} \log p_{t}(X_{t}) = - \operatorname{div} \big( v_{t}(X_{t}) \big) \, \textnormal{d}t, \qquad 0 \leqslant t \leqslant 1.
\end{equation}
Recall the function $q(x) = \mathrm{e}^{ - 2 \Psi(x)}$, for which
\begin{equation} \label{dlogqd}
\textnormal{d} \log q( X_{t} ) 
= -  \big\langle 2 \, \nabla \Psi(X_{t}) \, , \, \textnormal{d}X_{t} \big\rangle  
= -  \big\langle 2 \, \nabla \Psi(X_{t}) \, , \, v_{t}(X_{t}) \big\rangle \, \textnormal{d}t.
\end{equation}
For the likelihood ratio function $\ell_{t}(\, \cdot \,)$ of \hyperref[nlrfitramfcc]{(\ref*{nlrfitramfcc})} we get from \hyperref[dlogpd]{(\ref*{dlogpd})} and \hyperref[dlogqd]{(\ref*{dlogqd})} that
\begin{equation} \label{hwwtte}
\textnormal{d} \log \ell_{t}(X_{t}) =  \big\langle 2 \, \nabla \Psi(X_{t}) \, , \, v_{t}(X_{t}) \big\rangle \, \textnormal{d}t \, - \, \operatorname{div} \big( v_{t}(X_{t}) \big) \, \textnormal{d}t, \qquad 0 \leqslant t \leqslant 1.
\end{equation}
Taking expectations in the integral version of \hyperref[hwwtte]{(\ref*{hwwtte})}, we obtain that the difference 
\begin{equation}
H(P_{t} \, \vert \, \mathrm{Q})  - H(P_{0} \, \vert \, \mathrm{Q}) 
= \mathds{E}_{\nu}\big[ \log \ell_{t}(X_{t})\big] - \mathds{E}_{\nu}\big[ \log \ell_{0}(X_{0}) \big]
\end{equation}
is equal to
\begin{equation}
\mathds{E}_{\nu}\bigg[\int_{0}^{t} \Big(  \big\langle 2 \, \nabla \Psi(X_{\theta}) \, , \, v_{\theta}(X_{\theta}) \big\rangle - \operatorname{div} \big( v_{\theta}(X_{\theta})\big) \Big) \, \textnormal{d}\theta \bigg]
\end{equation}
for $t \in [0,1]$. Consequently,
\begin{equation} \label{liadn}
\lim_{t \downarrow 0} \frac{H(P_{t} \, \vert \, \mathrm{Q}) - H(P_{0} \, \vert \, \mathrm{Q})}{t} 
= \mathds{E}_{\nu}\Big[    \big\langle 2 \, \nabla \Psi(X_{0}) \, , \, v_{0}(X_{0}) \big\rangle - \operatorname{div} \big(v_{0}(X_{0})\big) \Big].
\end{equation}
Integrating by parts, we see that 
\begingroup
\addtolength{\jot}{0.7em}
\begin{align}
\mathds{E}_{\nu}\big[\operatorname{div} \big(v_{0}(X_{0})\big)\big] 
&= \int_{\mathds{R}^{n}} \operatorname{div} \big(v_{0}(x)\big) \, p_{0}(x) \, \textnormal{d}x = - \int_{\mathds{R}^{n}} \big\langle v_{0}(x) \, , \nabla p_{0}(x) \big\rangle \, \textnormal{d}x  \label{ibppo} \\
&= - \big\langle \nabla \log p_{0}(X_{0}) \, , \, v_{0}(X_{0}) \big\rangle_{L^{2}(\nu)}. \label{ibppopt}
\end{align}
\endgroup
Recalling \hyperref[liadn]{(\ref*{liadn})}, and combining it with the relation $\nabla \log \ell_{t}(x) = \nabla \log p_{t}(x) + 2 \, \nabla \Psi(x)$, as well as with \hyperref[ibppo]{(\ref*{ibppo})} and \hyperref[ibppopt]{(\ref*{ibppopt})}, we get
\begin{equation} 
\lim_{t \downarrow 0} \frac{H(P_{t} \, \vert \, \mathrm{Q}) - H(P_{0} \, \vert \, \mathrm{Q})}{t} 
= \big\langle \nabla \log \ell_{0}(X_{0}) \, , \, v_{0}(X_{0}) \big\rangle_{L^{2}(\nu)}.
\end{equation}
Since $v_{0} = \gamma$, this leads to \hyperref[hlotlte]{(\ref*{hlotlte})}. 
\end{proof}

\end{appendices}






\bibliographystyle{alpha}
{\footnotesize
\bibliography{references}}


\end{document}